\documentclass[11pt,letterpaper,reqno]{amsart}
\usepackage[left=2.5cm,right=2.5cm,top=3cm,bottom=3cm]{geometry}
\usepackage[utf8]{inputenc}
\usepackage[english]{babel}
\usepackage{amsmath}
\usepackage{amssymb}
\usepackage{amsfonts}
\usepackage{amsthm}
\usepackage{mathrsfs}
\usepackage{graphicx}
\usepackage{enumitem}
\setlist[itemize]{leftmargin=*}
\usepackage{xcolor}
\usepackage{url}
\usepackage{bm}
\usepackage[all]{xy}
\usepackage{hyperref}
\usepackage{esint}
\usepackage{todonotes}
\usepackage{soul}
\usepackage{stmaryrd}
\hypersetup{
	colorlinks,
	linkcolor={blue!60!black},
	citecolor={green!60!black},
	urlcolor={green!60!black}
}
\usepackage{doi}
\usepackage[capitalise,nameinlink]{cleveref}
\usepackage[stretch=50]{microtype}
\usepackage[maxnames=50,doi=false,isbn=false,url=false,eprint=false,backend=bibtex,giveninits=true,sorting=nyt]{biblatex}

\DeclareFieldFormat[inbook]{chapter}{#1}
\renewbibmacro{in:}{}
\DeclareFieldFormat[article]{volume}{vol\adddot\addnbspace #1}
\DeclareFieldFormat[article]{number}{no\adddot\addnbspace #1}
\renewbibmacro{volume+number+eid}{%
	\setunit{\addcomma\space}%
	\printfield{volume}%
	\setunit{\addcomma\space}%
	\printfield{number}%
	\setunit{\addcomma\space}%
	\printfield{eid}%
}
\AtEveryBibitem{
  \clearlist{language}
}
        
        \crefformat{equation}{#2(#1)#3}
        \crefrangeformat{equation}{#3(#1)#4 to~#5(#2)#6}
        \crefmultiformat{equation}{#2(#1)#3}{ and~#2(#1)#3}{, #2(#1)#3}{ and~#2(#1)#3}
        \usepackage{mathtools}
        \usepackage{etoolbox}
        \usepackage{chngcntr}

        \newcommand{\N}{\mathbb{N}}
        \newcommand{\Z}{\mathbb{Z}}
        
        \newcommand{\R}{\mathbb{R}}
        \newcommand{\C}{\mathbb{C}}
        \renewcommand{\S}{\mathbb{S}}

        \newcommand{\de}{\partial}

        \renewcommand{\bar}{\overline}
        \DeclareMathOperator{\diam}{diam}

        \DeclareMathOperator{\spt}{spt}
        
        \newcommand{\vol}{\mathrm{vol}}

        \renewcommand{\epsilon}{\varepsilon}

        \newcommand{\sff}{\mathrm{I\!I}}
        
        \theoremstyle{definition}
        \newtheorem{definition}{Definition}
        \newtheorem{rmk}[definition]{Remark}
        \newtheorem*{definition*}{Definition}
        \newtheorem*{rmk*}{Remark}
        \newtheorem*{ack*}{Acknowledgement}
        \newtheorem*{acks*}{Acknowledgements}
        \theoremstyle{plain}
        \newtheorem{thm}[definition]{Theorem}
        \newtheorem{lemma}[definition]{Lemma}
        \newtheorem{corollary}[definition]{Corollary}
        \newtheorem{proposition}[definition]{Proposition}

        \newtheorem{question}[definition]{Question}

        \newtheorem*{thm*}{Theorem}
        
        \newtheorem*{lemma*}{Lemma}
        \newtheorem*{corollary*}{Corollary}
        \newtheorem*{proposition*}{Proposition}
        \newtheorem*{claim*}{Claim}
        \newtheorem*{conj*}{Conjecture}
        
        \numberwithin{equation}{section}
        \numberwithin{definition}{section}

        \renewcommand{\textbf}[1]{\bm{\mathrm{#1}}}
        \renewcommand{\div}{\operatorname{div}}
        
        \newcommand{\jump}[1]{\left[\!\left[ #1 \right]\!\right]}
\makeatletter
\renewcommand{\tocsection}[3]{%
  \indentlabel{\@ifnotempty{#2}{\bfseries\ignorespaces#1 #2\quad}}\bfseries#3}
\renewcommand{\tocsubsection}[3]{%
  \indentlabel{\@ifnotempty{#2}{\ignorespaces#1 #2\quad}}#3}

\def\@tocline#1#2#3#4#5#6#7{\relax
  \ifnum #1>\c@tocdepth 
  \else
    \par \addpenalty\@secpenalty\addvspace{#2}%
    \begingroup \hyphenpenalty\@M
    \@ifempty{#4}{%
      \@tempdima\csname r@tocindent\number#1\endcsname\relax
    }{%
      \@tempdima#4\relax
    }%
    \parindent\z@ \leftskip#3\relax \advance\leftskip\@tempdima\relax
    \rightskip\@pnumwidth plus1em \parfillskip-\@pnumwidth
    #5\leavevmode\hskip-\@tempdima{#6}\nobreak
    \leaders\hbox{$\m@th\mkern \@dotsep mu\hbox{.}\mkern \@dotsep mu$}\hfill
    \nobreak
    \hbox to\@pnumwidth{\@tocpagenum{\ifnum#1=1\bfseries\fi#7}}\par
    \nobreak
    \endgroup
  \fi}
\AtBeginDocument{%
\expandafter\renewcommand\csname r@tocindent0\endcsname{0pt}
}

\def\l@subsection{\@tocline{2}{0pt}{2.5pc}{5pc}{}}
\makeatother
\begin{document}
\title[CMC surfaces in 3-manifolds]{Existence of constant mean curvature surfaces with controlled topology in 3-manifolds}
\author{Filippo Gaia \and Xuanyu Li}
\address{Department of Mathematics,
Building 380, Stanford, CA 94305, USA}
\email{fgaia@stanford.edu}
\address{Department of Mathematics, Cornell University, Ithaca, NY 14853, USA}
\email{xl896@cornell.edu}
        \maketitle
\begin{abstract}
    We establish the existence of a non-trivial, branched immersion of a closed Riemann surface $\Sigma$ with constant mean curvature (CMC) $H$ into any closed, orientable 3-manifold $\mathcal{M}$, for almost every prescribed value of $H$. The genus of the surface $\Sigma$ is bounded from above by the Heegaard genus $h$ of $\mathcal{M}$. 
    
    Starting from a family of sweep-outs of $\mathcal{M}$ by surfaces of genus $h$, we apply a min-max construction for a family $\{E_{H,\sigma_k}\}_{k\in \mathbb{N}}$ of perturbations of the energy involving the second fundamental form of the immersions to produce almost-critical points $u_k$ of $E_{H,\sigma_k}$. We then show, following ideas introduced by Rivière and developed by Pigati and Rivière, that the maps $u_k$ converge to a ``CMC-parametrized varifold". This limiting object is then shown to be a smooth, branched immersion with the prescribed mean curvature $H$.
\end{abstract}
\section{Introduction}
\subsection{Main results}

In this paper, we establish the following result regarding existence of constant mean curvature (CMC) surfaces.

\begin{thm}\label{Prop: existence of qualifying sequence}
    Let $(\mathcal{M}^3,g)$ be a smooth, closed, oriented Riemannian manifold with Heegaard genus $h$.
    For almost every $H>0$ and for $H=0$, there exists a closed Riemann surface $\Sigma$ with genus $g(\Sigma)\leq h$ and a non-trivial branched $H$-constant-mean-curvature immersion $u:\Sigma\rightarrow\mathcal{M}$.
\end{thm}

Here we say that $u$ is a branched $H$-CMC immersion if, in conformal coordinates, it satisfies
\begin{align}\label{eq: equation of CMC}
    \Delta u+A_u(\nabla u,\nabla u)=H u_x\times u_y,
\end{align}
where we assume $\mathcal{M}$ is isometrically embedded in some Euclidean space $\R^Q$ and $A_{u(x)}$ denotes the second fundamental form of $\mathcal{M}\subset \R^Q$ at $u(x)$, and $\times$ is the cross product in $T_{u(x)}\mathcal{M}$.
The image of such a map has constant mean curvature equal to $H$. $H$-CMC immersions are critical points to the following functional
\begin{align*}
    E_H(v)=\operatorname{Area}(v)+H\vol(f_v),
\end{align*}
where $\vol(f_v)$ roughly stands for the volume bounded by map $v$; see Remark \ref{rmk: definition of volume} and Section \ref{Sec: Preliminaires} below.

A wide range of methods has been developed to prove the existence of constant mean curvature surfaces; see, for instance, \cite{Heinz,Hilderbrandt,Brezis,StruweLargeSurfaces,StruweNonuniqueness,Struwe-Freeboundary,Kapouleas,Ye,DuzaarSteffen,MahmoudiMazzeoPacard,RosenbergSmith,BreinerKapouleas,Zhou-Zhucmc,ZhouZhupmc,BelletiniWickramasekera1,BelletiniWickramasekera2,BelletiniWickramasekera3}. We refer the reader to the survey by X. Zhou \cite{ZhouSurvey} for a comprehensive overview of the field. 
Despite this progress, the existence of $H$-CMC branched immersions with controlled genus in general three-manifolds has remained a challenging open problem. 
A significant advance in this direction was made by D. R. Cheng and X. Zhou \cite{Cheng-Zhou}, who proved that in any closed Riemannian three‑manifold which is topologically a three-sphere (or more generally has non-trivial third homotopy group) and for almost every $H>0$, there exists a nontrivial $H$-CMC branched immersed two‑sphere. Their work, however, leaves open the case of general three-manifolds. Theorem \ref{Prop: existence of qualifying sequence} fills this gap, and provides an alternative proof of the result of D. R. Cheng and X. Zhou. We also treat the case $H=0$$H=0$, even though sharper results are available in the literature in this case. The existence of minimal two-spheres in non-aspherical manifolds was obtained by J. Sacks and K. Uhlenbeck \cite{SacksUhlenbeck-2sphere}, while the existence of embedded minimal surfaces with genus control in three-manifolds was established by L. Simon and F. Smith \cite{SimonSmith}, see also \cite{ColdingDelellis} and \cite{DeLellisPellandini}.

The result of D. R. Cheng and X. Zhou is based on a penalization of the Dirichlet energy of the form $\varepsilon\int_{\mathbb{S}^2}\lvert\Delta u\rvert^2$. Using a min–max construction over sweep-outs of the target manifold $\mathcal{M}$ by two-spheres, they produce critical points of the penalized energy. Passing to the limit as $\varepsilon\to 0$, they obtain $H$-harmonic maps, that is, solutions of \eqref{eq: equation of CMC} which are not a priori required to be conformal. Since the two-sphere admits a unique conformal structure up to diffeomorphism, these maps are in fact $H$-CMC branched immersions.

For a general closed, orientable three-manifold $\mathcal M$, a sweep-out by two-spheres need not exist. By Heegaard theory, $\mathcal M$ nonetheless admits a sweep-out by a closed surface $\Sigma$, where $\Sigma$ has genus equal to the Heegaard genus of $\mathcal M$. If $\Sigma$ has positive genus, however, applying the methods of \cite{Cheng-Zhou} to this setting would in general yield only $H$-harmonic maps, which need not be conformal.

To address this limitation, in the present work
we use a different perturbation of the area functional, often referred to as the ``viscosity method'': for 
$\sigma>0$, we consider
\begin{align*}
    A_{\sigma}(u)=\operatorname{Area}(u)+\sigma^4\int\vert\sff^u\vert^4\,d\vol_{g_u}
\end{align*}
for $W^{2,4}$ immersions $u:\Sigma\to\mathcal{M}$,
where $\Sigma$ is an arbitrary Riemann surface and $\sff^u$ denotes the second fundamental form of $u(\Sigma)$ in $\mathcal{M}$.
This approach was originally introduced by T. Rivière in \cite{Riviere-IHES} to construct minimal immersions. The regularity part of the method, a stronger convergence, and the optimal control of the limiting multiplicity were then established by A. Pigati and T. Rivière in \cite{PR-Regularity, PR-Multiplicity-One}. A. Pigati then provided substantial improvements of the method in \cite{Pigati-FB}, while implementing it to the free boundary case.
Two key properties of these functionals are their invariance under reparameterization of $u$, and the fact that—as observed by J. Langer \cite{LangerCompactness} and P. Breuning \cite{Breuning}—for each fixed $\sigma>0$, the conformal structures on $\Sigma$ induced by maps lying in a sublevel set of $A_{\sigma}$ lie in a compact subset of the moduli space of $\Sigma$.

A. Pigati and T. Rivière showed that, given a sequence of critical (or almost critical) points $\lbrace u_k\rbrace_{k\in\N}\subset W^{2,4}(\Sigma,\mathcal{M})$ for the functionals $A_{\sigma_{k}}$ (where $\mathcal{M}$ has dimension $n\geq 2$), with $\sigma_k\rightarrow0$, and satisfying a natural entropy condition (Assumption (2) in Theorem \ref{thm: main-theorem}), the associated varifolds converge to a \textit{parametrized stationary varifold}. That is, the limit varifold admits a conformal parametrization 
for which the stationarity condition holds locally (in the domain) and is induced by a smooth branched minimal immersion from a Riemann surface $\Sigma'$ to $\mathcal{M}$, with $g(\Sigma')\leq g(\Sigma)$.

To prove Theorem~\ref{Prop: existence of qualifying sequence}, we follow similar ideas and, in the process, obtain a CMC analogue of the result of Pigati--Rivière, which may be of independent interest.\\
Given a closed, connected 3-manifold $\mathcal{M}$, smoothly embedded in $\R^Q$, there exists a family of sweep-outs $f:[0,1]\times\Sigma\rightarrow\mathcal{M}$, where $\Sigma$ is a Riemann surface of genus $h$, equal to the Heegaard genus of $\mathcal{M}$, and $f(t,\cdot)$ is an immersion for any $t\in (0,1)$. Applying a min–max procedure to such a family of sweep-outs and to the perturbed $H$-CMC functional $E_{H,\sigma}$ defined below, for almost every $H>0$ (and $H=0$) we show that there exist a sequence $\sigma_k\rightarrow0$ and a corresponding family of $W^{2,4}$ immersions $u_k$ such that, for every $k\in\N$,
\begin{enumerate}
        \item $u_k$ is $\theta_k$-critical for the functional
        \begin{align*}
            E_{H,\sigma_k}(v,f_v)=\operatorname{Area}(v)+\sigma_k^4\int_\Sigma\lvert\sff^v\rvert^4\vol_{g_v}+H\vol(f_v),
        \end{align*}
        \item 
        \begin{align*}
            \sigma_k^4\log\sigma_k^{-1}\int_\Sigma\lvert\sff^{u_k}\rvert^4\vol_{g_{u_k}}\leq \alpha_k,
        \end{align*}
        \item \begin{align*}
        \operatorname{Area}(u_k) \leq \overline A.
        \end{align*} 
    \end{enumerate}

\begin{rmk}\label{rmk: definition of volume}    
Here $\{\alpha_k\}_{k\in \mathbb{N}}$ is a sequence converging to zero, while the parameters $\theta_k$ can be chosen to be arbitrarily small. We choose $\theta_k:=\sigma_k^5$. The notion of $\theta_k$-almost criticality—a quantitative measure of proximity to a critical point—is introduced in Definition \ref{def: reduction of functional}. Following \cite{Pigati-FB}, we work with almost critical points rather than exact critical points; in particular, we do not need to verify the Palais–Smale condition for the functionals $E_{H,\sigma}$.

Roughly speaking, $\vol(f_v)$ denotes the volume enclosed by the immersion $v$. This functional is not, in general, globally well defined on the space of immersions. However, as observed in \cite{Cheng-Zhou} (see also Section \ref{Sec: Preliminaires} below), any two possible values of $\vol(f_v)$ differ by an integer multiple of $\vol(\mathcal{M})$. Consequently, the notion of a critical point (or almost critical point) of $E_{H,\sigma}$ is well defined.
\end{rmk}

The construction of good sweep-outs, for which both the area and the quantity in (2) are uniformly controlled, is achieved via a Struwe-type monotonicity argument applied to the two parameters $\sigma$ and $H$ (see Lemma \ref{lem: monotonicity-trick}, based on \cite[Lemma 3.3]{Cheng-Zhou}; this is not immediate, as the third term in $E_{H,\sigma}(v)$ could be negative). This argument guarantees the required control for almost every $H>0$, as well as for $H=0$. Once such a sequence has been constructed, we follow the strategy of Pigati and Rivière to show that the desired branched immersion arises— in an appropriate sense— as a limit of the maps $u_k$. More precisely, we obtain the following.
\begin{thm}\label{thm: main-theorem}
    Let $H\geq0$. Let $\sigma_k\to 0$ and let $\{u_k\}_{k\in \mathbb{N}}$ be a sequence of immersions in $W^{1,2}(\Sigma,\mathcal{M})$ satisfying conditions (1), (2) and (3).\\
    Then there exist Riemann surfaces $\Sigma_\infty^1,...,\Sigma_\infty^N$ with 
    \begin{align*}
        \sum_{i=1}^Ng(\Sigma_{\infty}^i)\leq g(\Sigma),
    \end{align*}
    and maps $u_\infty^i:\Sigma_\infty^i\to \mathcal{M}$ (for $i=1,...,N$) such that $u_\infty^i$ is a smooth $H$-CMC branched immersion.\\
    Moreover, the
    varifolds $\textbf{v}_{u_k}$ induced by the maps $u_k$ converge--- up to subsequences--- to the varifold $\Sigma_{i=1}^N\mathbf{v}_{u_\infty^i}$, where $\mathbf{v}_{u_\infty^i}$ is the varifold induced by $u_\infty^i$.
\end{thm}
In the proof of Theorem \ref{thm: main-theorem}, we introduce a few simplifications to the original argument of \cite{PR-Multiplicity-One} to control the multiplicity of the limiting varifold (see Subsection \ref{subsection: sketch} and Remark \ref{rmk: simplifications}).\\
To the best of our knowledge, this is the first application of the ``viscosity method" to obtain existence of minimal/CMC branched immersions in a concrete geometric setting (an existence result for geodesics by the ``viscosity method" was obtained in \cite{Michelat-Riviere}). 

\begin{rmk}\label{rmk: index-almost-crit}
The CMC surfaces constructed in Theorem \ref{thm: main-theorem} are expected to have Morse index at most one, as they arise from a one-parameter min–max procedure. In the case $H=0$, this has been established in \cite{Michelat-Index-Viscosity} and \cite{Riviere-index-viscosity}. If such an index bound holds for $H>0$, then by arguing as in \cite[Section 5]{Cheng-Zhou}, one can show that when $\operatorname{Ric}_g>0$, the areas of the $H$-CMC immersions obtained for almost every $H>0$ are uniformly bounded. By taking limits, this would imply the existence of a branched $H$-CMC immersion for every $H>0$.

In the present work, however, we restrict our attention to almost critical points of the functional $E_{H,\sigma}$, which do not provide sufficient control over the Morse index of the resulting immersions. Establishing such index bounds will be the subject of future work.
\end{rmk}

While our approach allows us to control the genus of the resulting immersions from above, the following question remains open.

\begin{question}
Is it possible to obtain lower bounds for the genus and the Morse index of the surfaces constructed by this method? More specifically, given a closed three-manifold $(\mathcal{M}^3,g)$ of Heegaard genus $h$, does there exist, for $H\ge 0$ (or for almost every $H\ge 0$), a nontrivial branched $H$-CMC immersion from a connected surface of genus $h$ into $\mathcal{M}$?
\end{question}

\subsection{Related backgroud on minimal and CMC surfaces}

Historically, substantial effort has been devoted to the construction of minimal surfaces with controlled topology. T. Radó \cite{Rado} and J. Douglas \cite{Douglas} independently solved the celebrated Plateau problem, establishing the existence of area-minimizing disks in $\mathbb{R}^3$. On general Riemannian manifolds, a major breakthrough was achieved by J. Sacks and K. Uhlenbeck \cite{SacksUhlenbeck-2sphere}, who constructed minimal two-spheres in any non-aspherical manifold via an $\alpha$-harmonic approximation of the Dirichlet energy; see also earlier work of R. Schoen and S. T. Yau \cite{SchoenYau}. Recently, T. Colding and W.Minicozzi \cite{ColdingMinicozzi-Width} introduced the harmonic replacement method to study the existence of minimal spheres directly, without relying on approximation schemes.

The aforementioned works rely fundamentally on the fact that the sphere admits only one conformal class up to diffeomorphism. In dimension two, the Dirichlet energy of a conformal map coincides with the area of its image. Consequently, if a map $u$ is a critical point of the energy functional— namely, a harmonic map— and is conformal, then its image represents a minimal surface. Since the sphere possesses a unique conformal class, any harmonic map from a sphere into a Riemannian manifold is automatically conformal and hence yields a minimal surface. This argument breaks down for Riemann surfaces of positive genus, where the conformal class is no longer unique. In this setting, the above methods give only the existence of harmonic maps. To overcome this difficulty, X. Zhou \cite{Zhoutorus,Zhouhighgenus} treated the conformal class as an additional parameter in the Dirichlet energy, combining this approach with the harmonic replacement method to establish the existence of minimal surfaces with controlled topology. 

Parallel to the search for minimal surfaces, a vast literature has developed around the existence of constant mean curvature hypersurfaces. The existence of solutions to the CMC Plateau problem in $\mathbb{R}^3$ was first established by E. Heinz \cite{Heinz} and S. Hildebrandt \cite{Hilderbrandt}. The Rellich conjecture, asserting the existence of at least two solutions to the CMC Plateau problem, was later resolved by H. Brezis and J.-M. Coron \cite{Brezis} and by M. Struwe \cite{StruweLargeSurfaces,StruweNonuniqueness}. M. Struwe also studied free boundary CMC surfaces in $\mathbb{R}^3$ using the heat flow method \cite{Struwe-Freeboundary}. Related results on prescribed mean curvature surfaces were obtained in \cite{Hildebrandt2,GulliverSpruck,GulliverSpruck2,Steffen,Duzaar}.

For closed CMC hypersurfaces, boundaries of isoperimetric regions are known to be smoothly embedded CMC hypersurfaces (up to a singular set of codimension seven), see \cite{AlmgrenMem,Morgan}. However, this approach does not provide control over the value of the mean curvature, nor over the topology of the hypersurface. On the other hand, R. Ye \cite{Ye} and F. Mahmoudi-R. Mazzeo-F. Pacard \cite{MahmoudiMazzeoPacard} used perturbation methods to construct $H$-CMC hypersurfaces for $H$ large. Using gluing techniques, N. Kapouleas \cite{Kapouleas,Kapouleas2} and C. Breiner and N.Kapouleas \cite{BreinerKapouleas} constructed many important examples of CMC surfaces in $\mathbb{R}^3$.

Regarding general existence results, X. Zhou and J. Zhu \cite{Zhou-Zhucmc,ZhouZhupmc} and C. Bellettini and N. Wickramasekera \cite{BelletiniWickramasekera1,BelletiniWickramasekera2,BelletiniWickramasekera3} established the existence of almost embedded prescribed mean curvature (PMC) hypersurfaces with prescribed mean curvature $h$, where $h$ is an element in a broad function class that in particular includes all constants. X. Zhou and J. Zhu treat ambient dimensions $3\le n+1\le 7$, using Almgren--Pitts min--max methods, whereas C. Bellettini and N. Wickramasekera obtain existence in arbitrary ambient dimension via an Allen--Cahn type approach. These methods, however, do not provide control over the topology of the resulting hypersurfaces.\\
More recently, D. R. Cheng and X. Zhou \cite{Cheng-Zhou} established the existence of $H$-CMC branched immersed spheres for a.e. $H>0$ when the target is a sphere (and for all mean curvature values if the target sphere has positive Ricci curvature). This provides a positive answer to a weaker version of a conjecture of H. Rosenberg and G. Smith \cite[Page 3]{RosenbergSmith}. The original Rosenberg–Smith conjecture, asserting the existence of embedded H-CMC 2-spheres for any value $H\geq 0$ in $\S^3$ with positive sectional curvature, is false in general: F. Torralbo \cite{Torralbo} and W. H. Meeks, P. Mira, J. Pérez and A. Ros \cite{MeeksMiraPerezRos} showed that in certain positively curved homogeneous three-spheres, there exist mean curvature values for which all immersed CMC two-spheres must have self-intersections. D. R. Cheng and X. Zhou’s results were further extended to the free boundary and capillary settings by D. R. Cheng \cite{Cheng-freeboundary,cheng-capillary}. See also the work by R. Gao and M. Zhu \cite{gao-zhu} on the existence of PMC sphere in any codimensions.

A parallel line of research is based on geometric measure theory. The Almgren–Pitts min–max theory \cite{AlmgrenTopology,Pitts} provided the first general existence result for closed minimal hypersurfaces. The theory was substantially advanced by the work of F. C. Marques and A. Neves, following their resolution of the Willmore conjecture \cite{MarquesNeves-Willmoreconjecture}. Yau’s conjecture on the existence of infinitely many closed minimal hypersurfaces was proved in this framework by F. C. Marques and A. Neves in the positive Ricci curvature setting \cite{MarquesNevesInfinitelyMinimalSurfaces}, and later in full generality by A. Song \cite{Song}. Meanwhile, X. Zhou \cite{Zhou-Multiplicityone} confirmed the multiplicity-one conjecture in min–max theory \cite{MarqueNevesAdv}. A variant of this theory--- based on variations by ambient isotopies--- was developed by Simon and Smith \cite{SimonSmith} in ambient dimension 3 (see also \cite{ColdingDelellis}). This approach permits to control the genus of the resulting minimal surface (see \cite{DeLellisPellandini}); however, this approach doesn't seem to generalize to the CMC case, as even for minimizers, regions with mean curvature zero might appear (see \cite{Sarnataro-Stryker}). Recent work combining multiplicity one property and Simon-Smith min-max theory with a refined analysis of the topology of surface spaces has produced minimal surfaces of prescribed topology in three-manifolds with bumpy metrics or positive Ricci curvature; examples include four minimal two-spheres in $\S^3$ by Z. Wang and X. Zhou \cite{wangzhou4spheres}, minimal tori in $\S^3$ by X. Li and Z. Wang \cite{LiWangtori} and A. Chu and Y. Li \cite{ChuLitori}, and higher-genus minimal surfaces in $\mathbb{S}^3$ by A. Chu, Y. Li, and Z. Wang \cite{ChuLiWanggenus2} and by A. Chu \cite{Chu}.
We also refer the reader to related developments in \cite{DeLellisRamic,Guaraco,LiokumovichMarqueNeves,IrieMarqueNeves,MarquesNevesSong,Montezuma,LiZhou,Li,Dey,Wang2}.

\subsection{Overview of the arguments} \label{subsection: sketch}
Here we give an overview of the proofs of Theorems \ref{Prop: existence of qualifying sequence} and \ref{thm: main-theorem}, focusing on the novel contributions of this work.

To prove \ref{thm: main-theorem}, we adapt the arguments of \cite{Riviere-IHES}, \cite{PR-Regularity}, \cite{PR-Multiplicity-One} and \cite{Pigati-FB} to the CMC setting.

First we observe that the varifolds $\textbf{v}_{u_k}$ induced by the maps $u_k$ have bounded mass, so they converge weakly--- up to subsequences--- to a limiting varifold $\textbf{v}$. One can combine the almost criticality of the maps $u_k$ with Assumption $(2)$ (sometimes referred to as ``entropy condition"), to show that the mass ratios of the varifolds $\textbf{v}_{u_k}$ are controlled from below (Lemma \ref{prop: almost monotonicity for large radius}), and that the limiting varifold $\textbf{v}$ has generalized mean curvature bounded in $L^\infty$ (Lemma \ref{lem: convergence_of_varifolds}). Now fix a reference metric $g_0$ on $\Sigma$, and assume that the metrics induced by the maps $u_k$ are conformal to $g_0$. We will explain later how to remove this assumption. Then by Assumption $(3)$, the maps $u_k$ are uniformly bounded in $W^{1,2}_{g_0}(\Sigma)$, and--- up to subsequences--- they converge weakly in $W^{1,2}_{g_0}(\Sigma)$ to a limiting map $u_\infty$.

The measures $\nu_k:=\frac{1}{2}\lvert \nabla u_k\rvert^2 d\vol_{g_0}$ have bounded mass, and thus converge weakly--- up to subsequences--- to some Radon measure $\nu$ on $\Sigma$. By the estimates on the mass ratios obtained before, we show that $\nu$ is of the form
\begin{align*}
    \nu=f\,d\vol_{g_0}+\sum_{i=1}^N\alpha_i\delta_{p_i},
\end{align*}
where $\alpha_i>c_Q$ for some constant $c_Q>0$ (Lemma \ref{lem: concentration property} and Lemma \ref{lem: limit of measures}). Next we show that $f=N\, J(du_\infty)$, where $N$ takes values in $\mathbb{N}$ and $J(du_\infty)$ denotes the Jacobian of $u_\infty$ (Lemma \ref{lem: N-jacobian}). This is proved by showing that, for Lebesgue points $x$ of $du$, for a sequence of radii $r_i\to 0$ and a subsequence $k_i$, the rescaled varifolds $\textbf{v}'_i:=(r_i^{-1}(\cdot-u(x)))_\ast\textbf{v}_{k_i, r_i}$ converge weakly to $N(x)\mathcal{H}^2\vert_{\operatorname{co}(\mathcal{C})}$, where $N(x)\in \mathbb{N}$ and $\operatorname{co}(\mathcal{C})$ is the convex hull of $\{du(x)[y]\vert y\in\partial B_1(0)\}$. Next, for a smooth, open $\omega\subset\Sigma$, denote by $\textbf{v}_\omega$ the integral varifold induced by $u\vert_\omega$ and the multiplicity function $N$ (see Remark \ref{rmk: rectifiable-varifold}). We show that the varifolds induced by $u_k\vert_\omega$ converge to $\textbf{v}_\omega$ (Lemma \ref{lem: Pigati-5.7}). Repeating the argument of Lemma \ref{prop: first variation} for variations supported away from $u(\partial\omega)$, and studying the term involving $H$ in the first variation (Lemma \ref{prop: convergence-of-H-term}), we show that for a.e. $\omega$ (in the sense of Definition \ref{def: parametrized-CMC})
\begin{align}\label{eq: PHCMC-non-conf}
    \int_\omega N\operatorname{div}(X)J(du_\infty)d\vol_{g_0}=H\int_\omega d\vol^\mathcal{M}_{u_\infty(x)}(X\circ u_\infty, \partial_x u_\infty, \partial_y u_\infty)d\vol_{g_0},
\end{align}
for any smooth vector field $X$ on $\mathcal{M}$ supported away from $u_\infty(\partial\omega)$.
Finally, one shows that there exist a quasi-conformal diffeomorphism $\varphi$ such that $u:=u_\infty\circ\varphi$ is conformal (Lemma \ref{lem: 5-12-P-FB}), so that $u$ satisfies \eqref{eq: PHCMC-non-conf} for a.e. $\omega$. We call such a tuple $(\Sigma, u, N)$ a \textit{parametrized $H$-CMC varifold} (see Definition \ref{def: parametrized-CMC}).

In general, this argument can only be performed locally, and one must take into account the possibility of the conformal class degenerating along the sequence, or the possible appearance of concentration points of the energy (``bubbles"). We will show in Subsection \ref{subsec: degeneration} that while the Dirichlet energy might accumulate around points or loops, no energy is lost in the "neck regions" (Lemma \ref{prop: energy estimate for two boundary components}), so that---after appropriate reparametrization--- one can recover the whole varifold $\textbf{v}$ as a sum of parametrized $H$-CMC varifolds defined on possibly multiple Riemann surfaces, whose sum of genera is bounded by $g(\Sigma)$.

Next we show that $N\equiv 1$ a.e.. To this end we study the convergence of the varifolds $\textbf{v}_i'$ in the proof of Lemma \ref{lem: N-jacobian}. In \cite{PR-Multiplicity-One}, Pigati and Rivière showed directly that for a sequence of maps $\{u_k\}_{k\in \mathbb{N}}$ as in Theorem \ref{thm: main-theorem} (with $H=0$), the corresponding varifolds converge to a parametrized stationary varifold with $N=1$ a.e.. Here we adapt their strategy to the sequence of varifolds $\textbf{v}_i'$, which--- if $x$ is a Lebesgue point of $du$ and $\operatorname{rank}(du)=2$--- converge to the varifold induced by $du(x)\vert_{B_1(0)}$. Working with such a simple limiting object allows us to simplify a few steps in the argument. The key idea in this argument is that one can define a notion of ``averaged multiplicity" $n(v_i, B_\alpha, B^\Pi_\beta)$ that encodes the average multiplicity of the image of $v_i:=r_i^{-1}(u_{k_i}-u(x))$ (defined on $B_1(0)$) in the ball $B_\alpha$ in $\mathbb{R}^Q$ (with respect to some ball $B^\Pi_\beta$ in a 2-dimensional plane $\Pi$ in $\mathbb{R}^Q$), and that if the map $v_i$ is sufficiently close to a linear map (as ensured by the closeness of $v_i$ to $\nabla u(x)$), one can inductively define rescalings $v_i^{(j)}$ of $v_i$, such that any $v_i^{(j)}$ has the same averaged multiplicity on appropriate balls (Lemma \ref{prop: iteration lemma}, cf. Lemma 5.3 in \cite{PR-Multiplicity-One}).

Now, on one hand the averaged multiplicity $n(v_i, B_\alpha, B^{\Pi_i}_\beta)$ (for appropriate $\alpha, \beta>0$ and planes $\Pi_i$) converges to $N(x)$ as $i\to\infty$. On the other hand, if $j$ is sufficiently large, Assumption $(2)$ is used to show--- using ideas going back to the works of Langer \cite{LangerCompactness} and Breuning \cite{Breuning}--- that $\pi_\Pi\circ u_{r_i, k_i}$ (where $\pi_\Pi$ denotes the orthogonal projection to the 2-plane $\Pi$ spanned by $\mathcal{C}$) is injective on a small ball, and therefore $n(v_i^{(j)}, B_\alpha, B^{\Pi_i}_\beta)=1$ for $j$ sufficiently large. This implies that $N(x)=1$ (Theorem \ref{thm: multiplicity-one}). One can now apply ideas first introduced by Rivière in \cite{Riviere-Targetharmonicmap} for target harmonic maps to show that $u$ is smooth and satisfies \eqref{eq: equation of CMC} (Theorem \ref{thm: smoothness-and-equation}), so that $u$ is a branched $H$-CMC immersion (Corollary \ref{cor: HW}). This completes the proof of Theorem \ref{thm: main-theorem}.

Next we explain how we construct a sequence of maps $u_k$ satisfying Assumptions $(1)$, $(2)$ and $(3)$ in Theorem \ref{thm: main-theorem}, for $H=0$ and a.e. $H>0$. In this part we follow the general strategy of \cite{Cheng-Zhou}, but for the functionals $E_{H,\sigma}$ defined above.
Given a degree one sweep-out $\gamma: \Sigma\times[0,1]\to\mathcal{M}$ (in the sense of Definition \ref{def: sweep-out-family}), where $\Sigma$ is a Riemann surface, $\gamma(\cdot, t)$ is a smooth immersion for any $t\in (0,1)$ and the images of $\gamma\vert_{\Sigma\times\{0\}}$, $\gamma\vert_{\Sigma\times\{1\}}$ are graphs in $\mathcal{M}$.
We denote $\mathcal{P}_h$ the family of all such sweep-outs, where $h$ is the genus of $\Sigma$.
For $t\in [0,1]$ set $f_{\gamma,t}(s,x):=\gamma(ts)(x)$.
We let
\begin{align*}
    \vol(f_{\gamma,t}) =\int_{\Sigma\times [0,1]}f_{\gamma,t}^\ast d\vol_\mathcal{M}
\end{align*}
to be the volume swept by $\gamma$ over the interval $[0,t]$. Note that for an immersion $u:\Sigma\to \mathcal{M}$, there might exist different such sweep-outs $\gamma$, $\gamma'$ such that $\gamma(t)=u=\gamma'(t')$. By \cite{Cheng-Zhou}, $\vol(f_{\gamma,t})$ and $\vol(f_{\gamma',t'})$ only differ by an integer multiple of $\vol(\mathcal{M})$. We will denote any such a sweep-out by $f_u$.

Morally, for any $\sigma>0$, we would like to find sweep-outs approximating the width relative to $E_{H,\sigma}$ and $\mathcal{P}_h$, but such width would be equal to infinity for any $\sigma>0$, as the penalization term $\int_\Sigma\lvert\sff^{\gamma(t)}\rvert^4$ would tend to infinity for $t\to 0$ and $t\to 1$. Instead, for any sweep-out $\gamma\in \mathcal{P}_h$ we find an interval $I_\gamma^H\subset (0,1)$ such that $\int_\Sigma\lvert\sff^{\gamma(t)}\rvert^4$ does not degenerate for $t\in I_\gamma^H$, but $I_\gamma^H$ contains the $t$ for which $E_{H,\sigma}(\gamma(t),f_{\gamma,t})$ approaches the width of the non-penalized problem $(\sigma=0)$ (see Definition \ref{def: good-intervals}). We then define 
\begin{align*}
    \omega_{H,\sigma}=\min_{\gamma\in \mathcal{P}_h}\max_{t\in I^H_{\gamma}}E_{H,\sigma}(\gamma(t),f_{\gamma,t}).
\end{align*}
Applying Struwe's monotonicity trick to the two parameters $\sigma$ and $H$, for which the functionals $E_{H,\sigma}$ satisfy certain monotonicity properties, we show that for almost any $H>0$ (and for $H=0$, for which one can repeat the argument just for the parameter $\sigma$), there exist a sequence $\sigma_k\to 0$ and sweep-outs $\gamma_k$ close to realizing the width and for which--- for $t$ close to realizing the width--- the area and the penalization term $\int_\Sigma\lvert\sff^{\gamma_k(t)}\rvert^4$ are controlled as desired (Lemma \ref{lem: monotonicity-trick}).

Finally, one can apply a deformation argument by pseudo-gradient flow (Lemma \ref{lem: new-existence-almost-critical-points}) to show that there must exist times $t_k$ such that the maps $u_k=\gamma( t_k)$ satisfy the desired properties. In the construction of the deformations, one has to be careful to preserve the properties of the interval $I_\gamma^H$. To this end, we follow the argument of \cite{Michelat-Riviere}.

\subsection{Organization of the paper}
In Section \ref{Sec: Preliminaires}, we recall and establish several basic properties of the functionals $E_{H,\sigma}$ and of their first variation. In Section \ref{sec: existence of cmc}, we establish the basic convergence scheme to prove existence of parametrized CMC varifolds. In Section \ref{sec: regularity}, we prove the regularity of the CMC parametrized varifolds constructed in Section \ref{sec: existence of cmc}, by showing that their multiplicity is equal to one, and prove Theorem \ref{thm: main-theorem}. In Section \ref{sec: existence of torus}, we construct a min-max theory to find the desired almost critical points, and complete the proof of Theorem \ref{Prop: existence of qualifying sequence}.

\subsection*{Acknowledgements}
The authors are grateful to Xin Zhou for suggesting this problem and for helpful discussions, and to Alessandro Pigati for valuable conversations. F. G. wants to thank Prof. Tristan Rivière for introducing him to the viscosity method. X. L. wants to thank Prof. Xin Zhou and Prof. Daniel Stern for their encouragement throughout this work and their intellectual guidance.\\
F. G. was supported by the Swiss National Science Foundation (SNSF)
through the Postdoc.Mobility grant, project number 230344.\\
X. L. was partially supported by NSF DMS-2243149 and NSF DMS-2404992.
\section{Preliminaries}\label{Sec: Preliminaires}
In this section, we discuss the basic properties of the functional $E_{H,\sigma}$. We assume that $\mathcal{M}$ is isometrically embedded in a large Euclidean space $\R^Q$. Let $h$ be the Heegaard genus of $\mathcal{M}$ and $\Sigma$ be a genus $h$ surface.

As in the works of Rivière and Pigati \cite{Riviere-IHES, Pigati-FB}, we work in the space of Sobolev immersions
\begin{align*}
    \mathfrak{M}:=W^{2,4}_{\text{imm}}(\Sigma,\mathcal{M})=\{ u\in W^{2,4}(\Sigma,\mathcal{M}):\operatorname{rank}(du)=2\}.
\end{align*}
Note that by the Sobolev embedding, each $u\in W^{2,4}(\Sigma,\mathcal{M})$ lies in $C^{1,\frac{1}{2}}(\Sigma,\mathcal{M})$. Hence, $du$ is well-defined everywhere.\par
The space $\mathfrak{M}$ is a Banach manifold. Indeed, it is an open subset of the Banach manifold $W^{2,4}(\Sigma,\mathcal{M})$, since the condition $\operatorname{rank}(du_x)=2$ is open in the $C^1$ topology.
Given $u\in\mathfrak{M}$, the tangent space of $\mathfrak{M}$ at $u$ is
\begin{align*}
    T_u\mathfrak{M}=\{ X\in W^{2,4}(\Sigma, \mathbb{R}^Q):X(x)\in T_{u(x)}\mathcal{M}\text{ for all }x\in\Sigma\}.
\end{align*}

Write the pull back metric of $u$ as $g_u=u^*g$. The norm we use on $T_u\mathfrak{M}$--- which makes $\mathfrak{M}$ into a Finsler manifold--- is
\begin{align*}
    \Vert X\Vert_u=\Vert X\Vert_{C^{1,1}(g_u)}+\Vert\nabla^2 X\Vert_{L^4(g_u)}.
\end{align*}
Here the norms and connections are all with respect to $g_u$ on $\Sigma$ and $g$ on $\mathcal{M}$. The natural distance associated to this norm is
\begin{align*}
    \operatorname{d}_{\mathfrak{M}}(u_1,u_2)=\inf\left\{\int_0^1\Vert h'(t)\Vert_{h(t)}\,dt:h\in C^1([0,1],\mathfrak{M}),h(0)=u_1,h(1)=u_2\right\}.
\end{align*}
Then we have the following.
\begin{proposition}[{\cite[Proposition 2.1]{Pigati-FB}}]
    $(\mathfrak{M},\operatorname{d})$ is a complete Finsler manifold.
\end{proposition}
As $\mathcal{M}$ has Heegard genus $h$, there exists a splitting $\mathcal{M}=N_1\cup_\varphi N_2$, where $N_1$, $N_2$ are handlebodies of genus $h$, and $\varphi: \partial N_1\to\partial N_2$ is a diffeomorphism (see, e.g., \cite{HeegaardSplittingNotes}). Note that a standard handlebody of genus $h$ can be constructed as the tubular neighborhood of a graph $\gamma_h$. Note also that $\Sigma\simeq \partial N_1\simeq\partial N_2$.
Then each handlebody $N_i$ can be identified with a quotient of $[0,1]\times\Sigma$; more precisely, there is a
quotient map $\pi: \Sigma\to \gamma_h$ such that the quotient $([0,1]\times \Sigma)/\sim$ (where $(0,x)\sim (0,x')$ if $\pi(x)=\pi(x')$, for $x,x'\in \Sigma$) is a handlebody of genus $h$.

We define $\mathcal{E}(u)$ to be the set of maps $f\in C^0([0,1]\times\Sigma, \mathcal{M})$ such that $f$ descend to a map on the quotient $([0,1]\times\Sigma)/\sim\, \simeq N_1$, and such that $f(1,\cdot)=u$.

Heuristically speaking, a map $f\in \mathcal{E}(u)$ describes a region bounded by $u(\Sigma)$. 
\begin{definition}
    For each $H\geq0,\sigma>0,u\in\mathfrak{M},f\in\mathcal{E}(u)\cap W^{1,3}([0,1]\times\Sigma,\mathcal{M})$, let
    \begin{align*}
        E_{H,\sigma}(u,f)=\operatorname{Area}(u)+\sigma^4\int_{\Sigma}\lvert\sff^u\rvert^4\,d\vol_{g_u}+H\vol(f),
    \end{align*}
    where $\sff^u$ is the second fundamental form of $u$ in $\mathcal{M}$ and 
    \begin{align*}
        \vol(f)=\int_{[0,1]\times\Sigma}f^*\,d\vol_g.
    \end{align*}
\end{definition}
The term $\vol(f)$ has been used by D. R. Cheng-X. Zhou in \cite{Cheng-Zhou} to characterize the volume bounded by $u$. A similar quantity was used by M. Struwe in \cite{Struwe-Freeboundary}. We recall the basic properties of $\vol(f)$ for $f\in\mathcal{E}(u)$ in the following Lemma.

\begin{lemma}[{\cite[Lemma 2.2]{Cheng-Zhou}}]\label{lem: define of volume}
    For each $u\in\mathfrak{M}$,
    \begin{enumerate}
        \item If $f_1,f_2\in\mathcal{E}(u)\cap W^{1,3}([0,1]\times\Sigma,\mathcal{M})$, then
        \begin{align*}
            \vol(f_1)-\vol(f_2)\in \vol_g(\mathcal{M})\Z;
        \end{align*}
        \item There exists $\delta_0>0$ such that if $f_1,f_2\in\mathcal{E}(u)\cap W^{1,3}([0,1]\times\Sigma,\mathcal{M})$ and $\lVert f_1-f_2\rVert_{C^0}<\delta_0$, then $\vol(f_1)=\vol(f_2)$.
    \end{enumerate}
\end{lemma}
In view of Lemma \ref{lem: define of volume}, we can naturally extend the definition of $\vol$ to all $f\in\mathcal{E}(u)$ by defining $\vol(f)=\vol(\tilde{f})$ with $\tilde{f}\in\mathcal{E}(u)\cap W^{1,3}([0,1]\times\Sigma,\mathcal{M})$ and $\lVert f-\tilde{f}\rVert_{C^0}<\delta_0$. With this definition, Lemma \ref{lem: define of volume} still holds for $f\in\mathcal{E}(u)$.

By Lemma \ref{lem: define of volume}, we can define the volume functional on simply connected open subsets of $\mathfrak{M}$, to consider its first variation.
\begin{definition}\label{def: reduction of functional}
    Let $\mathfrak{A}\subset\mathfrak{M}$ be a simply connect open set. Let $u_0\in\mathfrak{A}$ and $f_0\in\mathcal{E}(u_0)$.
    \begin{enumerate}
        \item For another $u\in\mathfrak{A}$, take a path $q\in C([0,1],\mathfrak{A})$ with $q(0)=u_0,q(1)=u$. Let us concatenate $q$ and $f_0$ by defining
        \begin{align*}
            f(t,\cdot)=\begin{cases}
                f_0(2t);&t\in\left[0,\frac{1}{2}\right];\\
                q(2t-1);&t\in\left[\frac{1}{2},1\right].
            \end{cases}
        \end{align*}
        Then $f\in\mathcal{E}(u)$. We define the local reduction of $E_{H,\sigma}$ on $\mathfrak{A}$ induced by $(u_0,f_0)$ by
        \begin{align*}
            E_{H,\sigma}^{\mathfrak{A}}(u)=E_{H,\sigma}(u,f).
        \end{align*}
        \item The first variation of $E_{H,\sigma}^{\mathfrak{A}}$ is given by
        \begin{align*}
            \delta E_{H,\sigma}^{\mathfrak{A}}(u)(X)=\frac{d}{dt}\bigg |_{t=0}E_{H,\sigma}^{\mathfrak{A}}(u_t),\text{ for }u_0=u,\de_tu_t|_{t=0}=X\in T_u\mathfrak{M},
        \end{align*}
        where $u_t,t\in(-1,1)$ is a $C^1$ family in $\mathfrak{M}$. We also define $\delta E_{H,\sigma}(u)=\delta E_{H,\sigma}^{\mathfrak{A}}(u)$.
        \item We say that $u\in\mathfrak{M}$ is a $\theta$-critical point of $E_{H,\sigma}$ if
        \begin{align*}
            \lvert\delta E_{H,\sigma}(u)(X)\rvert\leq\theta\Vert X\Vert_{u},\text{ for }X\in T_u\mathfrak{M}.
        \end{align*}
    \end{enumerate}
\end{definition}

\begin{rmk}
    By Lemma \ref{lem: define of volume}, $E_{H,\sigma}^{\mathfrak{A}}(u)$ does not depend on the choice of $q$. Also, when choosing a different $f_0\in \mathcal{E}(u_0)$, the values of $E_{H,\sigma}^{\mathfrak{A}}(u)$ will only change by a constant in $\vol_g (\mathcal{M})\Z$. Therefore, the first variation $\delta E_{H,\sigma}(u)$ is well-defined.
\end{rmk}

Now, let us compute the first variation of $E_{H,\sigma}$.

\begin{proposition}\label{prop: first variation}
    Let $u\in\mathfrak{M},H\geq0,\sigma>0$. Assume $\operatorname{Area}(u)\leq A$. The first variation $\delta E_{H,\sigma}(u)$ satisfies
        \begin{align} \label{eq: control-first-variation}
        \nonumber
            &\left\lvert \delta E_{H,\sigma}(u)(X)-\int_{\Sigma}\langle\nabla u,\nabla X\rangle\,d\vol_{g_u}-H\int_{\Sigma}(d\vol_g)_u(X, u_x, u_y)\,dx\wedge dy\right\rvert\\\leq &C\sigma^4\lVert\sff^u\rVert^3_{L^4(g_u)}(\lVert\lvert\nabla^2 X\rvert_{g_u}\rVert_{L^4(g_u)}+\Vert X\Vert_{L^{\infty}}A^{1/4}+\Vert\lvert\nabla X\rvert_{g_u}\Vert_{L^{\infty}}\lVert\sff^u\rVert_{L^4(g_u)}),
        \end{align}
        for all $X\in T_u\mathfrak{M}$. Recall that $\,d\vol_g$ denotes the volume form of $g$.
\end{proposition}
\begin{proof}
    The computation combines arguments from \cite{Cheng-Zhou} and \cite{Pigati-FB}.
    Let $X\in T_u\mathfrak{M}$.
    For the first variation of $\vol(f)$, exactly as in \cite[Section 2.3]{Cheng-Zhou} we have 
    \begin{align*}
        \delta \vol(f_u)(X)=\int_{\Sigma}(d\vol_g)_u(X, u_x, u_y)\,dx\wedge dy.
    \end{align*}
    It is also clear that the first variation of $\operatorname{Area}(u)$ is  
    \begin{align*}
        \int_\Sigma\operatorname{div}_{T_{u(x)}\Sigma}Xd\vol_{g_u}=\int_{\Sigma}\langle\nabla u,\nabla X\rangle\,d\vol_{g_u}
    \end{align*}
    (here $\nabla$ is the pull-back connection on $u^\ast T\mathcal{M}$).
    For the first variation of $\int_{\Sigma}\lvert\sff^u\vert^4$, take $u_t$ be a $C^1$ curve in $\mathfrak{M}$ with $u_0=u$ and $\de_tu_t|_{t=0}=X$. In local coordinates we have
    \begin{align*}
        \sff^{u_t}(du(\de_{\alpha}),du(\de_{\beta}))=\nabla du_t(\de_{\alpha},\de_{\beta});
    \end{align*}
    \begin{align}\label{eq: t-der-second-derivative}
        \nu\cdot\frac{D}{dt} \left[\nabla d u_t\left(\partial_\alpha,\partial_\beta\right)\right]\bigg\vert_{t=0}=\nu\cdot\left(\nabla\nabla X(\partial_\alpha,\partial_\beta)+R^\mathcal{M}(X,\partial_\alpha u)\partial_\beta u\right),
    \end{align}
    where $\nu$ is a unit vector in $T_{u(x)}\mathcal{M}$, orthogonal to $du(x)[T_x\mathcal{M}]$.\\
    For simplicity, write $g_t=g_{u_t}$. Similarly, in local coordinates we have  \begin{align} \label{eq: first variation of seond fundamental form}
        \frac{1}{2}\frac{d}{dt}\lvert\sff^{u_t}\rvert_{g_{t}}^2\bigg\vert_{t=0}=&g^{\alpha\beta}_0g^{\gamma\lambda}_0\langle \nabla d u(\partial_\alpha,\partial_\gamma), \nabla\nabla X(\partial_\beta,\partial_\lambda )+R^\mathcal{M}(X,\partial_\beta u)\partial_\lambda u \rangle\\ \nonumber
        &+\frac{d}{dt}\bigg\vert_{t=0}(g_t^{\alpha\beta}g_t^{\gamma\lambda})\langle\nabla d u(\partial_\alpha,\partial_\gamma),\nabla d u(\partial_\beta,\partial_\lambda)\rangle,
    \end{align}
    and recall that
    \begin{align}\label{eq: first variation of inverse metric}
        \frac{d}{dt}(g_t^{\alpha\beta}g_t^{\gamma\lambda})\bigg\vert_{t=0}=-g_0^{\alpha\eta} \frac{d}{dt}\bigg\vert_{t=0}(g_t)_{\eta\theta}g_0^{\theta\beta}g_0^{\gamma\lambda}-g_0^{\alpha\beta}g_0^{\gamma\eta} \frac{d}{dt}\bigg\vert_{t=0}(g_t)_{\eta\theta}g_0^{\theta\lambda},
    \end{align}
    \begin{align}\label{eq: first variation of metric}
        \frac{d}{dt}(g_t)_{\alpha\beta}\bigg|_{t=0}=\langle\de_{\alpha}u,\nabla_{\de_{\beta}}X\rangle+\langle\de_{\beta}u,\nabla_{\de_{\alpha}}X\rangle.
    \end{align}
    
    Combining \eqref{eq: first variation of seond fundamental form}, \eqref{eq: first variation of inverse metric} and \eqref{eq: first variation of metric} we obtain
    \begin{align}\label{eq: derivative-sff}
        \left\lvert\frac{d}{dt}\lvert\sff^{u_t}\rvert_{g_{t}}^2\bigg\vert_{t=0}\right\rvert\leq C(\lvert\sff^u\rvert_{g_u}(\lvert\nabla^2 X\rvert_{g_u}+\vert X\vert_{g_u})+\lvert\sff^u\rvert^2_{g_u}\lvert\nabla X\rvert_{g_u}).
    \end{align}
    Also, 
    \begin{align}\label{eq: derivative-jacobian}
        \left\lvert\frac{d}{dt}\sqrt{\det g_t}\bigg|_{t=0}\right\rvert=\left\lvert g_u^{\alpha\beta}\langle \nabla_{\partial_\alpha}X,\partial_\beta u\rangle\right\rvert \sqrt{\det g_u}\leq C\vert\nabla X\vert_{g_u}\sqrt{\det g_u}.
    \end{align}
    The control of $\frac{d}{dt}\vert_{t=0}\int\vert\sff^{u_t}\vert^4\,d\vol_{g_u}$ follows. Combining the previous estimates and applying H\"older's inequality, we obtain the bound for the first variation.
\end{proof} 

We have the following continuity properties.
\begin{lemma}\label{lem: closeness-A-E}
    Let $u,u'\in \mathfrak{M}$ with $\operatorname{dist}_\mathfrak{M}(u,u')\leq 1$. Assume that
    \begin{align*}
        \operatorname{Area}(u)\leq A,\quad \int_{\Sigma}\lvert\sff^u\rvert^4d\operatorname{vol}_{g_{u}}\leq E.
    \end{align*}
    Then
    \begin{align*}
        \lvert \operatorname{Area}(u)-\operatorname{Area}(u') \rvert&\leq C_A\operatorname{dist
        }_\mathfrak{M}(u, u'),\\
        \left\lvert\int_{\Sigma}\lvert\sff^u\rvert^4d\operatorname{vol}_{g_{u}}-\int_{\Sigma}\lvert\sff^{u'}\rvert^4d\operatorname{vol}_{g_{u'}}\right\rvert&\leq C_{A}(E+1)\operatorname{dist
        }_\mathfrak{M}(u, u'),
    \end{align*}
    where $C_A$ depends only on $\mathcal{M}$ and $A$.
\end{lemma}
\begin{proof}
    Let $\varepsilon>0$. Let $u_t\in C^1([0,1],\mathfrak{M})$ be a curve with $u_0=u$, $u_1=u'$ and length $\int_0^1\Vert\de_tu_t\Vert_{u(t)}\,dt\leq \operatorname{dist}_{\mathfrak{M}}(u, u')+\varepsilon$.\\
    As in \eqref{eq: derivative-jacobian}, we compute
    \begin{align*}
        \left\lvert\frac{d}{dt}\int_{\Sigma}d\operatorname{vol}_{g_{u_t}}\right\rvert=\left\lvert\int_{\Sigma}\langle \nabla\de_tu_t,d u_t\rangle d\vol_{g_{u_t}}\right\rvert\leq 2\left\lVert\de_t u_t\right\rVert_{u_t}\operatorname{Area}(u_t).
    \end{align*}
    Intergrating this inequality we get
    \begin{align}\label{eq: max-area}
        \max_{t\in [0,1]}\operatorname{Area}(u_t)\leq Ae^2,
    \end{align}
    and
    \begin{align*}
        \lvert \operatorname{Area}(u)-\operatorname{Area}(u')\rvert\leq 2Ae^2(\operatorname{dist}_{\mathfrak{M}}(u, u')+\varepsilon).
    \end{align*}
    For the continuity of the term involving the second fundamental form, we recall from \eqref{eq: first variation of seond fundamental form}, \eqref{eq: first variation of inverse metric} and \eqref{eq: first variation of metric} that
    \begin{align*}
        \left\lvert\frac{d}{dt}\lvert\sff^{u_t}\rvert^4_{g_t}\right\rvert\leq C\left(\lvert\sff^{u_t}\rvert^3_{g_t}\lvert \nabla^2 \de_t u_t\rvert_{g_t}+\lvert\sff^{u_t}\rvert^4_{g_t}\lvert \nabla \de_t u_t\rvert_{g_t}+\lvert\sff^{u_t}\rvert^3_{g_t}\lvert \de_tu_t\rvert_{g_t}\right).
    \end{align*}
    Therefore
    \begin{align*}
        &\left\lvert\frac{d}{dt}\int_{\Sigma}\lvert \sff^{u_t}\rvert^4_{g_t}d\operatorname{vol}_{g_t}\right\rvert\\\leq&C \left\lVert \de_t u_t\right\rVert_{u_t}\left(\left(\int_{\Sigma}\lvert\sff^{u_t}\rvert^4_{g_t}d\operatorname{vol}_{g_t}\right)^\frac{3}{4}(1+(Ae^2)^\frac{1}{4})+\int_{\Sigma}\lvert\sff^{u_t}\rvert^4_{g_t}d\operatorname{vol}_{g_t}\right)\\
        \leq&
        C\left\lVert \de_t u_t\right\rVert_{u_t}\left(\left(\int_{\Sigma}\lvert\sff^{u_t}\rvert^4_{g_t}d\operatorname{vol}_{g_t}\right)+1\right).
    \end{align*}
    Then
    \begin{align*}
        \left\lvert\frac{d}{dt}\log\left(\int_{\Sigma}\lvert\sff^{u_t}\rvert^4_{g_t}d\operatorname{vol}_{g_t}+1\right)\right\rvert\leq C\left\lVert\de_tu_t\right\rVert,
    \end{align*}
    Integrating this, we get
    \begin{align*}
        \max_{t\in [0,1]}\int_{\Sigma}\lvert\sff^{u_t}\rvert^4_{g_t}d\operatorname{vol}_{g_t}\leq C(E+1).
    \end{align*}
    Thus
    \begin{align*}
         \left\lvert\frac{d}{dt}\int_{\Sigma}\lvert \sff^{u_t}\rvert^4_{g_t}d\operatorname{vol}_{g_t}\right\rvert\leq C(E+1)
        \left\lVert \de_t u_t\right\rVert_{u_t}.
    \end{align*}
    Integrating once more and taking the limit $\varepsilon\to 0$ we obtain the desired statement.
\end{proof}

We also have the following continuity property for the first variation.
\begin{proposition}\label{prop: continuity of first variation}
    There exists $\tau_0>0$ such that the following holds. For any $H>0$ and $u,u'\in\mathfrak{M}$, if
        \begin{align*}
            \operatorname{d}_{\mathfrak{M}}(u,u')\leq\tau\leq\tau_0,\text{ and }\operatorname{Area}(u)\leq A, 
        \end{align*}
        and 
        \begin{align*}
            \int_\Sigma \lvert \sff^u\rvert^4d\vol_{g_u}\leq E,
        \end{align*}
        we have 
        \begin{align*}
            \lvert\lVert \delta E_{H,\sigma}(u)\rVert-\lVert\delta E_{H,\sigma}(u')\rVert\rvert\leq C\tau,
        \end{align*}
        where the constant depends on $E, A, \sigma, H$ and $\mathcal{M}$.
\end{proposition}
\begin{proof}
Pick a $C^1$ curve $u_t$ with $u_0=u,u_1=u'$ and 
    \begin{align*}
        \int_0^1\Vert \de_tu_t\Vert_{u_t}\,dt\leq 2\tau.
    \end{align*}
    Set $g_t=g_{u_t}$, then for all $X\in\Gamma(T\Sigma)$, we calculate as in the proof of Proposition 2.1 in \cite{Pigati-FB}
    \begin{align*}
        \left\vert\frac{d}{dt}g_t(X,X)\right\vert=2\vert\langle du_t(X),\nabla_X\de_tu_t\rangle\vert\leq2g_t(X,X)\lvert \de_tu_t\rvert_{u_t},
    \end{align*}
    since $\lvert du_t(X)\rvert=\lvert X \rvert_{g_t}$ and $\lvert\nabla_X\de_tu_t\rvert\leq \lvert \de_tu_t\rvert_{u_t}\lvert X\rvert_{g_t}$.
    Integrating the above inequality, we obtain $e^{-4\tau}g_u\leq g_{u'}\leq e^{4\tau}g_u$. By the same argument as above, for all $X,Y\in T\Sigma$,
    \begin{align*}
        \left\vert\frac{d}{dt}g_t(X,Y)\right\vert\leq2\vert X\vert_{g_t}\vert Y\vert_{g_t}\lvert \de_tu_t\rvert_{u_t}\leq2e^{4\tau}\vert X\vert_{g_u}\vert Y\vert_{g_u}\lvert \de_tu_t\rvert_{u_t}.
    \end{align*}
    Integrating again, we obtain
    \begin{align}\label{eq: distance-metrics}
        \vert g_u(X,Y)-g_{u'}(X,Y)\vert\leq 4\tau e^{4\tau}\vert X\vert_{g_u}\vert Y\vert_{g_{u}}.
    \end{align}

Let $\tau_0$ be such that, if in local coordinates ${(g_u)}_{\alpha\beta}(x)=\delta_{\alpha\beta}$, then $\lVert g_{u'}-\operatorname{Id}\rVert_{Op}\leq \frac{1}{2}$.\\
In the following, we allow the constants $C$ to depend on $A,E,\tau_0$ and $\mathcal{M}$.\\
For $u\in \mathfrak{M}$, $X\in T_u\mathfrak{M}$, let
\begin{align*}
    \mathfrak{F}(u)[X]:=\int_{\Sigma}\lvert\sff^u\rvert^2(\lvert\sff^{u}\rvert^2\langle du, \nabla X\rangle+4\langle\sff^u, \nabla\nabla X+R^\mathcal{M}(X,\cdot)\rangle-8\langle du\otimes\nabla X, \operatorname{tr}_1(\sff^u\otimes\sff^u)\rangle)d\vol_{g_u},
\end{align*}
where $\operatorname{tr}_1$ refers to the trace over the last index (in local coordinates $g_u^{\gamma\lambda}\langle\sff^u_{\alpha\gamma},\sff_{\beta\lambda}^u\rangle$).
Note that $\mathfrak{F}(u)$ corresponds to the first variation of the penalization term in $E_{H,\sigma}$.
Let $X':=\Pi_{u'}X$, where $\Pi_{u'}$ is the orthogonal projection from $\mathbb{R}^Q$ to $T_{u'(x)}\mathcal{M}$. We need to estimate $\mathfrak{F}(u)[X]-\mathfrak{F}(u')[X']$. We estimate the three summands in $\mathfrak{F}$ (denoted $\mathfrak{F}_1$, $\mathfrak{F}_2$ and $\mathfrak{F}_3$) separately.
    \begin{align*}
       \left\lvert \mathfrak{F}_1(u)[X]-\mathfrak{F}_1(u')[X']\right\rvert\leq& \lVert\lvert \nabla X\rvert_{g_u}\rVert_{L^\infty}\left\lvert\int_\Sigma \lvert\sff^u\rvert^4 \vol_{g_u}-\lvert\sff^{u'}\rvert^4 d\vol_{g_{u'}}\right\rvert\\&+\left\lvert\int_\Sigma \lvert\sff^{u'}\rvert^4d\vol_{g_{u'}}\right\rvert\lVert\langle du, \nabla X\rangle_{g_u} -\langle du', \nabla' X'\rangle_{g_{u'}}\rVert_{L^\infty}\\
       \leq&C\lVert X\rVert_u\tau.
    \end{align*}
    In order to estimate the difference in the $L^\infty$ norm, we used the fact that for any $x$, using normal coordinates centered at $x$, one can write the difference as
    \begin{align*}
        g^{\alpha\beta}_u\langle du(\partial_\alpha), \nabla_{\partial_\beta}X\rangle-g^{\alpha\beta}_{u'}\langle du'(\partial_\alpha), \nabla'_{\partial_\beta}X'\rangle
    \end{align*}
    and use estimates \eqref{eq: distance-metrics} (together with our choice of $\tau_0$), $\lVert \lvert\nabla X-\nabla'X'\rvert_{g_u}\rVert_{L^\infty}\leq C\lVert X\rVert_{u}\tau$ and $\lvert du-du'\rvert_{g_u}\leq C\tau$.\\
    Next, to estimate the $\mathfrak{F}_2$ term we observe that
     \begin{align*}
       \int_\Sigma\lvert\nabla\nabla X-\nabla'\nabla' X'\rvert_{g_u}^4d\vol_{g_u}\leq C\lVert X\rVert_{u}^4\tau^4
    \end{align*}
    and
    \begin{align*}
        \left\lVert g_u^{\alpha\beta}R^\mathcal{M}(X,\partial_\alpha u)\partial_\beta u-g_{u'}^{\alpha\beta}R^\mathcal{M}(X',\partial_\alpha u')\partial_\beta u'\right\rVert_{L^\infty}\leq C\lVert X\rVert_{u}\tau.
    \end{align*}
    The last estimate can be derived pointwise on normal charts for $g_u$. Moreover, there holds
    \begin{align}\label{eq: integral-control-sff}
        \int_\Sigma\lvert \sff^u-\sff^{u'}\rvert_{g_u}^4d\vol_{g_u}\leq& \int_\Sigma\left(\int_0^1\lvert\nabla\nabla\partial_t u_t+R^\mathcal{M}(\partial_t u_t,\cdot)\rvert_{g_u}+C\lvert \sff^{u_t}\rvert_{g_u}\lvert \nabla \partial_t u\rvert_{g_u} dt\right)^4 d\vol_{g_u}\\
        \nonumber
        \leq&C\left(\int_0^1\lVert\partial_t u_t\rVert_{u_t}dt\right)^4\leq C\tau^4,
    \end{align}
    as can be verified using \eqref{eq: t-der-second-derivative}, \eqref{eq: distance-metrics} and Minkowski's integral inequality.
    Using \eqref{eq: derivative-jacobian} to control the difference in the volume forms and applying H\"older's inequality, we obtain $\left\lvert \mathfrak{F}_2(u)[X]-\mathfrak{F}_2(u')[X']\right\rvert\leq C\lVert X\rVert_{u}\tau$. The corresponding bound for $\mathfrak{F}_3$ can be derived similarly from the estimates $\lvert du-du'\rvert_{g_u}\leq C\tau$, $\lVert \lvert\nabla X-\nabla'X'\rvert_{g_u}\rVert_{L^\infty}\leq C\lVert X\rVert_{u}\tau$, \eqref{eq: integral-control-sff} and \eqref{eq: derivative-jacobian}.
    Therefore we obtain
    \begin{align*}
        \lvert \mathfrak{F}(u)[X]-\mathfrak{F}(u')[X']\rvert\leq C\lVert X\rVert_{u}\tau.
    \end{align*}
    The other terms in $\delta E_{H,\sigma}$ can be controlled with a simpler argument, to obtain
    \begin{align*}
        \lvert \delta E_{H,\sigma}(u)[X]-\delta E_{H,\sigma}(u')[X']\rvert\leq C\lVert X\rVert_{u}\tau.
    \end{align*}
    One can verify that $\lvert \lVert X\rVert_u-\lVert X'\rVert_{u'}\rvert\leq C\lVert X\rVert_u\tau$. As this holds for any $X\in T_{u}\mathfrak{M}$, and one can repeat the argument for $u'$ in place of $u$ (for similar constants $A'$ and $E'$, by Lemma \ref{lem: closeness-A-E}), we obtain the desired statement.
\end{proof}

Finally, we remark that $E_{H,\sigma}$ has the following scaling property.
\begin{lemma}\label{lem: critical-rescaled-functional}
Let $r,\ell>0$.
Assume that $u$ is $\theta$-critical for $E_{H,\sigma}$ on $W^{1,2}(B_1(0),\mathcal{M})$. Set $\tau=\sigma \ell^{-1}$. For a fixed $x_0\in B_1(0)$ with $B_r(x_0)\subset B_1(0)$, let $p:=u(x_0)$ and set
    \begin{align*}
    v(x)=\ell^{-1}\left(u(x_0+rx)-p\right).   \end{align*}
    The map $v$ is $\ell^{-\frac{5}{2}}\theta$-almost critical for the functional
    \begin{align*}
        E_{\ell H,\tau}(v)=\operatorname{Area}(v\vert_{B_1(0)})+\tau^4\int_{B_1(0)}\lvert\sff^{v}\rvert^{4}d\operatorname{vol}_{g_v}+\ell H\int_{B_1(0)\times[0,1]}f_{v}^\ast d\vol_{\mathcal{M}_{p,\ell}},
    \end{align*}
    on $W^{2,4}(B_1(0),\mathcal{M}_{p,\ell})$, where $\mathcal{M}_{p,\ell}:\ell^{-1}(\mathcal{M}-p)$,
    and, if $\ell\leq 1$,
    \begin{align*}
        \tau^4\log\tau^{-1}\int_{B_1(0)}\lvert\sff^{v}\rvert^4d\vol_{g_v}\leq \ell^{-2}\sigma^4\log\sigma^{-1}\int_{B_r(x_0)}\lvert\sff^u\rvert^4d\vol_{g_u}.
    \end{align*}
\end{lemma}
\begin{proof}
    We compute
    \begin{align*}
        \operatorname{Area}(v\vert_{B_1(0)})=&\frac{1}{2}\int_{B_1(0)}\lvert\nabla v\rvert^2=\frac{\ell^{-2}}{2}\int_{B_1(0)}\lvert r\nabla u(x_0+ry)\rvert^2dy\\
        =&\frac{\ell^{-2}}{2}\int_{B_r(x_0)}\lvert \nabla u(z)\rvert^2 dz=\ell^{-2}\operatorname{Area}(u\vert_{B_r(x_0)}).
    \end{align*}
    Moreover, as $\sff^{v(y)}=\ell\sff^{u(x_0+ry)}$,
    \begin{align}\label{eq: rescaling-penalization}
    \nonumber
        \int_{B_1(0)} \lvert\sff^{v}\rvert^4d\operatorname{vol}_{g_v}=&\int_{B_1(0)}\ell^4\lvert\sff^{u}(x_0+ry)\rvert^4\frac{\lvert\nabla v(y)\rvert^2}{2}dy=\frac{\ell^4r^2}{2\ell^2}\int_{B_1(0)}\lvert\sff^{u}(x_0+ry)\rvert^4\lvert\nabla u(x_0+ry)\rvert^2 dy\\
        =&{\ell^2}\int_{B_r(x_0)}\lvert\sff^{u}(z)\rvert^4d\vol_{g_u}.
    \end{align}
    We also have
    \begin{align*}
        \int_{B_1(0)\times [0,1]}f_{v}^\ast   d\vol_{\mathcal{M}_{p,\ell}}=\ell^{-3}\int_{B_r(x_0)\times [0,1]}f_{u}^\ast d\vol_{\mathcal{M}},
    \end{align*}
    where $f_{v}(t, y)=\ell^{-1}f_{u}(t, x_0+ry)$.\\
    Thus
    \begin{align*}
        E_{\ell H,\tau}(v)=\ell^{-2}\left(\operatorname{Area}(u\vert_{B_r(x_0)})+\sigma^4\int_{B_r(x_0)}\lvert\sff^{u}\rvert^{4}d\vol_{g_u}+H\int_{B_r(x_0)\times[0,1]}f_u^\ast d\vol_\mathcal{M}\right).
    \end{align*}
    Now if $X$ is a variation of $u$, supported in $B_r(x_0)$, set $X^{r,\ell}(y):=\ell^{-1}X(x_0+ry)$. $X^{r,\ell}$ is a variation of $v$ supported in $B_1(0)$, and
    \begin{align*}
        \left\lvert\delta E_{\ell H,\tau}(v)(X^{r,\ell})\right\rvert=\left\lvert \ell^{-2}\delta E_{H,\sigma}(u)(X)\right\rvert\leq \ell^{-2}\lVert \delta E_{H,\sigma}(u)\rVert_{{W^{2,4}}^\ast_{g_{u}}}\lVert X\rVert_{W^{2,4}_{g_{u}}}.
    \end{align*}
    Note that
    \begin{align*}
        &\lVert X^{r,\ell}\rVert_{L^\infty}=\ell^{-1}\lVert X\rVert_{L^\infty},\, \lVert \nabla^{g_{v}}X^{r,\ell}\rVert_{L^\infty}=\lVert \nabla^{g_{u}}X\rVert_{L^\infty},\\&\lVert\nabla^{g_{v}}\nabla^{g_{v}} X^{r,\ell}\rVert_{L^4_{g_{v}}(B_1(0))}=\ell^\frac{1}{2}\lVert\nabla^{g_{u}}\nabla^{g_{u}} X\rVert_{L^4_{g_{u}}(B_r(x_0))}.
    \end{align*}
    Therefore $\lVert X\rVert_{W_{u}^{2,4}}\leq \ell^{-\frac{1}{2}}\lVert X^{r,\ell}\rVert_{W_{v}^{2,4}}$ (if $r\leq 1)$, so that
    \begin{align*}
        \left\lvert\delta E_{\ell H,\tau}(v)(X^{r,\ell})\right\rvert\leq \ell^{-\frac{5}{2}}\lVert \delta E_{H,\sigma}(u)\rVert_{{W^{2,4}}^\ast_{g_{u}}}\lVert X^{r,\ell}\rVert_{W_{v}^{2,4}}\leq \ell^{-\frac{5}{2}}\theta\lVert X^{r,\ell}\rVert_{W_{v}^{2,4}}.
    \end{align*}
    This shows that $v$ is a $\ell^{-\frac{5}{2}}\theta$-critical point of $E_{\ell H,\tau}$.\\
    Finally, note that, if $\ell\leq 1$,
    \begin{align*}
        \tau^4\log\tau^{-1}\int_{B_1(0)}\lvert\sff^{v}\rvert^4d\vol_{g_v}=& \sigma^4\ell^{-4}\log(\sigma^{-1}\ell)\ell^2\int_{B_r(x_0)}\lvert\sff^{u}\rvert^4d\vol_{g_u}\\\leq&\ell^{-2}\sigma^4\log\sigma^{-1}\int_{B_r(x_0)}\lvert\sff^u\rvert^4d\vol_{g_u}.
    \end{align*}
\end{proof}

\section{The existence of CMC parametrized varifolds}\label{sec: existence of cmc}

\subsection{Asymptotic behavior}
In this section we show that a family of maps $\{u_k\}_{k\in \mathbb{N}}$ satisfying the assumptions of Theorem \ref{thm: main-theorem} converges, up to subsequences and in a suitable sense, to a CMC parametrized varifold, defined as follows.
\begin{definition}\label{def: parametrized-CMC}
    Given $u\in W^{1,2}(\Sigma,\mathcal{M})$ and an integer valued measurable function $N$ on $\Sigma$, we say that the triple $(\Sigma, u, N)$ is a parametrized varifold with constant mean curvature equal to $H$ if for a.e. $\omega\subset \Sigma$, for any smooth vector field $X$ on $\mathcal{M}$ supported away from $u(\partial \omega)$ there holds
    \begin{align*}
        \sum_{i}^2\int_\omega N\langle \partial_i u,DX\partial_iu\rangle dx=H\int_\omega X\cdot \ast\partial_{x_1}u\wedge\partial_{x_2}u\, dx=H\int_\omega u^\ast \alpha_X ,
    \end{align*}
    where $\alpha_X:=\ast X^\flat$.
\end{definition}
Here and in the following, we say that a property holds \textit{for a.e. $\omega\subset\Sigma$} if, for every non-negative $\rho\in C^\infty(\Sigma)$, for a.e. $\lambda>0$ the property holds for the super-level set $\lbrace\rho>\lambda\rbrace$.

In this section we follow closely the arguments of \cite{Pigati-FB}, especially Sections 3-6.\\
Denote the 2-Grassmannian on $\mathcal{M}$ by $G_2(\mathcal{M})$. For any $u\in \mathfrak{M}$, define the push-forward varifold $\textbf{v}_u$ by
\begin{align*}
    \textbf{v}_u(\varphi)=\int_{\Sigma}\varphi(u(x),du(T_{u(x)}\Sigma))\,d\operatorname{vol}_{g_u},\text{ for all }\varphi\in C^{\infty}(G_2(\mathcal{M})).
\end{align*}
For any open set $\omega\subset\Sigma$, we also define the local push-forward by
\begin{align*}
    \textbf{v}_{u,\omega}(\varphi)=\int_{\omega}\varphi(u(x),du(T_{u(x)}\Sigma))\,d\operatorname{vol}_{g_u},\text{ for all }\varphi\in C^{\infty}(G_2(\mathcal{M})).
\end{align*}
For $u=u_k$ we sometime write $\textbf{v}^k_\omega$ for $\textbf{v}_{u_k,\omega}$.
Let $\pi:G_2(\mathcal{M})\rightarrow\mathcal{M}$ be the natural projection. For each 2-varifold $\textbf{v}$, $\mu_{\textbf{v}}(U)=\textbf{v}(\pi^{-1}(U))$ is the area measure on $\mathcal{M}$ associated to $\textbf{v}$.

\begin{lemma}\label{lem: convergence_of_varifolds}
    Let $\{u_k\}_{k\in \mathbb{N}}$ be as in Theorem \ref{thm: main-theorem}.
    For each smooth open set $\omega\subset\Sigma$, if $u_k(\de\omega)$ converge in the Hausdorff distance to a curve $\Gamma\subset\mathcal{M}$, then the induced varifolds $\textbf{v}_{u_k,\omega}$ converge in the sense of varifolds, up to subsequences, to a varifold $\textbf{v}_\omega$ with bounded first variation away from $\Gamma$. More precisely
    \begin{align*}
        \vert\delta \textbf{v}_\omega(X)\vert\leq H\int_{\mathcal{M}}\vert X\vert d\mu_{\textbf{v}},\text{ for all smooth}X\in \Gamma(T\mathcal{M}), \text{ with }X\text{ supported away from }\Gamma.
    \end{align*}
\end{lemma}

\begin{proof}
    Let $X\in \Gamma(T\mathcal{M})$ be such that $X=0$ near $\Gamma$. For any $k\in \mathbb{N}$, set $X_k=X\circ (u_k|_{\omega})$. Since by assumption $u_k(\de\omega)$ converges in the Hausdorff distance to $\Gamma$, we have $X_k=0$ near $\de\omega$, for $k$ large enough. Extend  $X_k$ (by zero) to a vector field on $\Sigma$ along $\mathcal{M}$, supported in $\omega$. Let $x_\alpha$ denote local coordinates on $\omega$. We compute 
    \begin{align}\label{eq: explicit-derivative-pullback}
    \nonumber
    \nabla_\frac{\partial}{\partial x_\alpha} X_k=&\nabla^{T\mathcal{M}}_{\partial_\alpha u_k}X,\\
        \nabla\nabla X_k \left(\frac{\partial}{\partial x_\alpha} , \frac{\partial}{\partial x_\beta}\right)=&\nabla^{T\mathcal{M}}\nabla^{T\mathcal{M}}X(\partial_\alpha u_k, \partial_\beta u_k)+\nabla^{T\mathcal{M}}X(\sff^{u_k}(\partial_\alpha u_k, \partial_\beta u_k) ).
    \end{align}
    Here $\nabla$ denotes the pullback connection on $u_k^\ast T\mathcal{M}$. In the second estimate we used the fact that $\nabla du_k(X,Y)=\sff^{u_k}({u_k}_\ast X,{u_k}_\ast Y)$, see \cite{Pigati-FB}, Proposition 2.6. 
    In particular,
    with respect to the pullback metric $g_{u_k}$ on $\Sigma$, we have
    \begin{align*}
        \vert\nabla X_k\vert\leq\vert \nabla^{T\mathcal{M}} X\circ u_k\vert,
    \end{align*}
    and
    \begin{align}\label{eq: norm-of-variation-in-limit}
       \Vert X_k\Vert_{u_k}\leq\Vert X\Vert_{C^{0,1}}+\overline{A}^{\frac{1}{4}}\Vert(\nabla^{T\mathcal{M}})^2 X\Vert_{L^{\infty}}+\Vert\nabla^{T\mathcal{M}} X\Vert_{L^{\infty}}\Vert\sff^{u_k}\Vert_{L^4(g_{u_k})}.
    \end{align}
    Hence, by assumption $(2)$ in Theorem \ref{thm: main-theorem}, there holds
    \begin{align}\label{eq: sigmaXk-infinitesimal}
        \sigma_k\Vert X_k\Vert_{u_k}\leq\sigma_k(\Vert X\Vert_{C^{0,1}}+\overline{A}^{\frac{1}{4}}\Vert\nabla^2 X\Vert_{L^{\infty}})+o\left(\frac{1}{\log^\frac{1}{4}\sigma_k^{-1}}\right).
    \end{align}
    Note that the right hand side tends to $0$ as $k\to\infty$.
    Combining the estimate for the first variation of $E_{H,\sigma_k}$ obtained in Proposition \ref{prop: first variation} with \eqref{eq: explicit-derivative-pullback}, we get
    \begin{align*}
            &\left\lvert\delta E_{H,\sigma_k}(u_k)(X_k)-\int_{\Sigma}\langle\nabla u_k,\nabla X_k\rangle\,d\vol_{g_{u_k}}-H\int_{\Sigma}(d\vol_g)_{u_k}(X_k,(u_k)_x,(u_k)_y)\,dx\wedge dy\right\rvert\\\leq &C\sigma_k^4\lVert\sff^{u_k}\rVert^3_{L^4(g_{u_k})}\left(\overline{A}^{1/4}\lVert(\nabla^{T\mathcal{M}})^2 X\rVert_{L^\infty}+\overline{A}^{1/4}\Vert X\Vert_{L^{\infty}}+\lVert \sff^{u_k}\rVert_{L^4(g_{u_k})}\Vert\nabla^{T\mathcal{M}} X\Vert_{L^{\infty}}\right),
    \end{align*}
    for each vector field $Y\in T_{u_k}\mathfrak{M}$ (here $g$ denotes the metric of $\mathcal{M}$).
    Note that the right hand side of the above expression tends to $0$ as $k\to\infty$ by assumption (2) in Theorem \ref{thm: main-theorem}.\\
    Observe that
    \begin{align*}
        \int_{\omega}\langle\nabla u_k,\nabla X_k\rangle\,d\vol_{g_{u_k}}=\int_{G_2(\mathcal{M})}\div_LX(p)\,d\textbf{v}_{u_k,\omega}(p,L)=\delta \textbf{v}_{u_k,\omega}(X),
    \end{align*}
    \begin{align*}
        \left\lvert H\int_{\Sigma}(\,d\vol_g)_{u_k}(X_k,(u_k)_x,(u_k)_y)\,dx\wedge\,dy\right\rvert\leq H\int_{\mathcal{M}}\vert X\vert\,d\mu_{{\textbf{v}}_{u_k,\omega}}.
    \end{align*}
    Since we have the uniform area bound Area $(u_k)\leq \overline{A}$, after passing to a subsequence we may assume that $\textbf{v}_{u_k,\omega}\rightharpoonup \textbf{v}_\omega$ (varifold convergence), for some varifold $\textbf{v}_\omega$.
Since $u_k$ is assumed to be a $\theta_k$-critical point of $E_{H,\sigma_k}$, and $\theta_k\leq\sigma_k$, there holds 
   \begin{align*}
       \vert\delta E_{H,\sigma_k}(u_k)(X_k)\vert\leq\theta_k\Vert X_k\Vert_{u_k}\leq \sigma_k\Vert X_k\Vert_{u_k},
   \end{align*} 
   and the right hand side tends to zero as $k\to\infty$ by \eqref{eq: sigmaXk-infinitesimal}.
    Combining the above estimates, we obtain
    \begin{align}\label{eq: convergence-first-variation-to-H-term} \lim_{k\to\infty}\left\lvert \int_{Gr_2(\mathcal{M})}\div_LX(p)\,d\textbf{v}_\omega(p,L)-H\int_{\Sigma}(\,d\vol_g)_{u_k}(X_k,(u_k)_x,(u_k)_y)\,dx\wedge\,dy\right\rvert=0,
    \end{align}
    from which we deduce that
    \begin{align*}
        \vert\delta \textbf{v}_\omega(X)\vert=\left\vert\int_{Gr_2(\mathcal{M})}\div_LX(p)\,d\textbf{v}_\omega(p,L)\right\vert\leq H\int_{\mathcal{M}}\vert X\vert\,d\mu_{\textbf{v}_\omega}.    \end{align*}
\end{proof}
When $\omega=\Sigma$, We simply denote by $\textbf{v}$ the varifold $\textbf{v}_{\Sigma}$ obtained in Lemma \ref{lem: convergence_of_varifolds}.

In what follows, we assume that the sequence $\{u_k\}_{k\in\mathbb N}$ induces a single conformal structure on $\Sigma$. The case of varying conformal structures will be treated in Subsection $\ref{subsec: degeneration}$.
Accordingly, for each $k$ there exists a diffeomorphism $\varphi_k:\Sigma\to\Sigma$ such that
$$(u_k\circ \varphi_k)^*g = e^{f_k}\, g_0,$$
for some fixed reference metric $g_0$ on $\Sigma$ and some function $f_k:\Sigma\to\mathbb R$. Setting $\tilde u_k := u_k\circ \varphi_k$, we note that $\{\tilde u_k\}$ still satisfies the hypotheses of Theorem $\ref{thm: main-theorem}$, since $E_{H,\sigma}$ is invariant under reparametrizations. For simplicity, we relabel $\tilde u_k$ by $u_k$.
For any $k\in \mathbb{N}$, set $\nu_k:=\operatorname{vol}_{g_{u_k}}=\frac{1}{2}\lvert d u_k\rvert^2\vol_{{g_0}}$ and $\mu_k:=(u_k)_\ast\nu_k$. $\nu_k$ and $\mu_k$ are Radon measures respectively on $\Sigma$ and $\mathcal{M}$. Up to subsequences, $\{\nu_k\}_{k\in \mathbb{N}}$ converges weakly-$\ast$ to a Radon measure $\nu$ on $\Sigma$, while $\{\mu_k\}_{k\in \mathbb{N}}$ converges weakly-$\ast$ to a Radon measure $\mu$ on $\mathcal{M}$. Moreover, as the sequence $\{u_k\}_{k\in \mathbb{N}}$ is bounded in $W^{1,2}_{g_0}$ (by assumption (3) in Theorem \ref{thm: main-theorem} and the fact that $u_k$ is assumed to be conformal as a map from $(\Sigma, g_0)$), up to subsequences $\{u_k\}_{k\in \mathbb{N}}$ converges weakly in $W^{1,2}_{g_0}$ to a limiting map $u$.

Below we outline the main steps in the analysis of the asymptotic behavior of the sequences $\{u_k\}_{k\in \mathbb{N}}$ and $\{\nu_k\}_{k\in \mathbb{N}}$. Proofs that closely follow \cite{Pigati-FB} are omitted; for those arguments, we refer the reader to \cite{Pigati-FB}.

First of all, the almost criticality of the functions $u_k$ with respect to $E_{H,\sigma_k}$ can be exploited to deduce lower bounds for the mass ratio of the measures $\mu_k$:


\begin{proposition}[Cf. {\cite[Proposition 4.2]{Pigati-FB}}]\label{prop: almost monotonicity for large radius} Fix $k\in \mathbb{N}$. Let $x\in \Sigma$. Let $p=u_k(x)$. Let $\omega$ be a neighborhood of $x$ and set $\mu^\omega_k:=(u_k|_{\omega})_*\nu_k$.
Define a Radon measure on $\mathcal{M}$ by
    \begin{align*}
        \lambda_k^\omega(U)=\sigma_k^4\int_{u_k^{-1}(U)\cap\omega}\vert\sff^{u_k}\vert^4\,d\nu_k\text{ for any measurable }U\subset\mathcal{M}.
    \end{align*}
    Suppose that there holds
    \begin{align*}
        \lambda_k^\omega(B_t(p))\leq\delta_k\mu_k^\omega(B_{5t}(p))\text{ for all }t>0,
    \end{align*}
    for some $\delta_k>0$.
    Then if $\delta_k,\sigma_k$ are sufficiently small, we have for all $0<s<r<\diam(\mathcal{M})$
    \begin{align*}
        \frac{\mu^\omega_k(B_r(p))}{r^2}\geq(c-C\delta_k\log(r/s))\frac{\mu^\omega_k(B_s(p))}{s^2}-C\sigma_k^2.
    \end{align*}
    Moreover, for
    $\sigma_k<r<\diam(\mathcal{M)}$, we have
    \begin{align*}
        \frac{\mu_k^\omega(B_r(p))}{r^2}\geq c-C\delta_k\log(r/\sigma)-C\sigma_k^2.
    \end{align*}
    The constants $c,C$ depend only on $\mathcal{M}$ and $H$. 
\end{proposition}
\begin{rmk}[Cf. {\cite[Remark 4.5]{Pigati-FB}}]
The argument used to prove Proposition \ref{prop: almost monotonicity for large radius} (namely, the proof of \cite[Proposition 4.2]{Pigati-FB}) also shows that any 2-varifold $\textbf{v}$ supported in $\mathcal{M}$ with generalized mean curvature (with respect to $\mathcal{M}$) bounded in $L^\infty$ by $H$ (like the varifold $\textbf{v}$ obtained in Lemma \ref{lem: convergence_of_varifolds}) satisfies
\begin{align*}
    \frac{\lvert\textbf{v}\rvert(B_r(p))}{r^2}\geq (1+C\sqrt{r})^{-1}\frac{\lvert\textbf{v}\rvert(B_s(p))}{s^2},
\end{align*}
    for any $p\in \mathcal{M}$ and $0<s<r<\operatorname{diam}(\mathcal{M})$. In particular, this implies that the density
    \begin{align*}
        \theta(\textbf{v},p):=\lim_{s\to\infty}\frac{\lvert\textbf{v}\rvert(B_s(p))}{s^2},
    \end{align*}
    exists for any $p\in \mathcal{M}$, and
    \begin{align}\label{eq: control-density}
        c\theta (\textbf{v},p)r^2\leq \lvert \textbf{v}\rvert(B_r(p))\leq C\lvert\textbf{v}\rvert(\mathcal{M})r^2,
    \end{align}
    for any $p\in \mathcal{M}$ and $r\in (0,\operatorname{diam}(\mathcal{M}))$. All the constants depend only on $\mathcal{M}$ and $H$.
\end{rmk}
Proposition \ref{prop: almost monotonicity for large radius} can be combined with a covering argument to show that the varifolds induced by $u_k\vert_{B_r(x)}$ (for small balls $B_r(x)$ in a conformal chart) converge to varifolds $\textbf{v}_{B_r(x)}$ whose density is bounded below by some constant $c>0$, away from $u(\partial B_r)$. Therefore, Lemma \ref{lem: alnerative for density lower bound} (which corresponds to Lemma A.4 and Remark A.5 in \cite{Pigati-FB}) can be applied to obtain the following ``$\varepsilon$-regularity-type" dichotomy.
\begin{lemma}[Cf. {\cite[Proposition 5.1]{Pigati-FB}}]\label{lem: concentration property}
    Given $x\in\Sigma$ and $0<r<1$, assume that $u_k|_{\de B_r(x)}$ converges to the trace $u|_{\de B_r(x)}$ in $C^0$ and that $s=\diam(u(\de B_r(x)))<c_v$, for the constant $c_v=c_v(H,\mathcal{M})$ appearing in Lemma \ref{lem: alnerative for density lower bound}. Then either 
    \begin{enumerate}
        \item $\limsup_{k\rightarrow\infty}\nu_k(B_r(x))\geq c_Q$ for a constant $c_Q=c_Q(H,\mathcal{M})$ depending only on H, $\mathcal{M}$; or
        \item Any weak limit $\mu^{B_r(x)}$ of $\{(u_k|_{B_r(x)})_*\nu_k\}_{k\in \mathbb{N}}$ is supported in a $2s$-neighborhood of $u(\de B_r(x))$.
    \end{enumerate}
\end{lemma}
Case (1) of Lemma \ref{lem: concentration property} describes single-point energy concentration (``bubbling"). The Lemma implies that any atom in the limiting measure $\nu$ has mass at least $c_Q$, therefore there are only finitely many such concentration points. On the complement of atoms, Case (2) implies absolute continuity of the limit measure with respect to $\vol_{g_0}$. More precisely:
\begin{lemma}[Cf. {\cite[Theorem 5.2]{Pigati-FB}}]\label{lem: limit of measures}
    The limiting measure $\nu$ has finitely many atoms, with weight at least $c_Q$, for a constant $c_Q$ depending only on $\mathcal{M}$ and $H$. On the complement $\tilde\Sigma$ of the finite set of atoms, $\nu$ is absolutely continuous with respect to $\operatorname{vol}_{g_0}$, and $u$ has a continuous representative. Moreover, for any open set $\omega\Subset\tilde\Sigma$ with $\nu(\partial\omega)=0$, $(u_k\vert_{\omega})_\ast\nu_k\rightharpoonup (u\vert_{\omega})_\ast\nu_{\infty}$.
\end{lemma}
\begin{rmk}
    In the proof of Lemma \ref{lem: limit of measures}, one shows that for any compact $K\subset\tilde\Sigma$,
    \begin{align}\label{eq: convergence-dist-L1}
        \limsup _{k\to\infty}\int_K\lvert u_k-u\rvert d\nu_k=0
    \end{align}
    and
    \begin{align}\label{eq: control-nu-by-du}
        \nu(K)\leq C\int_{K}\lvert du\rvert^2,
    \end{align}
    with $C$ depending only on $\mathcal{M}$ and $H$.
\end{rmk}
In fact, we have the following, more precise description of the limiting measure $\nu$.
\begin{lemma}[Cf. {\cite[Theorem 5.3]{Pigati-FB}}]\label{lem: N-jacobian}
    The absolutely continuous part of $\nu$, which we denote by $m\operatorname{vol}_{g_0}$, has $m=0$ a.e. on the set of points where $du$ doesn't have rank 2. Moreover, $m=NJ(du)$ a.e. for a bounded, integer valued function $N\geq 1$ on the set of points where $du$ has rank 2.
\end{lemma}
\begin{proof}
    Let $x\in \Sigma$ be a Lebesgue point for $du$ with $\nu(\{x\})=0$. We will work in a  conformal chart (for $g_0$) centered at $x$, and consider $u$ and $u_k$ as maps from such a chart (so that $x$ is identified with $0$). We will show that
    \begin{align*}
        \frac{\nu(B_r(x))}{\pi r^2}\to N\lvert\partial_1 u\wedge\partial_2u\rvert(x)
    \end{align*}
    for some bounded integer $N\geq 1$, as $r\to 0$ along some subsequence.
    For all $r>0$ small enough, let $\textbf{v}_{k,r}$ be the varifold induced by $u_k\vert_{B_r}$. By Lemmas A.4 and A.5 in \cite{PR-Regularity}, we can select an arbitrarily small $r$ such that the trace $u\vert_{\partial B_r}$ satisfies
    \begin{align}\label{eq: condition-radii-1}
        u(ry)=u(0)+rdu(0)[y]+o(r)\text{ for }\lvert y\rvert=1,
    \end{align}
    and such that 
        the traces $u_k\vert_{\partial B_r}$ converge to $u\vert_{\partial B_r}$ in $C^0$ up to subsequences.\\
    By Lemma \ref{lem: concentration property}, if $r$ is sufficiently small, any (subsequential) weak limit of $\lvert \textbf{v}_{k,r}\rvert$ is supported in a ball $B_{Cr}(u(0))$, where $C$ depends on $\lvert \nabla u(0)\rvert$.\par
    Moreover, any (subsequential) weak limit $\textbf{v}_r=\lim_{k\to \infty}\textbf{v}_{k,r}$ has generalized mean curvature bounded by $H$ in $\mathcal{M}\smallsetminus u(\partial B_r)$ (by Lemma \ref{lem: convergence_of_varifolds}) and satisfies $\lvert \textbf{v}_r\rvert(B_s(p))\leq C_{\mathcal{M}, \lvert \mathbf{v}\rvert(\mathcal{M})}s^2$ for all $p\in \mathcal{M}$. Indeed, $\lvert \mathbf v\rvert$ obeys \eqref{eq: control-density}, and since $\lvert \mathbf v_{k,r}\rvert\le \lvert \mathbf{v}\rvert$ the same bound holds uniformly for $\lvert \mathbf v_{k,r}\rvert(B_s(p))$; passing to the limit in $k$ yields the estimate for $\mathbf v_r$.\\
    Let $r_i\to 0$ be a sequence of radii as above. We claim that there exist a sequence $\{k_i\}_{i\in \mathbb{N}}$ going to infinity such that the varifolds $\textbf{v}'_i:={(r_i^{-1}(\cdot-p))}_\ast\textbf{v}_{k_i, r_i}$ (where $p=u(x)$) satisfy the following properties
    \begin{enumerate}
        \item $\{\textbf{v}_i'\}_{i\in\mathbb{N}}$ is a tight sequence converging to a varifold $\textbf{v}'_\infty$;
        \item $\textbf{v}_\infty'$ satisfies
        \begin{align*}
        \lvert \textbf{v}_\infty'\rvert(B_s(q))\leq 2C_{\mathcal{M}, \lvert \textbf{v}\rvert(\mathcal{M})}s^2\text{ for all }q\in \mathbb{R}^Q\text{ and all }s>0;
    \end{align*}
      \item \begin{align*}
        r_i^{-2}\sigma_{k_i}^4\int_{B_{r_i}}\lvert\sff^{u_{k_i}}\rvert^4\vol_{u_{k_i}}\to 0,\quad r_i^{-1}\sigma_{k_i}\to 0;
    \end{align*}
    \item $\textbf{v}_\infty'$ has compact support and is stationary in $\mathbb{R}^Q\smallsetminus\mathcal{C}$, with
        \begin{align*}
        \mathcal{C}=\left\{du(0)[y]\vert y\in \partial B_1(0)\right\};
    \end{align*}
    \item \begin{align}\label{eq: condition-convergence-nuk-nuinfty}
        \lvert\textbf{v}_{\infty}'\rvert(\mathbb{R}^Q)=\lim_{i\to\infty}\lvert\textbf{v}_i'\rvert(\mathbb{R}^Q)=\lim_{i\to\infty}\frac{\nu_{k_i}(B_{r_i})}{r_i^2}=\lim_{i\to\infty}\frac{\nu(B_{r_i})}{r_i^2}.
    \end{align}
    \end{enumerate}    
    We now prove the claim. Fix $i\in \mathbb{N}$. We first show that there exist $\overline k_i$ such that if $k_i\geq \overline k_i$ and $\nabla u(x)\neq0$, then
    \begin{align}\label{eq: good-choice-k}
        \nu_{k_i}\left(\left\{z\in B_{r_i}(x)\vert\lvert u_{k_i}(z)-u(z)\rvert>2\lvert\nabla u(x)\rvert r_i\right\}\right)\leq r_i^3,
    \end{align}
    while if $\nabla u(x)=0$, then
    \begin{align*}
        \nu_{k_i}\left(\left\{z\in B_{r_i}(x)\vert\lvert u_{k_i}(z)-u(z)\rvert>2 r_i\right\}\right)\leq r_i^3.
    \end{align*}
    In fact, if this wasn't true, if $\nabla u(x)\neq 0$ there would exist a sequence $k_n\to\infty$ such that for any $k_n$,
    \begin{align*}
        \int_{B_{r_i}(x)}\lvert u-u_{k_n}\rvert d\nu_{k_n}\geq &2\lvert\nabla u(x)\rvert r_i\nu_{k_n}\left(\{z\in B_{r_i}(x)\vert\lvert u_{k_n}(z)-u(z)\rvert>2\lvert\nabla u(x)\rvert r_i\}\right)\\\geq& 2\lvert\nabla u(x)\rvert r_i^4,
    \end{align*}
    but taking $\limsup$ for $n\to\infty$, one would reach a contradiction to \eqref{eq: convergence-dist-L1}. The proof for the case $\nabla u(x)=0$ is similar.
    
    By Lemma \ref{lem: image-u-close-to-ball} below and the weak-$\ast$ convergence $\nu_k\rightharpoonup \nu$, for a.e. $x$ as above, choosing $k_i$ larger if necessary we can ensure that for any $k_i\geq \overline k_i$,
    \begin{align*}
        \nu_{k_i}(\{z\in B_{r_i}(x)\vert \lvert u(z)-p\rvert> 3\lvert \nabla u(x)\rvert r_i\})\leq 2\varepsilon_ir_i^2
    \end{align*}
    (or, if $\nabla u(x)=0$, the inequality holds with $1$ instead of $\lvert \nabla u(x)\rvert$),
    where the $\varepsilon_i$ are the same as in Lemma \ref{lem: image-u-close-to-ball}, and satisfy $\varepsilon_i\to 0$.
    Now note that
    \begin{align*}
        &\{z\in B_{r_i}(x)\vert \lvert u_{k_i}(z)-p\rvert>5\lvert\nabla u(x)\rvert r_i\}\\\subset&\{z\in B_{r_i}(x)\vert \lvert u_{k_i}(z)-u(z)\rvert>2\lvert\nabla u(x)\rvert r_i\text{ and } \lvert u(z)-p\rvert\leq 3\lvert\nabla u(x)\rvert r_i\}\\
        &\cup\{z\in B_{r_i}(x)\vert  \lvert u(z)-p\rvert>3\lvert\nabla u(x)\rvert r_i\}.
    \end{align*}
    Therefore
    \begin{align*}
        &\nu_{k_i}(\{z\in B_{r_i}(x)\vert \lvert u_{k_i}(z)-p\rvert> 5\lvert\nabla u(x)\rvert r_i\})\\\leq &\nu_{k_i}(\{z\in B_{r_i}(x)\vert\lvert u_{k_i}(z)-u(z)\rvert>2\lvert\nabla u(x)\rvert r_i\})
        +\nu_{k_i}(\{z\in B_{r_i}(x)\vert\lvert u(z)-p\rvert>3\lvert\nabla u(x)\rvert r_i\})\\
        \leq&r_i^2(r_i+2\varepsilon_i).
    \end{align*}
    Thus, for any $ k_i\geq \overline k_i$, we have
    \begin{align}\label{eq: mass-vanishing-infty}
        &(r_i^{-1}(\cdot-p))_\ast\textbf{v}_{ k_i, r_i}(\mathbb{R}^Q\smallsetminus B_{5\lvert\nabla u(x)\rvert}(0))\\
        \nonumber=&r_i^{-2}\nu_{ k_i}(\{z\in B_{r_i}(x)\vert \lvert u_{ k_i}(z)-p\rvert> 5\lvert\nabla u(x)\rvert r_i\})\\
        \nonumber\leq&r_i+2\varepsilon_i,
    \end{align}
    and--- if $\overline k_i$ is large enough---
    \begin{align*}
        (r_i^{-1}(\cdot-p))_\ast\textbf{v}_{ k_i, r_i}( B_{5\lvert\nabla u(x)\rvert}(0))&\leq r_i^{-2} \mu_{ k_i}(B_{5\lvert\nabla u(x)\rvert r_i}(p))\\ &\leq 2\frac{\lvert \textbf{v}\rvert(B_{6\lvert\nabla u(x)\rvert r_i}(p))}{r_i^2}\leq 72 C\lvert \nabla u(x)\rvert^2.
    \end{align*}
    If $\nabla u(x)=0$, the same argument applies for $1$ instead of $\lvert\nabla u(x)\rvert$.
    Hence, for any such choice of $ k_i$, we have that the sequence $\textbf{v}'_{i}:=(r_i^{-1}(\cdot-p))_\ast\textbf{v}_{ k_i, r_i}$ is tight and converges up to subsequences to some varifold $\textbf{v}'_\infty$ on $\mathbb{R}^Q$. This proves $(1)$.
    
    Next, let $\{q_j\}_{j\in \mathbb{N}}$ and $\{s_j\}_{j\in \mathbb{N}}$ be enumerations of $\mathbb{Q}^Q$ and $\mathbb{Q}_{>0}$ respectively.
    Then for any $i\in \mathbb{N}$, we can choose $\overline k_i$ sufficiently large so that for any $k_i\geq \overline k_i$, for any $j\leq i$
    \begin{align*}
        (r_i^{-1}(\cdot-p))_\ast \textbf{v}_{k_i, r_i}(B_{s_j}(q_j))&=r_i^{-2}\textbf{v}_{k_i, r_i}(B_{r_i s_j}(p+r_iq_j))\\&\leq 2r_i^{-2}\textbf{v}(B_{r_i s_j}(p+r_iq_j))\leq 2Cs_j^2.
    \end{align*}
    Therefore, choosing $\overline k_i$ sufficiently large, we can ensure that $(2)$ holds for all rational $q,s$, hence for all $q,s$.
    
    It is clear that condition $(3)$ can be satisfied by choosing $\overline k_i$ to be large enough.
    
    Next, \eqref{eq: mass-vanishing-infty} (or the analogous statement for $\nabla u(x)=0$) implies that $\textbf{v}_\infty'$ has compact support. Moreover, since the traces $u_k\vert_{\partial B_{r_i}}$ converge to $u\vert_{\partial B_{r_i}}$, choosing again $\overline k_i$ large enough we may assume that for $k_i\geq \overline k_i$, the curves $r_i^{-1}(u_{ k_i}(\partial B_{r_i})-p)$ converge in the Hausdorff distance to
    \begin{align*}
        \lim_{i\to\infty}(r_i^{-1}(u(\partial B_{r_i})-p))=\{du(x)[y]\vert y\in\partial B_1(0)\}
    \end{align*}
    (here the equality follows from \eqref{eq: condition-radii-1}). Hence, one can repeat the argument in the proof of Lemma \ref{lem: convergence_of_varifolds}\footnote{Instead of working with the second fundamental form $\sff_\mathcal{M}$ with respect to $\mathcal{M}$, here one works with the second fundamental form $\sff_{\mathbb{R}^Q}$ with respect to $\mathbb{R}^Q$, and uses the fact that $\sff_{\mathbb{R}^Q}^u=\sff_{\mathcal{M}}^u+\sff_{\mathcal{M}\subset\mathbb{R}^Q}$.} for the functions $r_i^{-1}(u_{{k}_i}(r_i\cdot)-p): B_1(0)\to\mathbb{R}^Q$ (for any choice $ k_i\geq \overline k_i$) and variations supported away from $\mathcal{C}$. Thanks to the rescaling properties of the functional $E_{H,\sigma}$ (Lemma \ref{lem: critical-rescaled-functional}), one sees that choosing $k_i$ larger if necessary (such that $r_i^{-\frac{5}{2}}\theta_{ k_i}\to 0$), it is possible to ensure that the limit varifold $\textbf{v}_\infty'$ is stationary away from $\mathcal{C}$. This completes the proof of $(4)$.
    
    Finally, the first equality in $(5)$ follows from the fact that the sequence $\{\textbf{v}_i'\}_{i\in \mathbb{N}}$ is tight, the second one follows from the definition of $\textbf{v}_{i}'$, while the last one can be achieved by choosing $\overline k_i$ so large that for any $k_i\geq \overline k_i$ there holds
    \begin{align*}
        \lvert\nu_{k}(B_{r_i})-\nu(B_{r_i})\rvert\leq o(r_i^2).
    \end{align*}
    This concludes the proof of the claim.
    
    Property (4) and the convex hull property of stationary varifolds (see \cite{SimonGMT}, Theorem 19.2) imply that $\textbf{v}'_\infty$ is supported in $\operatorname{co}(\mathcal{C})$ (the convex hull of $\mathcal{C}$).
    If the rank of $\nabla u(x)$ is less than $2$, then $\operatorname{co}(\mathcal{C})$ is a point or a segment, hence it can be covered by $O(s^{-1})$ balls of radius $s$. Therefore by property (2) we have that $\lvert \textbf{v}'_\infty\rvert(\mathcal{C})=0$, so the Lemma holds in this case. If the rank of $\nabla u(x)$ is 2, then by the constancy Theorem (Theorem 41.1 in \cite{SimonGMT}), we have that $\textbf{v}_\infty'=N\mathcal{H}^2\vert_{\operatorname{co}(\mathcal{C})}$ for some constant $N$. As $\mathcal{H}^2(\operatorname{co}(\mathcal{C}))=\pi\lvert \partial_1 u\wedge\partial_2 u\rvert(0)$,  $N$ is bounded by $2C_{\mathcal{M},\lvert \textbf{v}\rvert(\mathcal{M})}$ because of property (2) (applied to a ball $B_s(0)$ contained in $\operatorname{co}(\mathcal{C})$), and $N\geq 1$ by the convergence of $r_i^{-1}(u_{ k_i}(r_i\cdot)-p)$ to $\nabla u(x)$ on $\partial B_1(0)$ and property (5). Therefore it will be enough to show that $N\in \mathbb{N}$.
    To this end, we may assume that $\operatorname{co}(\mathcal{C})$ is contained in $\mathbb{R}^2\times\{0\}$. Let $\alpha>0$ be such that $\mathcal{C}$ encloses a ball $B_{2\alpha}$ in $\mathbb{R}^2\times\{0\}$. One can show that for any $\overline{X}\in C_c^\infty(B_\alpha, \mathbb{R}^2)$,
\begin{align*}
    \left\lvert\int_{B_{r_i}}\div (\overline{X})J_i\,d\operatorname{vol}_{\tilde u_i}\right\rvert\leq \delta_i\lVert d\overline X\rVert_{L^\infty},
\end{align*}
where $\tilde u_i=r_i^{-1}(u_{k_i}-p)$, $J_i$ is the Jacobian of the map $\pi\circ \tilde u_i$, $\pi$ is the orthogonal projection from $\mathbb{R}^Q$ to $\mathbb{R}^2\times\{0\}$ and $\delta_i\to 0$.
    In fact, one can follow the argument in the proof of Theorem 5.3 in \cite{Pigati-FB}, which is based on the control of the first variation of the the rescaled functionals $E_{r_iH, r_i^{-1}\sigma_{k_i}}$ and on approximating vector fields $X\in C_c(B_\alpha\times \mathbb{R}^{Q-2})$ by convolution with smooth functions supported in $B_{\tau_i}$, with $\tau_i=r_i^{-1}\sigma_{k_i}$. This allows us to obtain estimates in terms of $dX$ (instead of the more immediate estimates in terms of $(\nabla^{T\mathcal{M}_{p, r_i}})^2 X$). The only difference in our setting is the appearance of the additional term involving $H$ in \eqref{eq: control-first-variation}. This term can be estimated as follows\footnote{Given a vector field $X\in C_c(B_\alpha\times\mathbb{R}^{Q-2})$, one defines $Y_i$ as the projection of $X\ast \rho_{\tau_i}$ on $r_i^{-1}(\mathcal{M}-p)$; one considers $X=\chi \overline{X}$, with $\chi$ a cut-off equal to $1$ near $\overline{B_\alpha^{\mathbb{R}^Q}}$ See the proof of Theorem 5.3 in \cite{Pigati-FB} for details.}:
    \begin{align*}
        \left\lvert r_iH\int_{B_{r_i}}(d\operatorname{vol}_g)_{\tilde u_i}(Y_i(\tilde u_i),(\tilde u_i)_x,(\tilde u_i)_y)dx\wedge dy \right\rvert\leq r_i H\lVert Y_i\rVert_{L^\infty}\frac{\nu_{k_i}(B_{r_i})}{r_i^2},
    \end{align*}
which can be controlled by $\delta_i\lVert d X\rVert_{L^\infty}$ for a sequence $\delta_i\to 0$.

We can then conclude the proof as in \cite{Pigati-FB}: for any $i\in \mathbb{N}$ set $\textbf{v}_i'':=\pi_\ast\textbf{v}_i'$. By the area formula, for any $i$ there exists an integer valued multiplicity function $N_i$ such that
\begin{align*}
    \left\lvert \int_{B_\alpha}N_i\operatorname{div}(\overline{X})d\mathcal{L}^2\right\rvert=\left\lvert\int_{B_{r_i}}\operatorname{div}(\overline{X})J_i\,d\operatorname{vol}_{\tilde u_i}\right\rvert\leq \delta_i\lVert d\overline{X}\rVert_{L^{\infty}}.
\end{align*}
By Allard's constancy Lemma ((4) in \cite{allard-integrability}; see Lemma A.7 in \cite{PR-Regularity} for a version of the result that applies directly to this setting), there is a constant $\overline N_i$ such that $\lVert N_i-\overline N_i\rVert_{L^{1,\infty}}$ goes to zero as $i\to \infty$. As $N_i$ is integer valued, it follows that $\operatorname{dist}(\overline N_i, \mathbb{N})\to 0$. Since $\textbf{v}_i''$ converges weakly to $\textbf{v}_\infty'$, we conclude that
\begin{align}\label{eq: definition-N-limit}
\pi (\alpha/2)^2N=\lim_{i\to\infty}\int_{B_{\alpha/2}}N_i d\mathcal{L}^2=\lim_{i\to\infty}\pi(\alpha/2)^2\overline{N_i},
\end{align}
which implies that $N\in \mathbb{N}$.

\end{proof}

\begin{lemma}\label{lem: image-u-close-to-ball}
    For almost any $x\in \Sigma$, for a sequence of radii $r_i\to 0$ as in \eqref{eq: condition-radii-1}, there holds
    \begin{align*}
        \nu(\{z\in B_{r_i}(x)\vert \lvert u(z)-p\rvert>2\lvert\nabla u(x)\rvert r_i\})\leq \varepsilon_ir_i^2,
    \end{align*}
    with $p=u(x)$ and $\varepsilon_i\to 0$, provided $\nabla u(x)\neq0$. If $\nabla u(x)=0$, we have
    \begin{align*}
        \nu(\{z\in B_{r_i}(x)\vert \lvert u(z)-p\rvert>2 r_i\})\leq \varepsilon_ir_i^2.
    \end{align*}
\end{lemma}
\begin{proof}
    By Theorem 6.2 in \cite{Evans-Gariepy} and Jensen's inequality, for almost any $x\in \Sigma$, there holds
    \begin{align*}
        \delta_i:=r_i^{-2}\int_{B_{r_i}(x)}\left\lvert\frac{u(z)-p}{r_i}-\nabla u(x)\cdot \frac{z-x}{r_i}\right\rvert\to 0.
    \end{align*}
    Therefore
    \begin{align*}
        \lvert\{z\in B_{r_i}(x)\vert\lvert u(z)-p-\nabla u(x)\cdot (z-x)\rvert>\varepsilon r_i\}\rvert\leq \frac{\delta_ir_i^2}{\varepsilon}.
    \end{align*}
    Note that
    \begin{align*}
        A_{x,r_i}:= &\left\{z\in B_{r_i}(x)\vert \lvert u(z)-p\rvert\geq 2\lvert\nabla u(x)\rvert r_i\right\}\\ &\subset \left\{z\in B_{r_i}(x)\vert\lvert u(z)-p-\nabla u(x)\cdot (z-x)\rvert>\lvert\nabla u(x)\rvert r_i\right\}.
    \end{align*}
    Then
    \begin{align}\label{eq: estimate-measure-A}
       \lvert\nabla u(x)\rvert \lvert A_{x, r_i}\rvert\leq \delta_i r_i^2.
    \end{align}
    By \eqref{eq: control-nu-by-du}, it will be enough to show that $r_i^{-2}\int_{{A_{x, r_i}}}\lvert \nabla u\rvert^2\to 0$.
    In fact, if $x$ is a Lebesgue point of $\nabla u$ and $\lvert \nabla u\rvert^2$, we have
    \begin{align*}
        r_i^{-2}\int_{{A_{x, r_i}}}\lvert \nabla u\rvert^2\leq 2r_i^{-2}\int_{{B_{r_i}(x)}}\lvert \lvert \nabla u\rvert-\lvert \nabla u(x)\rvert\rvert^2+2r_i^{-2}\lvert {A_{x, r_i}}\rvert\lvert \nabla u(x)\rvert^2.
    \end{align*}
    The first term on the right hand side tends to zero as $x$ is a Lebesgue point of $du$ and $\lvert du\rvert^2$, the second terms tends to zero by \eqref{eq: estimate-measure-A}.
\end{proof}

\begin{rmk}\label{rmk: rectifiable-varifold}
For any open set $\omega\subset\tilde \Sigma$, let $\mathcal{G}_\omega$ be the subset of Lebesgue points for $du$ where the differential has rank 2. Equip the image $u(\mathcal{G}_\omega)$ with the multiplicity
\begin{align*}
    \theta_\omega(p)=\sum_{x\in \mathcal{G}_\omega\cap u^{-1}(p)}N(x).
\end{align*}
Note that the set $u(\mathcal{G}_\omega)$ 
is rectifiable by Lemma A.2 in \cite{PR-Regularity}. Then the set $u(\mathcal{G}_\omega)$, with the multiplicity function $\theta_\omega$, form an integral varifold in $\mathcal{M}$, which we denote $\textbf{v}_\omega$. Note that by Lemma \ref{lem: N-jacobian} and the coarea formula (see Theorem 11 in \cite{Hajlasz-coarea}), $\lvert \textbf{v}_\omega\rvert=(u\vert_\omega)_\ast\nu_\infty$.
\end{rmk}
The same adaptations as above allow us to prove that the varifold $\textbf{v}_\omega$ is the varifold limit of the varifolds induced by $u_k\vert_\omega$:

\begin{lemma}[Cf. {\cite[Theorem 5.7]{Pigati-FB}}]\label{lem: Pigati-5.7}
    Given an open subset $\omega\Subset\tilde\Sigma$ with $\nu(\partial \omega)=0$, the varifolds induced by $u_k\vert_\omega$ converge to $\textbf{v}_\omega$.
\end{lemma}
Next we note that the proof of Lemma \ref{lem: convergence_of_varifolds} (see in particular \eqref{eq: convergence-first-variation-to-H-term}) applied to vector fields $X$ supported away from $u(\partial\omega)$ implies that
\begin{align*}
    \delta \textbf{v}_\omega (X)=\lim_{k\to\infty} H\int_{\omega}u_k^\ast\alpha_X,
\end{align*}
where $\alpha_X:=\ast X^\flat$.
The limit on the right hand side is computed in the following Lemma.
\begin{lemma}\label{prop: convergence-of-H-term}
    For any domain $\omega\Subset\tilde\Sigma$, if $X$ is supported away from $u(\partial \omega)$ we have
        \begin{align*}
        \lim_{k\to\infty}\int_\omega u_k^\ast \alpha_X=\int_\omega u^\ast \alpha_X.
    \end{align*}
\end{lemma}
\begin{proof}
    Let $\rho$ be a non-negative $C^\infty$ function on $\Sigma$, let $\omega$ be a super-level set of $\rho$ such that $\partial \omega$ is $C^1$.
    Extend $\alpha_X$ to a smooth $2$-form $\overline{\alpha}_X$ on $\mathbb{R}^Q$, and write
    \begin{align*}
        \overline{\alpha}_X=\sum_{i<j}a_{ij}dx_i\wedge dx_j,
    \end{align*}
    where the indices run from $1$ to $Q$, and  where $a_{ij}$ are smooth functions on $\mathbb{R}^Q$.
    We need to show that for any $i,j$
    \begin{align*}
        \lim_{k\to\infty} \int_\omega a_{ij}(u_k)d u^i_k\wedge d u_k^j=\int_\omega a_{ij}(u)du^i\wedge d u^j.
    \end{align*}
    By \eqref{eq: convergence-dist-L1}, we have
    \begin{align*}
        \left\lvert\int_\omega \left(\alpha_{ij}(u_k)-\alpha_{ij}(u)\right)d u_k^i\wedge du_k^j\right\rvert\leq \lVert \alpha_{ij}\rVert_{C^1(\mathcal{M})}\int_\omega\lvert u_k-u\rvert d\nu_k\to 0.
    \end{align*}
    As $X$ (and thus also $\alpha_X$) is supported away from $u(\partial \omega)$,  we have
    \begin{align*}
        \int_\omega \alpha_{ij}(u)du_k^i\wedge du_k^j=-\int_\omega d(\alpha_{ij}\circ u)\wedge u_k^id u_k^j
    \end{align*}
    (this can be proved by approximating $u$ and $u_k$ in $W^{1,2}(\omega)$ by smooth functions, and taking advantage of the fact that $u_k$, $u$ are uniformly bounded in $L^\infty$).
    Now by Cauchy-Schwarz and the fact that $\lVert u_k-u\rVert_{L^\infty}\leq 2\operatorname{diam}(\mathcal{M})$
    \begin{align*}
        &\left\lvert \int_\omega d(\alpha_{ij}\circ u)\wedge(u_k^i-u^i)d u_k^j\right\rvert\\&\leq 2\lVert \alpha_{ij}\rVert_{C^1(\mathcal{M})}\text{diam}^\frac{1}{2}(\mathcal{M})\left(\int_\omega\lvert \nabla u\rvert^2 dx\right)^\frac{1}{2}\left(\int_\omega\lvert u_k-u\rvert d\nu_k\right)^\frac{1}{2}.
    \end{align*}
    The right hand side converges to zero by \eqref{eq: convergence-dist-L1}. Finally,
    \begin{align*}
        \left\lvert \int_\omega d(\alpha_{ij}\circ u)\wedge u^id(u_k^j-u^j)\right\rvert\to0,
    \end{align*}
    by the weak convergence of $\{u_k\}_{k\in \mathbb{N}}$ in $W^{1,2}$.
    Combining the above estimates, we get
    \begin{align*}
        \lim_{k\to\infty}\int_\omega \alpha_{ij}(u_k)du_k^i\wedge du_k^j=-\int_\omega d(\alpha_{ij}\circ u)\wedge u^idu^j=\int_\omega \alpha_{ij}(u)du^i\wedge du^j,
    \end{align*}
    as desired.\\
\end{proof}
Finally, we show that the limit varifolds $\textbf{v}_\omega$, for $\omega\subset\tilde \Sigma$, can be patched together to obtain a parametrized CMC-varifold, whose parametrization is given--- up to reparametrization--- by $u$.

\begin{lemma}[Cf. {\cite[Theorem 5.12]{Pigati-FB}}]\label{lem: 5-12-P-FB}
    There exists a Riemann surface $\Sigma'$ and a quasi-conformal homeomorphism $\varphi: \Sigma'\rightarrow\Sigma$ such that
    $(\Sigma', u\circ\varphi, N\circ\varphi)$ is a parametrized $H$-CMC varifold and for a.e. open $\omega\subset\tilde\Sigma$ with $\nu(\partial \omega)=0$, the varifold $\textbf{v}_\omega$ coincides with the parametrized varifold $(\varphi^{-1}(\omega), u\circ\varphi, N\circ\varphi)$.
\end{lemma}
\begin{proof}
    For a.e. open $\omega \subset \Sigma$, $u_k\vert_{\partial\omega}\to u\vert_{\partial\omega}$ in $C^0$ up to subsequences, and $\partial\omega\cap\mathcal{A}=\emptyset$, where $\mathcal{A}$ is the finite set of atoms of $\nu$. Let $\omega$ be such a domain, and $r>0$ be such that for any $x\in \omega\cap\mathcal{A}$, the balls $B_r(x)$ (w.r.t. the reference metric $g_0$ on $\Sigma$)  satisfy $B_r(x)\Subset \omega$ and $u_k\vert_{\partial B_r(x)}\to u\vert_{\partial B_r(x)}$ in $C^0$ (up to subsequences). Let $\tilde\omega_r=\omega\smallsetminus\bigcup_{x\in \omega\cap\mathcal{A}}B_r(x)$. Repeating the argument of Lemma \ref{lem: convergence_of_varifolds} for vector fields $X$ on $\mathcal{M}$ supported away from $u(\partial\tilde\omega)$, and combining \eqref{eq: convergence-first-variation-to-H-term} with Lemma \ref{prop: convergence-of-H-term}, we obtain that
    \begin{align*}
    \delta \textbf{v}_{\tilde\omega_r} (X)= H\int_{\tilde\omega_r}u^\ast\alpha_X,
\end{align*}
where $\alpha_X:=\ast X^\flat$, for any vector field $X$ on $\mathcal{M}$ supported away from $u(\partial\tilde\omega_r)$. 

We claim that \eqref{eq: first-variation-cmc-local} remains true for $\omega$ in place of $\tilde \omega$. In fact, note that for $x\in \omega\cap\mathcal{A}$,
\begin{align*}
    \lim_{r\to 0}\operatorname{diam} u(\partial B_r(x))=0
\end{align*}
(see \cite{PR-Regularity}, Lemma A.3). Therefore there exist a finite subset $F$ of $\mathcal{M}$ such that for any $X$ supported away from the points in $F$,
\begin{align}\label{eq: first-variation-cmc-local}
    \delta \textbf{v}_{\omega} (X)= H\int_{\omega}u^\ast\alpha_X.
\end{align}
Note also that since the varifolds induced by $u_k$ converge to $\textbf{v}$, there holds $\textbf{v}\geq \textbf{v}_\omega$, therefore by \eqref{eq: control-density} there holds $\lvert\textbf{v}_\omega\rvert(B_s(p))\leq Cs^2$ for $p\in F$. For any $j\in \mathbb{N}$, $p\in F$, let $\varphi_p^j$ be a cut-off function on $\mathcal{M}$ with $\varphi_p^j=0$ on $B_{\frac{1}{2j}}(p)$, $\varphi_p^j=1$ outside $B_{\frac{1}{j}}(p)$ and $\lVert d\varphi_p^j\rVert_{L^\infty}\leq 3j$ (so that $\lim_{j\to\infty}\int_{\mathcal{M}}\lvert\nabla^\mathcal{M}\varphi_p^j\rvert d\lvert \textbf{v}_\omega\rvert=0$).
Then, given a smooth vector field $X$ on $\mathcal{M}$ supported away from $u(\partial\omega)$, there holds
\begin{align*}
    &\left\lvert\langle\textbf{v}_\omega,\operatorname{div}(X)\rangle-H\int_\omega u^\ast\alpha_X\right\rvert\\ =&\lim_{j\to\infty}\left\lvert\langle\textbf{v}_\omega,\operatorname{div}(X\Pi_{p\in F}\varphi_j^p)\rangle-H\int_\omega \vol_g(X\Pi_{p\in F}\varphi_j^p\circ u,\partial_x u, \partial_y u) dx\wedge dy\right\rvert=0,
\end{align*}
where the last equality follows from \eqref{eq: first-variation-cmc-local}. This completes the proof of the claim.

Next we claim that, a.e. in any local conformal chart $h: U\to U'$ for $\Sigma$, with $U\subset\tilde\Sigma$, there holds
\begin{align}\label{eq: N-controls-deformation}
    N \lvert\partial_xu\wedge\partial_y u\rvert\geq \frac{1}{2}\lvert du\rvert^2.
\end{align}
Indeed, for any smooth open $V\Subset U$, there holds
\begin{align*}
    \int_V\frac{1}{2}\lvert du\rvert^2d\mathcal{L}^2\leq\liminf_{k\to\infty}\nu_k(V)\leq \nu(\overline{V})=\int_{\overline{V}}N\lvert\partial_x u\wedge\partial_y u\rvert d\mathcal{L}^2, 
\end{align*}
where the last equality follows from Lemma \ref{lem: N-jacobian}. As $V$ is arbitrary, the claim is verified.
We can then repeat verbatim the proof of Theorem 5.12 in \cite{Pigati-FB} (which is based on Theorem 4.24 in \cite{Imayoshi-Taniguchi}) to show that there exists a conformal structure on $\Sigma$ such that--- calling $\Sigma'$ a copy of $\Sigma$ with this structure--- the identity map $\varphi:\Sigma'\to\Sigma$ is $K$-quasi-conformal (with $K$ controlled in terms of $\lVert N\rVert_{L^\infty}$), and  $u\circ\varphi$ is conformal as a map from $\Sigma'$. Then $(\Sigma', u\circ \varphi, N\circ\varphi)$ is a parametrized $H$-CMC varifold satisfying the desired properties.
\end{proof}

\subsection{Degeneration of the conformal structure and bubbling}\label{subsec: degeneration}
In this section we study the possible degeneration of the conformal class along a sequence of almost-critical points $\{u_k\}_{k\in \mathbb{N}}$ and the possibility of concentration points of the energy ("bubbles"). In particular, we justify why in the previous part of this section there was no loss of generality in assuming that all the elements of the sequence $\{u_k\}_{k\in \mathbb{N}}$ induce the same conformal class on $\Sigma$.

In fact, one can repeat verbatim the arguments of Section 6 in \cite{Pigati-FB}, provided we can prove a result corresponding to Proposition 6.1 of \cite{Pigati-FB} in our setting.
In the following, we first give a proof of such a result (Lemma \ref{prop: energy estimate for two boundary components}), and then give a brief overview of the arguments of Section 6 in \cite{Pigati-FB}.\\

\begin{lemma}
    \label{prop: energy estimate for two boundary components}
    Given a sequence of open domains $U_k\subset\Sigma$ whose boundaries $\de U_k$ are contained in the support of two compact closed curves $\alpha_{k,i},i=1,2$, we have either 
    \begin{enumerate}
        \item $\limsup_{k\to\infty}\nu_k(U_k)\geq c_Q$; or
        \item $\limsup_{k\to\infty}\nu_k(U_k)\leq\delta(\limsup_{k\to\infty}\max\{d_{k,1},d_{k,2}\},C)$,
    \end{enumerate}
    where $d_{k,i}$ is the diameter of $u_k(\alpha_{k,i})$, $C$ depends only on $\overline{A}$ (the upper bound of $\text{Area}(u_k)$), and $\delta: (0,\infty)^2\to(0,\infty)$ is a function with $\lim_{s\to 0}\delta(s,t)=0$ for any $t\in (0,\infty)$.
\end{lemma}
\begin{proof}
    Note that for $i=1,2$, $u_k(\alpha_{k,i})$ is contained in $ \overline{B_{d_{k,i}}(p_{k,i})}$ for some points $p_{k,i}\in \mathcal{M}$. Up to a subsequence we may assume that $p_{k,i}\to p_i$, $d_{k,i}\to d_i$ for some $p_i\in \mathcal{M}$, $d_i\geq 0$.
    Up to subsequences, we may assume that the varifolds $\textbf{v}_{u_k\vert_{U_k}}$ converge to some varifold $\textbf{v}_\infty$.
    Then, repeating the argument in the proof of Lemma \ref{lem: convergence_of_varifolds} with variations supported in the complement of $\overline{B_{d_1}(p_1)}\cup \overline{B_{d_2}(p_2)}$, we obtain that the limiting varifold $\textbf{v}_\infty$ has generalized mean curvature bounded (in $L^\infty$) in the complement of $\overline{B_{d_1}(p_1)}\cup \overline{B_{d_2}(p_2)}$, has mass at most $Cr^2$ on balls of radius $r$ (as $\lvert \textbf{v}_\infty\rvert\leq \vert \textbf{v}\rvert$, and $\textbf{v}$ satisfies \eqref{eq: control-density}), and has density bounded below by a constant $c>0$ a. e. on its support (this can be shown as in the proof of Lemma \ref{lem: concentration property}, or Proposition 5.1 in \cite{Pigati-FB}, by an application of Proposition  \ref{prop: almost monotonicity for large radius}).
    As a result, the claim follows from Lemma \ref{lem: degeneration-conf} below.
\end{proof}

We first discuss how to remove the assumption of the fixed conformal structure, and show that no energy is lost in the possible degeneration of the conformal structure, assuming for the moment that there is no concentration point for the energy; later we will discuss how to deal with them.

First, suppose $\Sigma$ is a sphere. Since $\mathbb{S}^2$ has a unique conformal structure up to diffeomorphism, we may assume, after precomposing each $u_k$ with a diffeomorphism of $\Sigma$, that all $u_k$ induce the same conformal structure on $\Sigma$.

Next let's consider the case where $\Sigma$ is a torus. Then $(\Sigma, g_{u_k})$ is conformal to $\mathbb{C}/(\mathbb{Z}+\mathbb{Z}\lambda_k)$, for some complex $\lambda_k$ with $\lvert\lambda_k\rvert\geq 1, \lvert\Re(\lambda_k)\rvert\leq\frac{1}{2}$. Up to precomposing with a diffeomorphism of $\Sigma$, we can actually assume that $(\Sigma, g_{u_k})=\mathbb{C}/(\mathbb{Z}+\mathbb{Z}\lambda_k)$.
Up to subsequences, we may assume that the sequence of $\ell_k:=\lvert \lambda_k\rvert$ converges to some $\ell_\infty\in [1,\infty]$. If $\ell_\infty<\infty$ we may assume that $\lambda_k\to \lambda_\infty$ for some complex $\lambda_\infty$. Set $\Sigma_\infty=\mathbb{C}/(\mathbb{Z}+\mathbb{Z}\lambda_\infty)$ and denote by $g_\infty$ the natural flat metric on $\Sigma_\infty$. Then we can find diffeomorphisms $\varphi_k:\Sigma_\infty\to (\Sigma, g_{u_k})$ such that the pullback of the conformal class $[g_{u_k}]$ converges smoothly to $[g_\infty]$ on $\Sigma_\infty$. One can then define $\tilde \nu_k$ as $\frac{1}{2}\lvert d(u_k\circ \varphi_k)\rvert^2d\vol_{\Sigma_\infty}$ and repeat all the proofs of this section (with the exception of Lemma \ref{lem: N-jacobian}) for $u_k\circ\varphi_k$ instead of $u_k$. As the weak limit of $\{\tilde \nu_k\}_{k\in \mathbb{N}}$ coincides with the weak limit of 
$\{\nu_k\}_{k\in \mathbb{N}}$, the statements in this section remain valid in this case.
For the proof of Lemma \ref{lem: N-jacobian}, we used the conformality of the maps $u_k$. For this result, we can argue locally: one can precompose $u_k$ with  conformal maps $\varphi_k:B_1^2\to (\Sigma, g_{u_k})$ which are diffeomorphisms with their image and converge smoothly to $\varphi_\infty$, the inverse  of a conformal chart for $(\Sigma_\infty, g_\infty)$.
As $B_1(0)$ only has one conformal class, all maps $u_k\circ\varphi_k$ induce the same conformal class on $B_1(0)$. One can then prove the statement of Lemma \ref{lem: N-jacobian} for the limit of the maps $\{u_k\circ\varphi_k\}_{k\in \mathbb{N}}$.

If $\ell_\infty=\infty$, since $\lvert\Re(\lambda_k)\rvert\leq \frac{1}{2}$, we can regard $(\Sigma, g_{u_k})$ as $\mathbb{S}^1\times\ell_k\mathbb{S}^1$, with a conformal structure converging to the flat one. The circle $\ell_k \mathbb{S}^1$ can be subdivided into arcs $I_{k,1},...I_{k, N_k}$ of length in $[L,2L]$ (for an appropriate $L>0$). Up to subsequences, the energy concentrates around finitely many unions of such arcs, denoted $J_{k,1},...J_{k,h}$, of finite length and satisfying $\lim_{k\to\infty}\operatorname{dist}(J_{k,i}, J_{k,j})=\infty$, for $i\neq j$. More precisely, Lemma \ref{prop: energy estimate for two boundary components} implies that
\begin{align*}
    \lim_{R\to\infty}\limsup_{k\to\infty} \int_{\mathbb{S}^1\times(\ell_k\mathbb{S}^1\smallsetminus\bigcup_{j=1}^h RJ_{k,j})}\lvert du_k\rvert^2=0.
\end{align*}
Here $RJ_{k,j}$ is the interval with the same center as $J_{k,j}$, dilated by a factor $R$.
For $j\in \{1,...,h\}$, we obtain weak limits $u_{\infty,j}: \mathbb{S}^1\times\mathbb{R}\to\mathcal{M}$ of the maps $u_k\vert_{\mathbb{S}^1\times RJ_{k,j}}$,
where $RJ_{k,j}$ is identified with an interval of the same length centered at zero. Arguing as above, we can repeat the arguments of this section for any such sequence of maps.
Note that $\mathbb{S}^1\times\mathbb{R}$ is conformal
to the sphere minus two points, thus we can view the domain as the sphere. The resulting varifold is parametrized $H$-CMC across the two points, as can be seen by multiplying any variation by a function that vanishes around the points, as in the proof of Lemma \ref{lem: 5-12-P-FB}.

If $\chi(\Sigma)<0$, we might take $g_{u_k}$ to be a metric of constant Gaussian curvature $-1$. By Bers' theorem, we can decompose $(\Sigma, g_{u_k})$ into hyperbolic pairs of pants, with length of their boundaries uniformly bounded by a constant depending only on the genus of $\Sigma$ (see Theorem IV.3.7 in \cite{Hummel}). Let $\{\beta_{k,i}\}_{i=1}^p$ denote the geodesics which bound the pairs of pants. For any $i\in \{1,...,p\}$, up to subsequences, the lengths $\ell(\beta_{k,i})$ converge to some value in $[0,C]$, as $k\to\infty$.
Let $I$ be the index set of  geodesics for which $\ell(\beta_{k,i})\to0$ as $k\to\infty$. By
Proposition IV.5.1 in \cite{Hummel}, there exists a (possibly disconnected) limit surface $\Sigma_\infty$, which equals a closed Riemannian surface minus finitely many points, and diffeomorphisms $\psi_k: \Sigma_\infty\to\Sigma\smallsetminus\bigcup_{i\in I}\beta_{k,i}$ such that the metric $\psi_k^\ast g_k$ converges locally to the hyperbolic structure of $\Sigma_\infty$. Arguing as above, we can then repeat the arguments of this section for the sequence of maps $u_k\circ\psi_k$. Apart from concentration points (discussed later), the area might concentrate in collar neighborhood around the curves $\{\beta_{k,i}\}_{i\in I}$. These neighborhoods can be conformally identified with cylinders $\mathbb{S}^1\times [0, L_{k,i}]$,
with $L_{k,i} \to\infty$ as $k\to\infty$, and one can recover the missing part of the area as in the degenerating cylinder case. As above, one verifies that the resulting object extends to a parametrized $H$-CMC varifold across the punctures.\\

Finally, we consider the case of concentration of energy (``bubbling"). The issue is local, and as there can be only finitely many concentration points, it is enough to study the case of a single one. Let $p$ be a concentration point of the energy. After precomposing the maps $u_k$ with conformal charts $\varphi_k:B_1(0)\to\Sigma$ centered at $p$, we have the tight convergence
\begin{align*}
    \nu_k'=\frac{1}{2}\lvert d (u_k\circ\varphi_k)\rvert^2\vol_{g_0}\rightharpoonup m\mathcal{L}^2+\alpha\delta_0,
\end{align*}
of measures on $B_1(0)$. Considering charts to a smaller neighborhood of $p$ if necessary, we can ensure that $\int_{B_1(0)}m<\frac{c_Q}{2}$, while Lemma \ref{lem: limit of measures} implies that $\alpha\geq c_Q$. For any $k$, let $B_{r_k}(x_k)$ be a ball of minimal radius such that
\begin{align}\label{eq:choice-radii-bubbling}
    \int_{B_{r_k}(x_k)\cap B_1(0)}\frac{1}{2}\lvert d(u_k\circ\varphi_k)\rvert^2\geq \alpha-\frac{c_Q}{2},
\end{align}
so that the integral is exactly $\alpha-\frac{c_Q}{2}$ and $r_k\to 0$, $x_k\to 0$. Now Lemma \ref{prop: energy estimate for two boundary components} can be used to show that
\begin{align*}
    \lim_{R\to\infty}\limsup_{k\to\infty} \nu_k'((B_{R^{-1}}(x_k)\smallsetminus B_{Rr_k}(x_k))\cap B_1(0))=0.
\end{align*}
    Thus setting $\Psi_k:=u_k\circ\varphi_k(x_k+r_k\cdot)$ one has that the corresponding measures $\frac{1}{2}\lvert d\Psi_k\rvert^2$ converge to a measure $\nu$ of total mass $\alpha$. Along this convergence, there might be further concentration points, but the minimality condition in \eqref{eq:choice-radii-bubbling} can be used to show that their masses are at most $\alpha-\frac{c_Q}{2}$, so that the defect energy decomposes into a finite sum of bubble energies, and the neck energies vanish. The limiting maps are defined on larger and larger portions of $\mathbb{R}^2$, the limit of which can be regarded as maps from spheres minus a point. As the mass ratios of the corresponding varifolds are bounded from above, the limiting object can also be viewed as a parametrized $H$-CMC varifold whose domain is a sphere (by an argument similar to the one in the proof of Lemma \ref{lem: 5-12-P-FB}).\par

Combining the results presented in this section, we obtain the following result.
\begin{thm}\label{thm: convergence to H-pv}
    Let $H\geq 0$. Let $\{u_k\}_{k\in\mathbb{N}}$ be a sequence as in Theorem \ref{thm: main-theorem}. Then there exists a subsequence such that the varifolds induced by $u_k$ converge to a parametrized $H$-CMC varifold $\textbf{v}$ in $\mathcal{M}$. The connected components $\Sigma_i$ of its domain satisfy $\sum_ig(\Sigma_i)\leq g(\Sigma)$.
\end{thm}
\begin{proof}
    The arguments in this subsection imply that--- up to a subsequence--- for each $k\in \mathbb{N}$ there exist disjoint open sets $U_{k,1},...,U_{k,N}\subset\Sigma$ (for some fixed $N$) such that the varifolds induced by $u_k\vert_{U_{k,i}}$ converge to a parametrized $H$-CMC varifold, and $\nu_k(\Sigma\setminus\cup_{i=1}^N U_{k,i})\to 0$ as $k\to\infty$.
    More precisely, for any 
    $i\in \{1,...,N\}$ there exist a compact Riemann surface $\Sigma_i$, a finite
set $P_i\subset \Sigma_i$, open sets $\Omega_{k,i}\subset \Sigma_i\setminus P_i$ exhausting
$\Sigma_i\setminus P_i$, and diffeomorphisms $\phi_{k,i}:\Omega_{k,i}\to U_{k,i}$ such that
$u_k\circ\phi_{k,i}$ converges locally to a parametrized $H$-CMC varifold $(\Sigma_i\smallsetminus P_i, u_\infty^i, N_i)$ (in the sense described in Lemma \ref{lem: Pigati-5.7} and Lemma \ref{lem: 5-12-P-FB}). Arguing as in the proof of Lemma~\ref{lem: 5-12-P-FB}, we see that this construction extends to a parametrized $H$-CMC varifold $(\Sigma_i, u_\infty^i, N_i)$.
Since $\nu_k(\Sigma\setminus\cup_{i=1}^N U_{k,i})\to 0$, the varifolds induced by $u_k$ converge in the sense of varifolds to the sum of the parametrized varifolds $(\Sigma_i, u_\infty^i, N_i)$.
Finally, the surfaces $\Sigma_i$ arise from $\Sigma$ by pinching (cutting along disjoint
simple closed curves) and by adding bubble spheres; thus the genus cannot increase and
$\sum_i g(\Sigma_i)\le g(\Sigma)$.
\end{proof}

\section{Regularity and multiplicity one}\label{sec: regularity}
In this section we show that the maps obtained as weak limits in the convergence process described above are in fact smooth branched CMC immersions.
The key step in the argument is the proof of the ``multiplicity one property" in this setting, i.e. we need to show that the multiplicity $N$ arising in the limiting parametrized varifold is in fact equal to one.
\begin{rmk}\label{rmk: simplifications}
    Our proof follows the general strategy of \cite{PR-Multiplicity-One}, with some modifications. In \cite{PR-Multiplicity-One}, Pigati and Rivière showed that any parametrized varifold arising as a limit of maps satisfying the assumptions of Theorem \ref{thm: main-theorem} (with $H=0$) has multiplicity $N$ equal to 1.
    
    Instead, we show that for a suitably chosen blow-up sequence around ``good points", multiplicity one is preserved in the limit. As the multiplicity $N$ can be detected by blow-ups, this is enough to prove the result. Since at generic points the limit of a blow-up sequence is a linear map into a 2-dimensional plane, at a few stages our argument is simpler than the original one. We remark that when $H=0$, our argument works in any codimension, therefore it provides a slightly simpler proof of the original multiplicity one Theorem of \cite{PR-Multiplicity-One} in any codimension.
    Like the proof of \cite{PR-Multiplicity-One}, however, our argument relies on the regularity result of \cite{PR-Regularity}.
    
    While Pigati and Rivière worked with a sequence of critical points for the functionals $E_{\sigma_k}$, we instead work with almost critical points, so that the results of this section can be applied directly to the maps obtained in the previous section.
\end{rmk}
Let $\{u_k\}_{k\in \mathbb{N}}$ be as in the conclusion of Theorem \ref{thm: convergence to H-pv}, assume that $u_k\rightharpoonup u$ weakly in $W^{1,2}$ and, for $\nu'_k:=\frac{1}{2}\lvert\nabla u_k\rvert^2$, assume that $\nu_k'\rightharpoonup \nu$ weakly-$\ast$.\\
Let $\mathcal{G}'$ denote the set of points such that
\begin{itemize}
    \item $x$ is not an atom of $\nu$;
    \item $x$ is a Lebesgue point for $u$, $du$ and $\lvert du\rvert^2$ for a conformal chart centered at $x$;
    \item $du(x)$ has rank 2;
    \item $u_k(x)\to u(x)$;
    \item 
    \begin{align}\label{eq: localized control on second fundamental form}
        \sigma_k^4\log\sigma_k^{-1}\int_{B_r(x)}\vert\sff^{u_k}\vert^4\,d\vol_{g_{u_k}}\leq\delta_k\int_{B_r(x)}\,d\vol_{g_{u_k}},\text{ for all }r\in(0,1/2),\text{ for a sequence }\delta_k\rightarrow0.
    \end{align}
\end{itemize}
Note that the set $\mathcal{G}'$ has full $\nu$ measure in $\Sigma$. For condition \eqref{eq: localized control on second fundamental form}, the proof can be found in \cite[Proof of Theorem 6.4]{PR-Multiplicity-One}.
We will show that for $x\in \mathcal{G}'$, $N(x)=1$. Repeating the argument iteratively for the reparametrizations around the possible bubbles (which converge to maps $\mathbb{S}^2\to\mathcal{M}$, as discussed at the end of Subsection \ref{subsec: degeneration}), one obtains the desired result for all the components of the limiting parametrized varifold.\par
Let $x$ be a point in $\mathcal{G}'$ and let $\Phi:U\to U'$ (with $B_1(0)\subset U'\subset\mathbb{C})$ be a conformal chart for a neighborhood of $x$, centered at $x$ and containing no atoms of $\nu$. In what follows, we identify \(U\) with \(U'\) through the chart \(\Phi\). We still denote \(u_k\circ \Phi^{-1}\) and \(u\circ \Phi^{-1}\) by \(u_k\) and \(u\), respectively, and, by a slight abuse of notation, we continue to denote the coordinate point \(\Phi(x)=0\) by \(x\). Let $p:=u(x)$.\\
For any $k\in \mathbb{N}$, for any $r>0$ sufficiently small, set
\begin{align*}
    v_{r,k}=r^{-1}(u_k(x+r\cdot)-p): B_1(0)\rightarrow\mathcal{M}_{p,r},\text{ where }\mathcal{M}_{p, r}=r^{-1}(\mathcal{M}-p).
\end{align*}
Choosing suitable sequences $r_i\to 0$ and $k_i\to \infty$, we can ensure that
$v_{r_i, k_i}$ converges a linear map:  the differential of $u$ at $x$. This is proved in the following Lemma.

\begin{lemma}\label{lem: convergence-blow-up}
    Let $\{r_i\}_{i\in \mathbb{N}}$ be a sequence such that $r_i\to 0$. Let $p:=u(x)$.
    Then the sequence of functions
    \begin{align*}
        u^{r_i}: B_1(0)\to \mathcal{M}_{p, r_i},\qquad z\mapsto r_i^{-1}(u(x+r_iz)-p)
    \end{align*}
    converges up to subsequences weakly in $W^{1,2}(B_1(0))$ to the linear map $v(z):=Du
    (x)[z]$.\\
    Moreover, for any $r_i$ we can find $k(r_i)\in \mathbb{N}$, denoted as $k_i$, such that $v_{r_i, k_i}\rightharpoonup v$ in $W^{1,2}$.\\
\end{lemma}
\begin{proof}
    As $x$ is a Lebesgue point of $\nabla u$, there exists a constant $C$ such that
    \begin{align*}
        \int_{B_1(0)}\lvert\nabla u^r \rvert^2=\frac{1}{r^2}\int_{B_r(x)}\lvert \nabla u\rvert^2\leq C\text{ for any }r>0.
    \end{align*}
    Moreover, we claim that
    the sequence $\{u^{r_i}\}_{i\in \mathbb{N}}$ is uniformly bounded in $L^2(B_1(0))$.
    Indeed, if $q\in(1,2)$, Theorem 6.2 of \cite{Evans-Gariepy} implies that for almost any $x\in \Sigma$,
    \begin{align*}
        \left(\frac{1}{\lvert B_r\rvert}\int_{B_r}\lvert u(x+y)-u(x)-Du(x)[y]\rvert^{q^\ast} dy\right)^\frac{1}{q^\ast}=o(r).
    \end{align*}
Dividing by $r$ and setting $y=rz$ we obtain
\begin{align*}
    &\frac{1}{\lvert B_1(0)\rvert^\frac{1}{q^\ast}}\lVert u^r-Du(x)[\cdot]\rVert_{L^{q^\ast}(B_1(0))}\\=&\left(\frac{1}{\lvert B_1(0)\rvert}\int_{B_1(0)} \lvert r^{-1}u(x+rz)-r^{-1}u(x)-Du(x)[z]\rvert^{q^\ast}dz\right)^\frac{1}{q^\ast}=o(1),
\end{align*}
    which proves the claim. Let $D:=\sup_{i\in \mathbb{N}}\lVert u^{r_i} \rVert_{L^2(B_1(0))}<\infty$.  
    As $\{u^{r_i}\}_{i\in \mathbb{N}}$ is bounded in $W^{1,2}(B_1(0))$, there is a subsequence of $\{ r_i\}_{i\in \mathbb{N}}$ (not relabeled) such that $\{u^{r_i}\}_{i\in \mathbb{N}}$ converges weakly in $W^{1,2}(B_1(0))$ to $v(z):=Du(x)[z]$. As the limit is independent from the subsequence, the convergence actually holds for the original sequence.
    Next note that
    \begin{align*}
        \int_{B_1(0)}\lvert\nabla v_{k,r}\rvert^2d\vol_{g_0}=\frac{1}{r^2}\int_{B_r(x)}\lvert\nabla u_k\rvert^2d\vol_{g_0}=\frac{2}{r^2}\int_{B_r(x)}d\nu_k',
    \end{align*}
    since $\nu_k'=\frac{1}{2}\lvert\nabla u_k\rvert^2_{g_0}d\vol_{g_0}$. Recall that $\nu_k'\rightharpoonup\nu$ and that $\nu=NJ(\nabla u)\vol_{g_0}$ for some integer valued function $N$, bounded by a constant $\overline{N}$. Then for any $r>0$
    \begin{align*}
        \limsup_{k\to\infty}\int_{B_1(0)}\lvert \nabla v_{k,r}\rvert^2d\vol_{g_0}\leq \frac{\overline{N}}{r^2}\int_{B_{2r}(x)}\lvert\nabla u\rvert^2d\vol_{g_0}= \frac{4\overline{N}}{(2r)^2}\int_{B_{2r}(x)}\lvert\nabla u\rvert^2d\vol_{g_0}\leq 4\overline{N}C.
    \end{align*}
    In particular, for any $r_i$ there exist $K_i\in \mathbb{N}$ such that for any $k\geq K_i$ there holds
    \begin{align*}
        \int_{B_1(0)}\lvert \nabla v_{k, r_i}\rvert^2\leq 5\overline{N}C.
    \end{align*}
Let $d$ be a distance for the weak $W^{1,2}$ convergence in
\begin{align*}
    M=\{f\in W^{1,2}(B_1(0))\vert\lVert f\rVert_{W^{1,2}}\leq 5\overline{N}C+2D\}.
\end{align*}
As $u^{r_i}\rightharpoonup v$ weakly in $W^{1,2}$, we know that $d(u^{r_i}, v)\to 0$.\\
Since \(u_k(x)\to u(x)\), for any $r_i$
\begin{align*}
    v_{r_i, k}\rightharpoonup u^{r_i}\text{ weakly in }W^{1,2}(B_1(0))
\end{align*}
as $k\to\infty$, and the convergence holds strongly in $L^2$. Therefore, for any $i\in \mathbb{N}$ there exist $k_i\geq K_i$ such that
\begin{align*}
    d(v_{k_i, r_i}, u^{r_i})\leq \frac{1}{i}.
\end{align*}
We conclude that, up to subsequences,
\begin{align*}
    v_{r_i, k_i}\rightharpoonup v \text{ weakly in }W^{1,2}(B_1(0)),
\end{align*}
as $i\to\infty$.
\end{proof}

Next we present the main step in the proof of the multiplicity one. The main idea, borrowed from \cite{PR-Multiplicity-One}, is to define an averaged multiplicity at various scales, and to show that this quantity doesn't change if we move from one scale to another in a suitable way.

In the following, given an immersion $u\in C^1(\Omega,\R^Q)$ from a 2-dimensional domain $\Omega$, $z\in \Omega$ and a 2-plane $\Pi\in\operatorname{Gr}_2(\R^Q)$, let $\pi_\Pi$ be the orthogonal projection from $\R^Q$ to $\Pi$ and set
\begin{align*}
    N(u,B_r(z),\Pi):\Pi\rightarrow\N\cup\lbrace\infty\rbrace,q\mapsto\#((\pi_\Pi\circ u)^{-1}(q)\cap B_r(z)),
\end{align*}
which is the multiplicity of the projection of the surface $u(B_r(z))$ onto $\Pi$.
For $s\in \Pi$ we also set
\begin{align*}
    n(u, B_r(z),B_t^{\Pi}(s))=\left[\fint_{B_t^{\Pi}(s)}N(u,B_r(z),\Pi)(q)d\mathcal{H}^2(q)+\frac{1}{2}\right]\in\N.
\end{align*}
By the proof of Lemma \ref{lem: N-jacobian} (see in particular \eqref{eq: definition-N-limit}), the multiplicity $N$ of the limiting parametrized varifold is computed as follows: if $\alpha>0$ is such that
\begin{align*}
    \mathcal{C}=\{du(x)[y]\vert y\in \partial B_1(0)\}
\end{align*}
encloses $B_{2\alpha}(p)$ in $\Pi=p+\operatorname{Im}(D u(x)[\cdot])$, then, since $N(x)\in \mathbb{N}$,
\begin{align}\label{eq: obtaining-N-in-the-limit}
    N(x) =\lim_{i\to\infty}\frac{1}{\lvert B_\frac{\alpha}{2}^\Pi(p)\rvert}\int_{B_\frac{\alpha}{2}^\Pi(p)}N(v_{r_i,k_i}, B_1(x),\Pi)(q)d\mathcal{H}^2(q)=\lim_{i\to\infty} n(v_{r_i,k_i}, B_1(x), B_\frac{\alpha}{2}^\Pi(p)),
\end{align}
for any sequence $r_i\to 0$ satisfying \eqref{eq: condition-radii-1}, and for any sequence $k_i\to\infty$ sufficiently fast such that $u_{k_i}|_{\de B_{r_i}}$ converges in $C^0$ to $u\vert_{\de B_{r_i}}$.
The aim is then to show that for large $i$ we have $n(v_{r_i, k_i}, B_1(x), B_\frac{\alpha}{2}^\Pi(p))=1$.
The key step in the proof will be Proposition \ref{prop: iteration lemma}, for which we need the following definition.
\begin{definition}
    A map $\varphi\in W_{\text{loc}}^{1,2}\cap C^0(\C,\C)$ is a $K$-quasiconformal homeomorphism of $\C$ if it is a homeomorphism from $\mathbb{C}$ to $\mathbb{C}$ and satisfies
    \begin{align*}
        \partial_{\overline{z}}\varphi=\mu \partial_z \varphi
    \end{align*}
    in the distributional sense, for some $\mu$ in
    \begin{align*}
        \mathscr{E}_K=\left\{\mu\in L^\infty(\C,\C)\bigg\vert \lVert \mu\rVert_{L^\infty}\leq \frac{K-1}{K+1}\right\}.
    \end{align*}
    We will denote by $\mathscr{D}_K$ the set of $K$-quasiconformal homeomorphisms $\varphi$ of $\C$ such that
    \begin{align*}
        \varphi(0)=0,\qquad\min_{x\in \partial B_1^2(0)}\lvert \varphi(x)\rvert=1.
    \end{align*}
    For any  $2$-plane $\Pi$ in $\R^Q$, we let $\mathscr{D}_K^\Pi$ denote the set of maps of the form $i_\Pi\circ \varphi$, where $i_\Pi$ is a linear isometry from $\C$ to $\Pi$, and $\varphi\in \mathscr{D}_K$.
\end{definition}

\begin{proposition}\label{prop: iteration lemma}
    Let $x\in \mathcal{G}'$. Let $p\in\mathcal{M}$ and $u:B_1(x)\rightarrow\mathcal{M}_{p,\ell}$ be a conformal immersion, $\theta$-critical for the functional
    \begin{align*}
        \operatorname{Area}(u)+\tau^4\int_{B_1(x)}\vert\sff_{\mathcal{M}_{p,\ell}}^u\vert^4d\vol_{g_u}+h\int_{B_1(x)\times [0,1]}f_u^\ast d\vol_{\mathcal{M}_{p,\ell}}.
    \end{align*}
    Given $E,K>0,A\in\N+\frac{1}{2}$, there exist a constant $\bar{\delta_0}>0$ depending only on $A,K$ with the following property.  For every \(\delta_0\in(0,\bar\delta_0)\), there
    exists a constant \(\varepsilon_0>0\), depending only on
    \(E,A,K,\delta_0,\mathcal M\), such that the following holds. Suppose that
    \begin{enumerate}[label=\arabic*.]
        \item $0<\ell,h,\tau<\varepsilon_0$ and $\theta\in (0,\tau^5)$;
        \item There exists a $K$-quasiconformal homeomorphism $f$
        in $\mathscr{D}_K^\Pi$ (for a 2-dimensional plane $\Pi$) such that
        \begin{align*}
            \vert u-f\vert<\delta_0\text{ on }\de B_1(0)\cup \de B_{s(K)}(0)\cup\de B_{s(K)^2}(0);
        \end{align*}
        
        \item $\frac{1}{2}\int_{B_1(0)}\vert\nabla u\vert^2\leq E$;
        \item $\operatorname{Area}(u(B_1(0))\cap B_1(0))\leq  A\pi$ and $\operatorname{Area}(u(B_1(0))\cap B_{\eta(K)}(0))/\eta(K)^2\leq A\pi$;
        \item $\tau^4\log\tau^{-1}\int_{B_1(0)}\vert\sff^u_{\mathcal{M}_{p,\ell}}\vert^4\,d\vol_{g_u}\leq\varepsilon_0.$
    \end{enumerate}
    Here the assumptions are stated for a conformal chart centered at $x$.
    Then we can find new scales $r',\ell'$, a new point $p'\in \mathcal{M}_{p,\ell}$ and a new plane $\Pi'$ such that
    \begin{enumerate}[label=\arabic*'.]
        \item $r',\ell'\in (\varepsilon_0,s(K))$;
        \item There exists a new $K’(A)$-quasiconformal homeomorphism $f'\in \mathscr{D}_{K'(A)}^{\Pi'}$ from $\mathbb{C}$ to $\Pi'$ with
        \begin{align*}
            u'=\ell'^{-1}(u(x+r'\cdot)-p'),
        \end{align*}
        satisfying $\vert u'-f'\vert\leq\delta_0$ on $\de B_1(0)\cup \de B_{s(K'(A))}(0)\cup\de B_{s(K'(A))^2}(0)$;
        \item $\frac{1}{2}\int_{B_1(0)}\vert\nabla u'\vert^2\leq E'(A):=25\pi AD(K'(A))^2$;
        \item $\operatorname{Area}(u'(B_1(0))\cap B_1(0)),\operatorname{Area}(u'(B_1(0))\cap B_{\eta(K'(A))}(0))/\eta(K'(A))^2\leq\pi A;$
        \item $n(u,B_{s(K)^2}(x),B_{\eta(K)}^{\Pi}(0))=n(u',B_{s(K'(A))^2}(0),B_{\eta(K'(A))}^{\Pi'}(0))$.
    \end{enumerate}
    Moreover, $u'$ is $\theta \ell'^{-\frac{5}{2}}$-critical for the functional
        \begin{align*}
           \operatorname{Area}(v)+(\tau \ell'^{-1})^4\int_{B_1(0)}\lvert\sff_{\mathcal{M}_{p', \ell'\ell}}^v\rvert^4d\vol_{g_v}+\ell'h\int_{B_1(0)\times[0,1]}f_v^\ast d\vol_{\mathcal{M}_{p',\ell'\ell}},
        \end{align*}
        on $W^{2,4}_{\text{imm}}(B_1(0), \mathcal{M}_{p',\ell'\ell})$.
\end{proposition}
Here, $K'(A)=(16A)^2$, and $D(K),s(K)$ are constants such that for all $K$-quasiconformal map $\varphi: \C\rightarrow\C$ with $\varphi(0)=0$ and $\min_{\lvert x\rvert=1}\lvert\varphi(x)\rvert\geq 1$ there holds
\begin{align*}
    \vert\varphi(z)\vert\leq D(K)\,\forall z\in B_1(0)\text{ and }\varphi(B_{s(K)})\subset B_{1/2}(0).
\end{align*}
Moreover, $\eta(K)$ is the largest number such that for any $\varphi$ as above
\begin{align*}
    \vert\varphi(z)\vert\geq4\eta(K),\text{ for all }\vert z\vert=s(K)^2.
\end{align*}
We will chose $\bar \delta_0$ to be the largest number smaller or equal to $\min\{\frac{1}{8}, 6\eta(K'(A))\}$, such that Lemma \ref{lem: appendixB-Lem4.3} holds for $\delta_0$ (for $s=s(K)$) and such that \eqref{eq: new-bound-A} holds.
\begin{proof}
    Without loss of generality, we may assume that $x=0$. We prove the statement by contradiction. Assume that for a sequence $\varepsilon_k\rightarrow 0$ (instead of $\varepsilon_0$), there exists a sequence $(u_k, p_k, \Pi_k, \ell_k, f_k)$ which satisfies the assumptions but does not satisfy the conclusion of the Proposition. Up to subsequences, we may assume that $p_k\to p$, $\Pi_k\to \Pi$ for some $p\in \mathcal{M}$ and some plane $\Pi\in \operatorname{Gr}_2(\mathbb{R}^q)$, and--- by compactness of $(\mathscr{D}_K, \lVert\cdot\rVert_{L^\infty})$ (see Corollary A.4 in \cite{PR-Multiplicity-One})--- that $f_k$ converges to a map $f\in \mathscr{D}_K^\Pi$ in $C^0_{\text{loc}}(\C)$.
    
    By taking a further subsequence, we may assume that the sequence $\{u_k\}_{k\in \mathbb{N}}$ converges weakly in $W^{1,2}(B_1(0))$ to a map $u$.
    In fact, Assumption $2$ and the fact that $f_k$ converges to a limiting map $f$ in $C^0_{\text{loc}}(\C)$ imply that the sequence of traces $\{u_k\vert_{\partial B_1(0)}\}_{k\in \mathbb{N}}$ is bounded in $L^\infty(\partial B_1(0))$. This, together with Assumption $3$, implies that the sequence is bounded in $W^{1,2}(B_1(0))$ (otherwise, one would find a contradiction by considering the limit of the maps $\tilde{u}_k=\frac{u_k}{\lVert u_k\rVert_{L^2}}$ in $W^{1,2}(B_1(0))$), and therefore it has a subsequence converging weakly in $W^{1,2}(B_1(0))$ to some map $u$.
    The weak convergence of the traces of $u_k$ to $u$ implies (testing against functions in $L^1$) that
    \begin{align}\label{eq: delta-0-selected-circles}
        \lvert u-f\rvert\leq \delta_0\text { on }\partial B_1(0)\cup \partial B_{s(K)}(0)\cup \partial B_{s(K)^2}(0).
    \end{align}
    Note that $\{u_k(\partial B_1(0))\}_{k\in \mathbb{N}}$ is a sequence of compact sets contained in a compact region of $\mathbb{R}^Q$, therefore up to subsequences it converges in the Hausdorff distance to a compact set $\Gamma_\infty$ contained in a $\delta_0$-neighborhood of $f(\partial B_1(0))$.
    
    \textbf{Step 1} \textit{\underline{Claim}: The varifolds induced by $u_k\vert_{B_{s(K)}}$ converge to a parametrized varifold $\textbf{v}_{B_{s(K)}}=(\varphi(B_{s(K)}(0)),u\circ\varphi^{-1}, N)$, where $\varphi$ is a $K'(A)$-quasiconformal homeomorphism from $B_{s(K)}(0)$ to an open $\Omega\subset\mathbb{C}$, and $\textbf{v}_{B_{s(K)}}$ is stationary in $\mathbb{R}^q\smallsetminus \Gamma_\infty$.
    }
    \begin{proof}
    
    The sequence $\{u_k\}_{k\in \mathbb{N}}$ satisfies--- up to subsequences---  the following conditions:
    \begin{enumerate}
        \item $u_k(\partial B_1(0))\to \Gamma_\infty$ in the Hausdorff distance.
        \item $\frac{1}{2}\int_{B_1(0)}\lvert \nabla u_k\rvert^2\leq E$
        \item $\tau_k^4\log\tau_k^{-1}\int_{B_1(0)}\lvert\sff^{u_k}_{\mathbb{R}^Q}\rvert^4d\vol_{g_{u_k}}\to 0$.
        \item $u_k$ is $\theta_k$-critical for
            \begin{align*}
        \operatorname{Area}(u)+\tau_k^4\int_{B_1(0)}\vert\sff_{\mathcal{M}_{p,\ell}}^u\vert^4+h_k\vol(f_u).
        \end{align*}
        \item $h_k\to 0$, $\tau_k\to 0$, $\theta_k\to 0$ and $\theta_k\leq \tau_k^5$.
    \end{enumerate}
    
    Recall that by the Gauss-Codazzi equations, the difference between $\sff^{u_k}_{\mathbb{R}^Q}$ and $\sff^{u_k}_{\mathcal{M}_{p_k, \ell_k}}$ can be expressed in terms of the second fundamental form of ${\mathcal{M}_{p_k, \ell_k}}$, therefore this difference goes to zero as $\ell_k\to 0$, so that condition $(3)$ above follows from Assumption $5$.
    Thus arguing just as in the proof of Lemma \ref{lem: N-jacobian}, one shows that for any open $\omega\Subset B_1(0)$, the map $u\vert_\omega$ is continuous and satisfies the convex hull property, i.e.
    \begin{align*}u(\overline\omega) \subset\operatorname{co}(u(\partial\omega)),
    \end{align*}
    and if $u(\omega)\cap \Gamma_\infty=\emptyset$, the varifolds induced by $u_k\vert_\omega$ converge in $\mathbb{R}^Q\smallsetminus\Gamma_\infty$ up to subsequences to a parametrized varifold $\textbf{v}_\infty=(\varphi(\omega), u\circ\varphi^{-1}, N)$,
    where $\varphi$ is a quasi conformal homeomorphism from $\omega$ to an open set $\Omega\subset \C$.\par
    Note that, by definition of $\mathscr{D}_K$ and $s(K)$, there holds $ f(B_{s(K)}(0))\smallsetminus B_\frac{1}{2}(0)=\emptyset$ and $f(\partial B_1(0))\cap B_1(0)=\emptyset$. Thus, since $\delta_0<\frac{1}{8}$, by the convex hull property and \eqref{eq: delta-0-selected-circles}
    \begin{align*}
        u(B_{s(K)}(0))\subset\operatorname{co}(u(\partial B_{s(K)}))\subset \overline{B_\frac{5}{8}(0)}.
    \end{align*}
    Therefore
    \begin{align}\label{eq: distance-from-boundary}
        \operatorname{dist}(u(B_{s(K)}(0)), \Gamma_\infty)\geq\frac{1}{4}.
    \end{align}
    Then $\textbf{v}_{B_{s(K)}}=(\varphi(B_{s(K)}(0)), u\circ\varphi^{-1}, N)$ is stationary in $\mathbb{R}^Q\smallsetminus u(\partial B_{s(K)}(0) )$ (i.e. it is a \textit{local parametrized stationary varifold}), for some quasiconformal homeomorphism $\varphi$ from $B_{s(K)}(0)$ to some open set $\Omega\subset\C$. Without loss of generality we may assume that $\varphi(0)=0$.
    
    We claim that the map $\varphi$ can be chosen to be $K'(A)$-quasiconformal. Indeed, since $\textbf{v}_{B_{s(K)}}$ is stationary and $\lVert \textbf{v}_{B_{s(K)}}\rVert(B_1(0))\leq A\pi $ by Assumption $4$, the monotonicity formula and \eqref{eq: distance-from-boundary} imply that the density at any point in $u(B_{s(K)}(0))$ is bounded above by $16A$. Thus, arguing as in the proof of Lemma \ref{lem: 5-12-P-FB} (for more details, see Section 4 in \cite{PR-Regularity}), we see that $\varphi$ can be chosen to be $K'(A)$-quasiconformal.
    \end{proof}
    \textbf{Step 2} \textit{\underline{Claim}: There exist $r', \ell'>0$ such that for $k$ large enough,
    \begin{align*}
        u_k'=\ell'^{-1}(u_k(r'\cdot)-p'_k)
    \end{align*}
    (for some point $p'_k\in \mathcal{M}_{p_k,\ell_k}$) satisfies the conclusions of the proposition.}
    \begin{proof}
     Set $\bar{u}=u\circ\varphi^{-1}$. As $\textbf{v}_{B_{s(K)}}$ is a local parametrized stationary varifold, Theorem 5.7 in \cite{PR-Regularity} implies that $\bar u$ is harmonic and $N$ is constant. Thus by Lemma \ref{lem: appendixB-Lem4.3} (for $u(s(K)\cdot)$ and $f(s(K)\cdot)$), $\pi_\Pi\circ\overline{u}$ is injective on $\varphi(B_\frac{s(K)}{2})$, and the differential $D \overline{u}(0)$ is a conformal linear map of rank two, spanning a $2$-plane $\Pi'$.
     By Lemma \ref{lem: appendixB-LemmaA.1} (applied to $\eta(K)^{-1}\pi_\Pi\circ u(s(K)\cdot)$ and $\eta(K)^{-1}f(s(K)\cdot)$) and the fact that $\eta(K)^{-1}\delta_0<1$, there exist $y\in B_{s(K)^2}(0)$ such that $\pi_\Pi\circ u(y)=0$.
     Note that $\lvert u(y)\rvert\leq \delta_0$; indeed, since 
     \begin{align*}
         \lvert u(s(K)^2x)-f(s(K)^2x)\rvert\leq \delta_0\text{ for }x\in \partial B_1(0),
     \end{align*}
     we have
     \begin{align*}
         u(\partial B_{s(K)^2}(0))\subset\{p\in \mathbb{R}^Q\vert \lvert\pi_\Pi^\perp(p)\rvert\leq \delta_0\}.
     \end{align*}
     Therefore, since $y\in B_{s(K)^2}(0)$, by the convex hull property there holds
     \begin{align*}
         \lvert u(y)\rvert=\lvert\pi_\Pi^\perp\circ u(y)\rvert\leq \delta_0.
     \end{align*}
     Note that the varifolds induced by $u_k\vert_{B_{s(K)^2}}$ converge to a limiting varifold $\textbf{v}_{B_{s(K)^2}}$.
     Since $\lVert \textbf{v}_{B_{s(K)^2}}\rVert(B_{\eta(K)}(0))\leq \eta(K)^2 A\pi$ by Assumption $4$, the stationarity of $\textbf{v}_{B_{s(K)^2}}$ in $B_{\eta(K)}(0)$ (as $B_{\eta(K)}(0)\cap u(\partial B_{s(K)^2}(0))=\emptyset$) and the monotonicity formula imply that the density at $u(y)$ is at most
    \begin{align}\label{eq: new-bound-A}
        \frac{\eta(K)^2A\pi}{\pi(\eta(K)-\delta_0)^2}\leq A+\frac{1}{4},
    \end{align}
    Recall that the density of $\textbf{v}_{B_{s(K)^2}}$ at $u(y)$ is an integer.
    Therefore the upper bound can be improved to $[A]=A-1/2$, given that we selected $A$ to be in $\N+1/2$.
    Since $N$ is constant, by the definition of the density we have
    \begin{align}\label{eq: monotonicity-t-small}
        \Vert \textbf{v}_{B_{s(K)}}\Vert(B_t(u(0)))\leq (A-1/4)\pi t^2,\text{ for all small }t.
    \end{align}
    Next let's choose $r'$ and $\ell'$. Since $u=\bar{u}\circ\varphi$, by smoothness of $\bar{u}$ we have that for $r'>0$ small enough, for all $x\in B_1(0)$
    \begin{align}\label{eq: closeness-to-nabla-overline-u}
        &\vert u(r'x)-u(0)-\nabla\bar{u}(0)\cdot\varphi(r'x)\vert\\
        \nonumber =&\vert\bar{u}\circ\varphi(r'x)-\bar{u}(0)-\nabla\bar{u}(0)\cdot\varphi(r'x)\vert\\
        \nonumber
        \leq&\frac{1}{2\sqrt{2}D(K'(A))}\delta_0\vert\nabla\bar{u}(0)\vert\vert\varphi(r'x)\vert,
    \end{align}
    since $\nabla \overline{u}(0)\neq 0$, by Lemma \ref{lem: appendixB-Lem4.3}.   
    Let's choose $r'$ small enough so that the above inequality holds (for any $x\in B_1(0)$) and $u_k(r'\cdot)\rightarrow u(r'\cdot)$ in $C^0$ on $\de B_1(0)\cup\de B_{s(K)}(0)\cup\de B_{s(K)^2}(0)$ (up to subsequences).
    We might need to take $r'$ smaller later, but note that it will always be possible to chose it such that all these conditions are met.\par 
    Let $\lambda=\min_{\de B_{r'}}\vert\varphi\vert$
    and set $\ell'=\frac{\vert\nabla\bar{u}(0)\vert\lambda}{\sqrt{2}}$.
    Set
    \begin{align*}
        u'=\frac{u(r'\cdot)-u(0)}{\ell'},f'=\frac{\nabla\bar{u}(0)\cdot \varphi(r'\cdot)}{\ell'}.
    \end{align*}
    Then using \eqref{eq: closeness-to-nabla-overline-u} and the fact that $\lambda^{-1}\varphi(r'\cdot)\in \mathscr{D}_{K'(A)}$ (and the definition of $D(K'(A))$), we obtain 
    \begin{align}\label{eq: u'-close-to-qchom}
        \vert u'-f'\vert\leq\frac{1}{2}\delta_0\text{ on } \overline{B_1(0)}.
    \end{align}
    Recall that $\Pi'=\operatorname{Im}(D\overline{u}(0))$ and note that $f'\in \mathscr{D}_{K'(A)}^{\Pi'}$, since $\varphi\in \mathscr{D}_{K'(A)}$ and $D \overline{u}(0)$ is conformal.
    Note also that by choosing $r'$ smaller if necessary, we can ensure that $r'$, $\ell'<s(K)$. For $k$ sufficiently large, the bound $r', \ell'>\varepsilon_k$ is also satisfied.
    \par
    Let $\textbf{v}'$ be the varifold given by
    \begin{align*}
        \left(\varphi(B_{r'}),\ell'^{-1}(\overline{u}-u(0)),N\right).
    \end{align*}
    By \eqref{eq: monotonicity-t-small}, if $r'$ is small enough (so that $\ell'$ is small enough), we can also guarantee that
    \begin{align}\label{eq: control-density1}
        \frac{\Vert \textbf{v}'\Vert(B_1(0))}{\pi},\frac{\Vert \textbf{v}'\Vert(B_{\eta(K'(A))}(0))}{\pi\eta(K'(A))^2}\leq A-\frac{1}{4}.
    \end{align}
    Since $\lambda^{-1}\varphi(r'\cdot)$ is $K'(A)$-quasiconformal, $\lambda^{-1}\varphi(B_{r'}(0))\subset D(K'(A))$. Moreover, recall that for any point in $u(B_{r'})$, the density of $\textbf{v}'$ is bounded from above by $16A$, by the last part of the proof of Step 1. Thus    
    choosing $r'$ smaller if necessary we have
    \begin{align}\label{eq: control-energy}
        \int_{B_{r'}(0)}N J(d u)&\leq \frac{16A}{2}\int_{B_{D(K'(A))\lambda}(0)}\vert\nabla\bar{u}\vert^2\\
        \nonumber
        &\leq 10A\pi(D(K'(A))\lambda)^2\vert\nabla\bar{u}(0)\vert^2=20\pi AD(K'(A))^2\ell'^2.
    \end{align}
    Now, for any $k\in \mathbb{N}$ let
    \begin{align*}
        u_k': B_1(0)\to \mathcal{M}_{p_k',\ell'\ell_k},\quad z\mapsto\ell'^{-1}(u_k(r'z)-p_k'),
    \end{align*}
    where $p_k'$ is the closest point to $p'=u(0)\in T_p\mathcal{M}$ in $\mathcal{M}_{p_k, \ell_k}$. Note that since $\mathcal{M}_{p_k,\ell_k}$ converges locally to $T_p\mathcal{M}$, $\{p_k'\}_{k\in \mathbb{N}}$ converges to $p'$.
    As $u_k(r'\cdot)\rightarrow u(r'\cdot)$ in $C^0$ on $\de B_1(0)\cup\de B_{s(K)}(0)\cup\de B_{s(K)^2}(0)$ (up to subsequences), \eqref{eq: u'-close-to-qchom} implies that up to subsequences, for $k$ sufficiently large $u'_k$ satisfies Conclusion $2'$. Moreover, as $\frac{1}{2}\lvert\nabla u_k\rvert^2d\mathcal{L}^2\rightharpoonup NJ(u)d\mathcal{L}^2$ in a neighborhood of $0$ by Lemma \ref{lem: N-jacobian}, \eqref{eq: control-energy} implies that for $r'$ sufficiently small 
    \begin{align*}
        \frac{1}{2}\int_{B_1(0)}\lvert \nabla u_k'\rvert^2\to \int_{B_1(0)}NJ(du')=\ell'^{-2}\int_{B_{r'}(0)}NJ(du)\leq 20\pi AD(K'(A))^2,
    \end{align*}
    so that for $k$ sufficiently large, $u'_k$ satisfies Conclusion $3'$.
    As the varifolds $\textbf{v}_k$ induced by $u_k'$ converge to $\textbf{v}'$ by Lemma \ref{lem: Pigati-5.7}, by \eqref{eq: control-density1} Conclusion $4'$ is satisfied by $u'_k$, for $k$ sufficiently large.\par
    Next, note that by Lemma \ref{lem: appendixB-LemmaA.1}, $\pi_\Pi\circ u(B_{s(K)^2}(0))\supset B_{\eta(K)}^\Pi(0)$, and recall that $\pi_\Pi\circ \overline{u}$ is a diffeomorphism from $\varphi(\overline{B_{\frac{s(K)}{2}}(0)})$ to its image by Lemma \ref{lem: appendixB-Lem4.3}.
    Therefore
    \begin{align*}
        \frac{\lVert {\pi_{\Pi}}_\ast \textbf{v}_{B_{s(K)^2}}\rVert(B_{\eta(K)}^{\Pi}(0))}{\pi \eta(K)^2}=\frac{1}{\pi \eta(K)^2}\int_{B_{s(K)^2}(0)\cap (\pi_\Pi\circ u)^{-1}(B_{\eta(K)}^\Pi(0))}\hspace{-3.3cm}NJ(\pi_\Pi\circ u)\hspace{1.5cm}=\frac{N\operatorname{Area}(B_{\eta(K)}(0))}{\pi \eta(K)^2}=N.
    \end{align*}
    As a result, if we denote by $\textbf{v}_{B_{s(K)^2}}^k$ the varifold induced by $u_k\vert_{B_{s(K)^2}(0)}$, for any $k\in \mathbb{N}$ we have
    \begin{align*}
        \fint_{B_{\eta(K)}^{\Pi_k}(0)}\hspace{-0.6cm}N(u_k, B_{s(K)^2}(0), \Pi_k)=\frac{\lVert {\pi_{\Pi_k}}_\ast \textbf{v}^k_{B_{s(K)^2}}\rVert\left(B_{\eta(K)}^{\Pi_k}(0)\right)}{\pi \eta(K)^2}\to  \frac{\lVert {\pi_{\Pi}}_\ast \textbf{v}_{B_{s(K)^2}}\rVert\left(B_{\eta(K)}^{\Pi}(0)\right)}{\pi \eta(K)^2}=N.
    \end{align*}
    For the convergence, we are using that $\lVert {\pi_{\Pi}}_\ast \textbf{v}_{B_{s(K)^2}}\rVert(\partial B_{\eta(K)}^{\Pi}(0))=0$ by the conclusion of Lemma  \ref{lem: appendixB-Lem4.3}.
    Hence 
    $n(u_k,B_{s(K)^2}(0),B_{\eta(K)}^{\Pi_k}(0))$ is equal to $N$ for $k$ large.\\
    Next we claim that $\pi_{\Pi'}\circ\ell'^{-1}(u(B_{r's(K'(A))^2})-p')\supset B_{\eta(K'(A))}^{\Pi'}$. Indeed,
    \eqref{eq: closeness-to-nabla-overline-u} implies
    that
    \begin{align*}
        \max_{\lvert x\rvert\leq1}\left\lvert\ell'^{-1}(u(r'x)-p')-\ell'^{-1}\nabla\overline{u}(0)\cdot\varphi(r'x)\right\rvert\leq \frac{\delta_0}{2}.
    \end{align*}
    As $\psi:=\ell'^{-1}\nabla\overline{u}(0)\cdot\varphi(r'\cdot)\in \mathscr{D}_{K'(A)}^{\Pi'}$, Lemma \ref{lem: appendixB-LemmaA.1} applied to
    \begin{align*}
        \left(4\eta(K'(A))\right)^{-1}\pi_{\Pi'}\circ\ell'^{-1}(u(r's(K'(A))^2\cdot)-p')\text{ and }\left(4\eta(K'(A))\right)^{-1}\psi(s(K'(A))^2\cdot),
    \end{align*}
    implies
    \begin{align*}
        B_{\eta(K'(A))}^{\Pi'}(0)\subset B_{4\eta(K'(A))-\frac{\delta_0}{2}}^{\Pi'}(0)\subset \pi_{\Pi'}\circ\ell'^{-1}(u(B_{r's(K'(A))^2}(0))-p'),
    \end{align*}
    as desired. Now let $\textbf{v}_{k, B_{s(K'(A))^2}(0)}$ denote the varifolds induced by $u_k'\vert_{B_{s(K'(A))^2}}$, and $\textbf{v}_{B_{s(K'(A))}}'$ the parametrized varifold $(\varphi(B_{r's(K'(A))^2}), \ell'^{-1}(\overline{u}-u(0)), N)$.
    Then, arguing as above, we obtain 
    \begin{align*}
        \fint_{B_{\eta(K'(A))}^{\Pi'}(0)}\hspace{-0.6cm}N(u_k', B_{s(K'(A))^2}(0), \Pi')=\frac{\lVert {\pi_{\Pi'}}_\ast \textbf{v}_{k, B_{s(K'(A))^2}}\rVert\left(B_{\eta(K'(A))}^{\Pi'}(0)\right)}{\pi \eta(K'(A))^2}\to  \frac{\lVert {\pi_{\Pi'}}_\ast \textbf{v}'_{B_{s(K'(A))^2}}\rVert\left(B_{\eta(K'(A))}^{\Pi'}(0)\right)}{\pi \eta(K'(A))^2}=N.
    \end{align*}
    We conclude that
    $n(u_k',B_{s(K'(A))^2}(0),B_{\eta(K'(A))}^{\Pi'})$ is equal to $N$ for $k$ sufficiently large. Therefore, for $k$ sufficiently large we have
    \begin{align*}
        n\left(u_k,B_{s(K)^2}(0),B_{\eta(K)}^{\Pi_k}(0)\right)=n\left(u_k',B_{s(K'(A))^2}(0),B_{\eta(K'(A))}^{\Pi'}(0)\right),
    \end{align*}
    as desired.
    \end{proof}
    The previous step shows that for $k$ large enough, $u_k'$ satisfies the conclusion of the proposition; this yields the desired contradiction.\par
    Thus there exist $\varepsilon_0$ and $r', \ell', u', \Pi'$ as in the conclusion of the proposition, and by Lemma \ref{lem: critical-rescaled-functional}, $u'$ is $\theta\ell'^{-\frac{5}{2}}$-critical for the rescaled functional.
\end{proof}

With this proposition in hand, we are able to establish the multiplicity one in the limit.

\begin{thm}\label{thm: multiplicity-one}
    The multiplicity function $N$ is equal to $1$ almost everywhere on $\Sigma$.
\end{thm}
\begin{proof}
    Note that it will be enough to show that for $x\in \mathcal{G}'$, $N(x)=1$, since
    \begin{align*}
        \int_{\Sigma\smallsetminus\mathcal{G}'}\lvert\nabla u\rvert^2d\vol_{g_0}=0,
    \end{align*}
    by the definition of $\mathcal{G}'$.\\
    Let $x\in \mathcal{G}'$. We work in a conformal chart centered at $x$.\\
    
    \textbf{Step 1} \textit{\underline{Claim}: there exist a constant $\lambda$ and sequences $r_i\to 0$, $k_i\to \infty$ such that $r_i>i^{-1}$ and  for $i$ large enough, $\lambda v_{r_i, k_i}$ satisfies the assumptions of Proposition \ref{prop: iteration lemma}, for constants $E$, $K$ and $A$ depending only on $\bar{A}:=\sup_k\operatorname{Area}(u_k)$ and $\mathcal{M}$.}\\

    Set $\lambda:=\left(\min_{\lvert y\rvert=1}\lvert Du(0)y\rvert\right)^{-1}$ and
    \begin{align*}
        E:=3\lVert N\rVert_{L^\infty}^2\pi,\quad K:=\lVert N\rVert_{L^\infty},\quad A\geq\frac{3\lVert N\rVert_{L^\infty}^2}{\eta(K)^2},
    \end{align*}
    where $A$ is the smallest number in $\mathbb{N}+\frac{1}{2}$ satisfying the inequality above. As discussed in the proof of Lemma \ref{lem: N-jacobian}, $\lVert N\rVert_{L^\infty}$ can be controlled in terms of $\mathcal{M}$ and $\lvert\textbf{v}\rvert(\mathcal{M})$ (which is bounded by $\overline{A}$).
    Let $\bar\delta_0$, $\bar\delta_0'$  be the constants in Proposition \ref{prop: iteration lemma} corresponding respectively to $K,A$ and $K'(A), A$. Let $\delta_0:=\frac{\min\{\bar \delta_0, \bar \delta_0'\}}{2}$ and let $\varepsilon_0$ and $\varepsilon_0'$ be the constants in Proposition \ref{prop: iteration lemma} corresponding respectively to \(E,A,K,\delta_0,\mathcal M\) and \(E'(A),A,K'(A),\delta_0,\mathcal M\). Set $\varepsilon_\ast:=\min\{\varepsilon_0,\varepsilon_0'\}$.
    By Lemma \ref{lem: convergence-blow-up}, we can pick sequences ${r_i}\to 0$ and ${\overline k_i}$, and a subsequence of $k_i$ (not relabeled) such that $r_i>i^{-1}$ and $v_{{r_i},{k_i} }\rightharpoonup D u(0)[\cdot]$ weakly in $W^{1,2}(B_1(0))$ for any sequence $\{k_i\}_{i\in \mathbb{N}}$ with $k_i\geq \overline k_i$, and such that
\begin{align*}
    \lvert \lambda v_{r_i, k_i}(y)-\lambda D u(0)[ y]\rvert\leq \delta_0 \text{ for }y\in\de B_1(0)\cup\de B_{s(K)}(0)\cup\de B_{s(K)^2}(0)
\end{align*}
when $i$ is sufficiently large. Note that $\lambda D u(0)[\cdot]\in \mathscr{D}_K^{\Pi}$ (where $\Pi=Du(0)[T_0\Sigma]$). Therefore Assumption $2$ is satisfied.\par
By \eqref{eq: condition-convergence-nuk-nuinfty}, choosing the numbers $\overline k_i$ to be sufficiently large, we can ensure that
\begin{align*}
    \lim_{i\to\infty}\frac{1}{2}\int_{B_{1}(0)}\lvert \nabla \lambda v_{r_i,k_i}\rvert^2d\vol_{g_0}=N(0)\lambda^2\pi\lvert\partial_x u\wedge\partial_yu\rvert(0)\leq 2N(0)^2\pi.
\end{align*}

Thus with our choice of $E$ Assumption 3 is satisfied. Moreover, since the maps $\lambda v_{r_i, k_i}$ are conformal, Assumption 4 is also satisfied by our choice of $A$.
    Finally note that by Lemma \ref{lem: critical-rescaled-functional},
if $u_k$ is $\theta_k$-critical for $E_{H,\sigma_k}$, then $\lambda v_{r_i, k}$ is $\theta_k (\lambda r_i^{-1})^{\frac{5}{2}}$-critical for $E_{ \lambda^{-1}r_iH, \sigma_k \lambda r_i^{-1}}$, and, if $\tau_i=\sigma_{k_i}\lambda r_i^{-1}$,
    \begin{align}\label{eq: control-entropy-upto-r}
        \tau_i^4 \log\tau_i^{-1}\int_{B_1(0)}\lvert\sff^{\lambda v_{r_i, k}}\rvert^4d\vol_{\lambda v_{r_i, k_i}}\leq \lambda^{2}r_i^{-2}\sigma_{k}^4\log\sigma_{k}^{-1}\int_{B_{r_i}(x)}\lvert\sff^{u_{k}}\rvert^4d\vol_{u_k}.
    \end{align}
Therefore, increasing $\bar{k}_i$ if necessary, we can ensure that for $i$ sufficiently large, for any $k_i\geq \bar k_i$, $\lambda v_{r_i, k_i}$ is $\theta_i$-critical for $E_{\tau_i, h_i}$ (where $h_i=\lambda^{-1}r_i H$) with $\tau_i,  h_i\to 0$ and $\theta_i\leq \tau_i^5$, so that Assumption $1$ is satisfied. On the other hand, choosing $\bar{k}_i$ larger if necessary, by \eqref{eq: control-entropy-upto-r} we can guarantee that Assumption $5$ is also satisfied for $\varepsilon_0$ corresponding to $E,K$ and $A$.

For sequences $\{r_i\}_{i\in \mathbb{N}}$ and $\{k_i\}_{i\in \mathbb{N}}$ as in the previous step, set $v_i:=\lambda v_{r_i, k_i}$.\\
\\
\textbf{Step 2} \textit{Applying Proposition \ref{prop: iteration lemma} iteratively, we show that
\begin{align*}
    n\left(v_i, B_{s(K)^2}(0), B_{\eta(K)}^{\Pi}(0)\right)=1,
\end{align*}
and we deduce that $N(x)=1$.}\\

For $i$ as above and $j\in \mathbb{N}$, we want to apply Proposition \ref{prop: iteration lemma} iteratively to obtain maps
\begin{align*}
    v_i^{(j)}:= \left(\ell^{(j)}_i\right)^{-1}\left(v_i^{(j-1)}(r^{(j)}_i\cdot)-p'^{(j)}_i\right),
\end{align*}
for $j\in \{1,...,m_i\}$
satisfying the conclusions of Proposition \ref{prop: iteration lemma}, where $v_i^{(0)}:=v_i$ and $\ell^{(j)}_i$, $r^{(j)}_i$, $p^{(j)}_i$ are the constants and the point given by Proposition \ref{prop: iteration lemma} at the $j$-th iteration.
For $i,j\in \mathbb{N}$, let
\begin{align*}
    L_i^{(j)}:=\left(\prod_{a=1}^j\ell^{(a)}_i\right), \quad  R_i^{(j)}:=\left(\prod_{a=1}^j r^{(a)}_i\right),\quad \tau_i^j:=\frac{\tau_i}{L_i^{(j)}}.
\end{align*}
(where the numbers $\ell^{(j)}$, $r^(j)$ are defined as long as we can iterate Proposition \ref{prop: iteration lemma})
Assume that $i$ is sufficiently large so that $\tau_i<\varepsilon_\ast$ and let $m_i-1$ be the largest integer $j$ such that $\tau_i^j< \varepsilon_\ast$. Then we can apply the proposition iteratively $m_i$ times.
In fact, as long as $\tau_i^j<\varepsilon_\ast$, at any step of the iteration the first four assumptions in Proposition \ref{prop: iteration lemma} are easily seen to be satisfied. Regarding Assumption 5, we note that if $k_i$ is sufficiently large so that the constants $\delta_k$ in \eqref{eq: localized control on second fundamental form} are smaller than $\frac{\varepsilon_\ast}{2E'}$, then for $j\leq m_i$
\begin{align*}
    (\tau_i^j)^4\log(\tau_i^j)^{-1}\int_{B_1(0)}\lvert \sff^{v_i^{(j)}}\rvert^4d\vol_{g_{v_i^{(j)}}}\leq & (L_i^{(j)})^{-2}\tau_i^4\log \tau_i^{-1}\int_{B_{R_i^{(j)}}(0)}\lvert \sff^{v_i}\rvert^4d\vol_{g_{v_i}}\\
    \leq & (L_i^{(j)})^{-2}\frac{\varepsilon_\ast}{2E'}\int_{{B_{R_i^{(j)}}(0)}}d\vol_{g_{v_i}}=\frac{\varepsilon_\ast}{2E'}\int_{B_1(0)}d\vol_{g_{v_i^{(j)}}}\leq \varepsilon_\ast.
\end{align*}
Observe that by choice of $m_i$, $v_i^{(m_i)}$ satisfies
\begin{align*}
    \int_{B_1(0)}\lvert\sff^{v_i^{(m_i)}}\rvert^4d\vol_{g_{v_i^{(m_i)}}}\leq \frac{\varepsilon_\ast}{(\tau_i^{(m_i)})^4\log\tau_i^{-1}}\leq \frac{1}{\varepsilon_\ast^3\log\tau_i^{-1}}.
\end{align*}
Therefore for $i$ sufficiently large we have that
\begin{align*}
    \int_{B_1(0)}\lvert\sff^{v_i^{(m_i)}}\rvert^4d\vol_{g_{v_i^{(m_i)}}}\leq \varepsilon_{E', K', \delta_0} \text{ and }r_iL_i^{(m_i)}\leq \varepsilon_{E', K', \delta_0},
\end{align*}
where $\varepsilon_{E', K', \delta_0}$ is the constant in Lemma \ref{lem: appendixB-LemmaA.1'} corresponding to $E', K'$ and $\delta_0$.
Then by Lemma \ref{lem: appendixB-LemmaA.1'}, $\pi_{\Pi^{(m_i)}}\circ v_i^{(m_i)}$ is a diffeomorphism from $\overline{B_{s(K'(A))^2}(0)}$ onto its image. Then for such $i$ we have
\begin{align*}
    n\left(v_i^{(m_i)}, B_{s(K')^2}(0), B_{\eta(K')}^{\Pi^{(m_i)}}(0)\right)=1.
\end{align*}
Now by Proposition \ref{prop: iteration lemma} we have
\begin{align*}
    n\left(v_i, B_{s(K)^2}(0), B_{\eta(K)}^{\Pi}\right)=n\left(v_i^{(1)}, B_{s(K')^2}(0), B_{\eta(K')}^{\Pi^{(1)}}(0)\right)
\end{align*}
and for any $j\in \{2,...,m_i\}$
\begin{align*}
    n\left(v_i^{(j-1)}, B_{s(K')^2}(0), B_{\eta(K')}^{\Pi^{(j-1)}}(0)\right)=n\left(v_i^{(j)}, B_{s(K')^2}(0), B_{\eta(K')}^{\Pi^{(j)}}(0)\right),
\end{align*}
so that for $i$ sufficiently large
\begin{align}\label{eq: mult-one-final}
    n\left(v_i, B_{s(K)^2}(0), B_{\eta(K)}^{\Pi}\right)=1.
\end{align}
Rescaling the domain by a factor $s(K)^2$ and the target by a factor $\lambda^{-1}$, we see that \eqref{eq: obtaining-N-in-the-limit} together with \eqref{eq: mult-one-final} imply that $N(x)=1$.
\end{proof}
Combining Theorem \ref{thm: convergence to H-pv} and Theorem \ref{thm: multiplicity-one}, we see that the map $u$ satisfies the following property: for almost any $\omega\subset \Sigma$, for any vector field $X$ in $\mathcal{M}$ supported away from $u(\partial \omega)$ there holds
    \begin{align}\label{eq: weak form H-CMC}
        \sum_{i}^2\int_\omega \langle \partial_i u,DX\partial_iu\rangle dx=H\int_\omega X\cdot (\partial_{x_1}u\times\partial_{x_2}u)=H\int_\omega u^\ast \alpha_X,
    \end{align}
    where $\alpha_X=\ast X^\flat$.


\begin{thm}\label{thm: smoothness-and-equation}
    The map $u$ is smooth, weakly conformal and satisfies
    \begin{align}\label{eq: CMC-eq-1}
        \operatorname{tr}_g\nabla du=H\nu,
    \end{align}
    where $\nu$ is the unit vector field in $u^\ast T\mathcal{M}$ normal to $u(\Sigma)$ induced by $\partial_{x_1}u\times\partial_{x_2}u$ (where $\frac{\partial}{\partial x_1}$, $\frac{\partial}{\partial x_2}$ is a local oriented frame of $T\Sigma$). Equivalently, in conformal coordinates,
    \begin{align*}
        \Delta u+A(\nabla u,\nabla u)=H\,\partial_x u\times\partial_y u.
    \end{align*}
\end{thm}
\begin{proof}
    For any open set $U\subset\Sigma$, let $\textbf{v}_U$ be the varifold induced by $u\vert_U$. Note that for almost any open $\omega\subset\Sigma$, $\textbf{v}_\omega$ is an integral varifold  with  generalized mean curvature in  $\mathcal{M}\smallsetminus u(\partial \omega)$ bounded by $H$.\\
\\
    \textbf{Step 1}: \textit{\underline{Claim}: $u$ is continuous and condition \eqref{eq: weak form H-CMC} holds for any domain $\omega\subset\Sigma$.}\\

    $u$ has a continuous representative by Lemma \ref{lem: limit of measures}.
    For any open $\omega\subset \Sigma$, let $X$ be a vector field on $\mathcal{M}$ supported away from $u(\partial \omega)$. Let $\rho$ be a function in $C^\infty_c(\omega)$ equal to $1$ on $u^{-1}(\operatorname{supp(X)})\cap\omega$. As \eqref{eq: weak form H-CMC} holds for $\{x\in \omega\vert\rho(x)>\lambda\}$ for a.e. $\lambda\in (0,1)$, we conclude that it also holds for $\omega$.\\
    Now recall the definition of the set $\mathcal{G}'$ at the beginning of Section 3.\\
    \\
    \textbf{Step 2}: \textit{\underline{Claim:} $u$ is smooth in a neighborhood of $\mathcal{G}'$}\\
    
    Set $\mathcal{B}':=\Sigma\smallsetminus \mathcal{G}'$, $\mathcal{B}:=u^{-1}(u(\mathcal{B}'))$ and $\mathcal{G}:=\Sigma\smallsetminus\mathcal{B}$.
    
Let $x\in \mathcal{G}'$. We work in a conformal chart centered at $x$. As $x$ is a Lebesgue point for $\lvert\nabla u\rvert^2$,
\begin{align*}
    \frac{1}{2}\int_{B_r}\lvert\nabla u\rvert^2=\pi s^2+o(r^2),
\end{align*}
where $s=\lvert\partial_1u\rvert(0)r$. Moreover for arbitrarily small radii $r>0$ we have
\begin{align*}
    u(ry)=u(0)+du(0)[ry]+o(r)\text{ for }\lvert y\rvert=1
\end{align*}
by Lemma A.4 in \cite{PR-Regularity}. Thus for $r$ small enough we can apply Allard's regularity theorem to the varifold $\textbf{v}_{B_r}$ (which has generalized mean curvature bounded in $L^\infty$ as a varifold in $\mathbb{R}^Q$) in the ball $B_{(1-\delta)s}(u(0))$, where $\delta$ is chosen to be sufficiently small. Allard's result implies that for some $\theta>0$, the varifold $\textbf{v}_{B_r}$ agrees--- up to rigid motions--- in $B_\theta(u(0))$ with the graph $S$ of a $C^{1,\alpha}$ function $f: \mathbb{R}^2\to\mathbb{R}^{Q-2}$, with multiplicity one.
Then for some $r'\in (0,\theta)$ $\lvert \mathbf{v}_{B_{r'}}\rvert$ is supported in $S$. Let $\tilde{\mathcal{G}}:=\mathcal{G}\cap\overline{B_{r'}}$. Then $u(y)\in S$ for any $y\in \tilde{\mathcal{G}}$. Considering $\operatorname{dist}(u,S)$ we conclude that $u(\overline{B_{r'}^2})\subset S$. Thus $u\vert_{B_{r'}}$ factors as $(\operatorname{id}\times f)\circ \Psi$ for some $\Psi\in C^0\cap W^{1,2}(\overline{B_{r'}},\mathbb{R}^2)$.  By the chain rule, any point $y\in \tilde{\mathcal{G}}$ is Lebesgue for $d\Psi$ and $d\Phi(y)$ is invertible.
As $\textbf{v}_{B_{r'}}$ has multiplicity $1$, $u$ is injective on $\tilde{\mathcal{G}}$. Thus for any $y\in\tilde{\mathcal{G}}$, $u(y)\notin u(\overline{B_{r'}}\smallsetminus\{y\})$ (as if $u(y')=u(y)$ for $y'\in \overline{B}_{r'}$, then $y'\in \tilde{\mathcal{G}}$), and the same holds for $\Psi$.
Therefore, for $y\in\tilde{\mathcal{G}}$ close to $0$, $\Psi\vert_{\partial B_{r'}(0)}-\Psi(y)$ and $\Psi\vert_{\partial B_t(y)}-\Psi(y)$ (for $t$ small) induce the same element of $\pi_1(\mathbb{R}^2\smallsetminus\{0\})$. If $y$ is sufficiently close to $0$, the first curve is homotopic (in $\mathbb{R}^2\smallsetminus\{0\}$) to $\Psi\vert_{\partial B_{r'}(0)}-\Psi(0)$, while the second is homotopic to $d\Psi(y)$ (again by Lemma A.4 in \cite{PR-Regularity}, as $y\in \tilde{\mathcal{G}}$). Thus $d\Psi$ is always orientation preserving or always orientation reversing on $\tilde{\mathcal{G}}$ near $0$. Thus $u$ in local conformal coordinates for $S$, solves the Cauchy-Riemann equations (up to conjugation) near $0$, so that $u$ is smooth near $0$.
This concludes the proof of the claim.\\
In particular, it follows from \eqref{eq: weak form H-CMC} that $u$ solves \eqref{eq: CMC-eq-1} on $\mathcal{G}'$.

Next we show that $u$ solves the equation on $\Sigma$.
Let $y\in \Sigma$ and pick a conformal chart $U\to B_1(0)$ centered at $y$, such that $u(U)$ lies in a coordinate chart for $\mathcal{M}$. Call $\{x^1,x^2, x^3\}$ the coordinates and let $u^i=x^i\circ u$. Write $e_k=\frac{\partial}{\partial x_k}$. It will be sufficient to show that for any $f\in C_c^\infty$,
\begin{align}\label{eq: weak-form-equation-ek}
    \int_{B_1(0)}\langle\nabla(f e_k), du\rangle=H\int_{B_1(0)}fe_k\cdot \partial_{x_1}u\times\partial_{x_2}u.
\end{align}
Indeed, if \eqref{eq: weak-form-equation-ek} holds, then $u$ is a weak solution of the system
\begin{align}\label{eq: PDE-system-CMC}
    -\partial_i(g_{jk}(u)\partial_iu^j)+&\Gamma_{pk}^a(u)g_{aq}(u)\partial_iu^p\partial_i u^q\\\nonumber &=Hg_{kl}(u) Q^l_{pq}(u)\partial_1u^p\partial_2^q u\quad\text{ for any }k\in \{1,...,Q\}.
\end{align}
Here $Q^j_{pq}$ denote the coefficients of the $(2,1)$-tensor on $\mathcal{M}$ defined by $\left(X\times Y\right)^j=Q^j_{pq}X^pY^q$, where $\times$ is the vector product on $T\mathcal{M}$. 
The smoothness then follows from elliptic regularity (see Proposition A.1 in \cite{Pigati-FB}), so that $u$ is a classical solution of \eqref{eq: PDE-system-CMC} and thus of \eqref{eq: CMC-eq-1}.\\
\\
\textbf{Step 3}: \textit{\underline{Claim}: \eqref{eq: weak-form-equation-ek} is satisfied}.\\

By the layer-cake formula and the coarea formula,
\eqref{eq: weak-form-equation-ek} is equivalent to 
\begin{align*}
    \int_0^\infty\left(-\int_{\partial\{f>\lambda\}}\langle (e_k\circ u),\partial_\nu u\rangle d\mathcal H^1+\int_{\{f>\lambda\}}\langle\nabla(e_k\circ u), du\rangle dx-H\int_{\{f>\lambda\}}e_k \cdot \partial_{x_1}u\times\partial_{x_2}u\, dx\right)d\lambda=0.
\end{align*}
We claim that for a.e. $\lambda$, the expression in brackets in the previous equation is equal to zero.

To see the claim, let $\lambda>0$ be a regular value of $f$ such that on $\partial\{f>\lambda\}$, $du$ vanishes $\mathcal{H}^1$-a.e. on $\mathcal{B}$. Set $\omega:=\{f>\lambda\}$.
For $\varepsilon>0$, let $\mathcal{B}_\varepsilon$ be the closed $\varepsilon$-neighborhood of $\mathcal{B}$ in $B_1(0)$.
Fix $\varepsilon>0$ and let $\rho$ be a smooth function on $\mathcal{M}$, vanishing on a neighborhood of $u(\partial\omega\cap\mathcal{B}_\varepsilon)$ and equal to $1$ outside of a slightly larger set. Then $u$ is a smooth immersion on a neighborhood of $S\cap\partial\omega$, where $S=\operatorname{supp}(\rho\circ u)$.
As in the proof of Proposition 7.7 in \cite{Pigati-FB}, one can reduce the proof to the case where $S\cap \partial\omega$ is covered by finitely many curves $\gamma_j$ in $B_1(0)\cap\mathcal{G}$, with endpoints in $B_1(0)\smallsetminus S$ and images $\Gamma_j=u(\gamma_j)$ transverse to each other (meaning also self-transverse).
Let $\chi$ be a smooth function from $[0,\infty)$ to $[0,1]$ such that $\chi(x)=0$ for $x\leq\frac{1}{2}$ and $\chi(x)=1$ for $x\geq 1$. Set $\chi_\eta:=\chi\left(\frac{\operatorname{dist}(\cdot,\Gamma)}{\eta}\right)$, where $\Gamma:=\bigcup_j \Gamma_j$, and let $U_r$ denote the $r$-neighborhood of $\bigcup_j\gamma_j$. Then one can show---taking advantage of the conformality of $u$ and the fact that the curves $\Gamma_j$ are transverse--- that
\begin{align*}
    \lim_{\eta\to 0}\int_{\omega\cap U_r}\rho(u)\langle e_k\circ u\otimes d(\chi_\eta\circ u), du\rangle=-\sum_j\int_{\gamma_j}\langle (\rho e_k)(u), \partial_\nu u\rangle d\mathcal{H}^1 +O(r),
\end{align*}
where $\nu$ is the outward unit normal for $\omega$.

Note that $u(\mathcal{B})\cap\Gamma=\emptyset$, hence $\chi_\eta=1$ near $u(\mathcal{B})$ for $\eta$ small, and in such case $\operatorname{supp}((1-\chi_\eta)\circ u)\subset\mathcal{G}$. Since $u$ satisfies \eqref{eq: CMC-eq-1} on $\mathcal{G}$, we have that
\begin{align*}
    \int_{\omega\smallsetminus U_r}\langle \nabla((\rho(1-\chi_\eta)e_k)\circ u), du\rangle=H\int_{\omega\smallsetminus U_r}\rho(1-\chi_\eta)e_k(u)\cdot\partial_{x_1}u\times\partial_{x_2}u.
\end{align*}
Therefore, integration by parts yields
\begin{align*}
    &\int_{\omega\smallsetminus U_r}\rho\circ u\langle (e_k\circ u)\otimes d(\chi_\eta\circ u), du\rangle\\
   = &\int_{\omega\smallsetminus U_r}(1-\chi_\eta)\circ u\langle (e_k\circ u)\otimes d(\rho\circ u), du\rangle+\int_{\omega\smallsetminus U_r}\rho(1-\chi_\eta)\circ u\langle \nabla (e_k\circ u), du\rangle\\
    &+\int_{\omega\cap \partial U_r}\rho(1-\chi_\eta)\circ u\langle (e_k\circ u), \partial_{\nu_{U_r}} u\rangle d\mathcal{H}^1-H\int_{\omega\smallsetminus U_r}\rho(1-\chi_\eta)e_k\cdot\partial_{x_1}u\times\partial_{x_2}u.
\end{align*}
Here $\nu_{U_r}$ is the unit normal pointing outside $U_r$.
As $(1-\chi_\eta)\circ u\to 0$ a.e. in $\omega\smallsetminus U_r$ and on $\partial U_r$ (for a.e. small $r$), the right hand side tends to zero as $\eta\to 0$.
By \eqref{eq: weak form H-CMC}, we have
\begin{align*}
    \int_\omega\langle\nabla((\rho\chi_\eta
e_k)\circ u), du\rangle=H\int_\omega\rho\chi_\eta
e_k\cdot\partial_{x_1}u\times\partial_{x_2}u,
\end{align*}
since $\rho\chi_\eta$ vanishes in a neighborhood of $u(\partial\omega)$.
Thus the previous computations imply (letting $r\to 0$ after taking the limit $\eta\to 0$)
\begin{align}\label{eq: last-step-weak-form-and-regularity}
    &\sum_j\int_{\gamma_j}\rho\circ u\langle(e_k\circ u),\partial_\nu u\rangle=-\lim_{\eta\to 0}\int_\omega\rho\circ u\langle(e_k\circ u)\otimes d(\chi_\eta\circ u), du\rangle\\
    \nonumber
    =&\int_\omega\langle (e_k\circ u)\otimes d(\rho\circ u), du\rangle+\int_\omega \rho\circ u\langle\nabla (e_k\circ u), du\rangle\\
    \nonumber&-H\int_\omega\rho
e_k(u)\cdot\partial_{x_1}u\times\partial_{x_2}u.
\end{align}
Now one shows exactly as in the last part of the proof of Proposition 7.7 in \cite{Pigati-FB} that $\mathcal{H}^1(u(\partial\omega\cap\mathcal{B}_\varepsilon))\to 0$ as $\varepsilon\to 0$ and that--- writing $\rho_\varepsilon$ instead of $\rho$ to mark the dependence on $\varepsilon$--- $\rho_\varepsilon\circ u\to 1$ a.e. in $\mathcal{G}$ and $\int_\omega\lvert d\rho_\varepsilon\rvert(u)\lvert du\rvert^2\to 0$ as $\varepsilon\to 0$. Therefore the desired statement follows from \eqref{eq: last-step-weak-form-and-regularity}.
\end{proof}

\begin{corollary}\label{cor: HW}
    The map $u$ is a branched $H$-CMC immersion on the components of $\Sigma$ where it is not constant, meaning that the  points $x\in \Sigma$ with $\nabla u(x)=0$ are isolated in any such component.
\end{corollary}
\begin{proof}
    By Theorem \ref{thm: smoothness-and-equation}, it is enough to show that if $u$ is not constant on a component $\Sigma_i$ of $\Sigma$, then the points $x\in \Sigma_i$ such that $\nabla u(x)=0$ are isolated. Let $x$ be such a point, let $\varphi: U\to B_1(0)$ be a conformal chart centered at $x$. Recall that $u$ (regarded as a function on $B_1(0)$) satisfies \eqref{eq: PDE-system-CMC} and is smooth, and so in particular (taking $U$ smaller if necessary)
    \begin{align}\label{eq: assumption - HW}
        \lvert\Delta u\rvert\leq K\lvert\nabla u\rvert\text{ in }B_1(0),
    \end{align}
    for some constant $K$ depending on $x$, $\varphi$ and $u$. We now apply Theorem 1 and 2 of \cite{Hartman-Wintner}\footnote{The results of \cite{Hartman-Wintner} are stated for scalar functions, but the proofs also work for vector valued functions satisfying \eqref{eq: assumption - HW}.} to $u$.
    Let
    \begin{align}\label{eq: def-vanishing-order}
        n=\sup\left\{k\in\mathbb{N}\vert u(z)-u(0)=o(\lvert z\rvert^k)\text{ as }z\to 0\right\}.
    \end{align}
    Note that this is finite by the cited Theorem 2 (as $u$ is not constant in the component of $x$).
    Theorem 1 now implies that the limits $L_1:=\lim_{z\to 0}u_z\cdot z^{-n}$ and $L_2:=\lim_{z\to 0}u_{\overline z}\cdot \overline{z}^{-n}$ exist.
    If there is a sequence of points $x_n\in B_1(0)$ with $\nabla u(x_n)=0$ and $x_n\to 0$, then we would have $L_1=L_2=0$.
    But then a Taylor expansion would show that $n$ is not maximal in \eqref{eq: def-vanishing-order}, leading to a contradiction.    
\end{proof}
\begin{rmk}
    Let $z_0$ be a branch point of $u$.
    Then, according to \cite[Corollary 7.2]{Gulliver} and \cite[Satz 2]{Kaul}, there exist coordinates $x,y$ on a neighborhood of $z_0$, and local coordinates $(y_1, y_2, y_3)$ on $\mathcal M$, such that
    \begin{align*}
        y_1\circ u(x,y)+iy_2\circ u(x,y)=(x+iy)^m,\quad y_3\circ u(x,y)=\psi(x+iy)
    \end{align*}
    for some integer $m\geq2$, where $\psi$ is of class $C^2$ and $D^k\psi(z)=O(\lvert z\rvert^{m+1-k})$ for $0\leq k\leq 2$.

In particular, if $\Omega$ is a sufficiently small neighborhood of $z_0$, then the varifold $u_\ast[\Omega]$ has a unique tangent cone at $u(z_0)$, which is a plane $\Pi$ with multiplicity $m$, and the Gauss map $\nu$ extends continuously across $z_0$ with $\nu(z_0)$ orthogonal to $\Pi$.
\end{rmk}
We can now combine the results of Sections \ref{sec: existence of cmc} and \ref{sec: regularity} to obtain a proof of Theorem \ref{thm: main-theorem}.
\begin{proof}[Proof of Theorem \ref{thm: main-theorem}]
    By Theorem \ref{thm: convergence to H-pv}, there exists a parametrized $H$-CMC varifold $(\Sigma', u,N)$ such that the varifolds induced by the immersions $u_k$ converge to $(\Sigma', u,N)$, for a conformal map $u\in W^{1,2}(\Sigma')$. By Theorem \ref{thm: multiplicity-one}, $N\equiv 1$, so by Theorem \ref{thm: smoothness-and-equation}, $u$ is smooth and satisfies \eqref{eq: CMC-eq-1}. Thus, by Corollary \ref{cor: HW}, $u$ is a branched $H$-CMC immersion on every component on which it is not constant.
\end{proof}

\section{Existence of CMC immersions in closed 3-manifolds}\label{sec: existence of torus}
In this section we apply the previous convergence results to show that in any Riemannian manifold $(\mathcal{M}^3,g)$ with Heegard genus $h$, one can find a $H$-CMC immersion, for a.e. $H>0$ and $H=0$. 
To this end, we will construct a sequence of maps as in Theorem \ref{thm: main-theorem},
applying min-max methods to the sweep-outs generated by the genus $h$ Heegard splitting of $\mathcal{M}$. 

\subsection{Min-max setup}
Let $(\mathcal{M}, g)$ be a closed, oriented manifold of Heegard genus $h$. Recall that there exists a splitting $\mathcal{M}=N_1\cup_\varphi N_2$, where $N_1$, $N_2$ are handlebodies of genus $h$, and $\varphi: \partial N_1\to\partial N_2$ is a diffeomorphism. Recall also that each handlebody $N_i$ can be identified with a quotient of $[0,1]\times\Sigma$ (see Section 2). One can then glue the two quotients $([0,1]\times \Sigma)/\sim$ identifying the two copies of $\{1\}\times\Sigma$ through the map $\varphi$, so that (after reparametrization in the $t$ variable) $\mathcal{M}$ itself can be identified with a quotient of $\Sigma\times [0,1]$.

We say that a continuous map $\gamma: [0,1]\times \Sigma\to \mathcal{M}$ is a degree one sweep-out of $\mathcal{M}$ if it descends to a map on $\mathcal{M}$ (regarded as a quotient of $[0,1]\times \Sigma$) which has degree one. Note that this implies in particular that $\gamma(0,\cdot)$ and $\gamma(1,\cdot)$ are graphs.

Consider the following family of degree one sweep-outs of $\mathcal{M}$. Recall that we denoted by $\mathfrak{M}$ the space of $W^{2,4}$ immersions of $\Sigma$ in $\mathcal{M}$.
\begin{definition}\label{def: sweep-out-family}
    \begin{align*}
        \mathcal{P}_h=\{\gamma\in C^0((0,1), \mathfrak{M})\cap C^0([0,1],C^0(\Sigma,\mathcal{M}))\vert \gamma \text{ is a degree one sweep-out of }\mathcal{M}\}.
    \end{align*}
\end{definition}

In the above setting, the interior of $N_1$ (resp. $N_2$) is foliated by surfaces $\{\Sigma_t\}_{t\in(0,1/2)}$ (resp. $\{\Sigma_t\}_{t\in(1/2,1)}$) which are all isotopic to $\Sigma$. When we turn this into the language of mappings from $\Sigma\times[0,1]\rightarrow\mathcal{M}$, we see that the set $\mathcal{P}_h$ is not empty.\\
Note also that as $t\to 0,1$, the images of $\gamma(t)$ degenerate to graphs.
In what follows, when we consider sweepouts by genus-$h$ surfaces, where $h$ is the Heegaard genus of $\mathcal{M}$, we will simply write $\mathcal{P}$ instead of $\mathcal{P}_h$.

Let $\gamma\in \mathcal{P}_h$.
    For each $t\in(0,1)$, there is a natural map $f_{\gamma,t}\in\mathcal{E}(\gamma(t))$ defined as
    \begin{align*}
        f_{\gamma,t}(s,x)=\gamma(ts)(x).
    \end{align*}

Having defined the space of sweep-outs, one would like to consider the width
\begin{align*}
    \inf_{\gamma\in\mathcal{P}}\max_{t\in[0,1]}E_{H,\sigma}(\gamma(t),f_{\gamma,t}).
\end{align*}
However, for $\sigma>0$ this quantity is equal to infinity: when $t\rightarrow0$ or $t\rightarrow1$, the term $\int_{\Sigma}\vert\sff^{\gamma(t)}\vert^4\,d\vol_{\gamma(t)}$ must blow up.
The reason is that, by the compactness results of Langer \cite{LangerCompactness} and Breuning \cite{Breuning}, a uniform bound on $\int_{\Sigma}|\sff^{u}|^{4}\,d\vol_{u}$ and $\operatorname{Area}(u)$ prevents degeneration: it forces $u$ to remain in a $C^1$-precompact subset of $\mathfrak{M}$ (up to reparametrization). Consequently, the family $\gamma(t)$ cannot degenerate to a graph as $t\to 0,1$.
This is because by the compactness results of Langer \cite{LangerCompactness} and Breuning \cite{Breuning}, given the control on $\int_{\Sigma}\vert\sff^{u}\vert^4\,d\vol_{u}$, the map $u$ must stay in a bounded region in $\mathfrak{M}$, and cannot degenerate to a graph.

In order to overcome this problem, we will work with sub-intervals of $[0,1]$. For $H>0$, set
\begin{align*}
    \omega_H=\inf_{\gamma\in\mathcal{P}}\max_{t\in[0,1]}E_{H,0}(\gamma(t),f_{\gamma,t})=\inf_{\gamma\in\mathcal{P}}\max_{t\in[0,1]}(\operatorname{Area}(\gamma(t))+H\vol(f_{\gamma,t})).
\end{align*}
Note that $\omega_H<\infty$, as can be seen by considering any smooth sweep-out. 
\begin{proposition}\label{prop: non-trivial of min-max}
    For any $H>0$, we have
    \begin{align*}
       \omega_H>H\vol_g(\mathcal{M}). 
    \end{align*}
    For $H=0$, there holds
    $\omega_0>C_\mathcal{M}$
    for some constant $C_\mathcal{M}>0$ depending only on $\mathcal{M}$.
\end{proposition}
\begin{proof}
First we show that for any $\gamma\in C^1([0,1]\times\Sigma,\mathcal{M})$ descending to $\mathcal{M}$ (regarded as a quotient of $[0,1]\times\Sigma$) and inducing a degree one map from $\mathcal{M}$ to itself, for any $\varepsilon>0$ small enough, there is $t_\varepsilon\in (0,1)$ with
\begin{align*}
         E_H(\gamma(t_\varepsilon), f_{\gamma,t_\varepsilon})>H\vol_g(\mathcal{M}).
    \end{align*}
or--- if $H=0$--- $\operatorname{Area}(\gamma(t_\varepsilon))\geq C$, where the constant $C$ depends only on $\mathcal{M}$.
    Let $\gamma$ be such a map and let $\varepsilon\in(0, \vol_g(\mathcal{M}))$ to be determined later.
    Let $t_\varepsilon\in (0,1)$ such that
    \begin{align*}
        \vol(f_{\gamma,t_\varepsilon})=\vol_g(\mathcal{M})-\varepsilon.
    \end{align*}
    For each point $x\in\Sigma$, assume $e_1,e_2$ is an oriented orthonormal basis of $T_x\Sigma$.
    Let $\theta$ be the integer-valued degree function defined for $p \in \mathcal{M}$ by:
\begin{align*}
    \theta(p) := \sum_{z \in \gamma^{-1}(p) \cap ([t_\varepsilon,1] \times \Sigma)} \operatorname{sgn}\left( d\vol_g(\nabla_{e_1}\gamma, \nabla_{e_2}\gamma, \de_t\gamma) \right).
\end{align*}
    Let 
    \begin{align*}
        A=\left\{p\in \operatorname{Im}(\gamma\vert_{[t_\varepsilon,1]})\big\vert\theta(p)>0\right\},
    \end{align*}
    We distinguish two cases: 
    first assume that $\vol_g(A)\leq \frac{1}{2}\vol_g(\mathcal{M})$.\\
    For any $n\in \mathbb{Z}_{>0}$ set
    \begin{align*}
        A_n=\left\{p\in  \operatorname{Im}(\gamma\vert_{[t_\varepsilon,1]})\big\vert\theta(p)\geq n\right\},
    \end{align*}
    and for any $n\in \mathbb{Z}_{<0}$ set
    \begin{align*}
        A_n=\left\{p\in \operatorname{Im}(\gamma\vert_{[t_\varepsilon,1]})\big\vert\theta(p)\leq n\right\}.
    \end{align*}
    Set also $A_0=\emptyset$.
    
    In the following, we denote by $\gamma_\#$ the push-forward operator for currents induced by $\gamma$, and denote the mass of currents by $\mathbb{M}$. We have
    \begin{align}\label{eq: estimate_area_gamma_t}
        \operatorname{Area}(\gamma(t_\varepsilon))\geq \mathbb{M}(\gamma_\#\partial\jump{[t_\varepsilon,1]\times\Sigma})=\mathbb{M}(\partial(\gamma_\#\jump{[t_\varepsilon,1]\times\Sigma})).
    \end{align}
    According to Corollary 27.8 in \cite{SimonGMT}, the sets $A_n$ have finite perimeter, there holds
    \begin{align*}
        \gamma_\#\jump{[t_\varepsilon,1]\times \Sigma}=\sum_{n\in \mathbb{Z}}\operatorname{sgn}(n)\jump{A_n},
    \end{align*}
    \begin{align*}
        \partial( \gamma_\#\jump{[t_\varepsilon,1]\times \Sigma})=\sum_{n\in \mathbb{Z}}\operatorname{sgn}(n)\partial\jump{{A_n}},
    \end{align*}
    and 
    \begin{align*}
        \mathbb{M}( \partial \gamma_\#\jump{[t_\varepsilon,1]\times \Sigma})=\sum_{n\in \mathbb{Z}}\operatorname{Area}(\partial A_n),
    \end{align*}
    where $\operatorname{Area}(\partial A_n)$ denotes the perimeter of the set $A_n$.

    Note that, by assumption, $\vol_g(A_n)\leq \frac{1}{2}\vol_g(\mathcal{M})$ for any $n\in \mathbb{N}$. Thus by the Levy-Gromov isoperimetric inequality (Lemma \ref{lem: Levy-Gromov}), there holds
    \begin{align*}
        \operatorname{Area}(\partial A_n)\geq  C\vol_g^\frac{2}{3}(A_n)
    \end{align*}
    for some constant $C$ depending only on $\mathcal{M}$. We deduce that
    \begin{align}\label{eq: chain-inequalities-area-gamma}
        \operatorname{Area}(\gamma(t_\varepsilon))\geq\sum_{n\in \mathbb{N}}\operatorname{Area}(\partial A_n)\geq C\sum_{n\in \mathbb{N}}\vol_g^\frac{2}{3}(A_n)
        \geq C\left(\sum_{n\in \mathbb{N}}\vol_g(A_n)\right)^\frac{2}{3}\geq C\varepsilon^\frac{2}{3}.
    \end{align}
    For the last step we used that, by Corollary 27.8 in \cite{SimonGMT},
    \begin{align*}
        \sum_{n\in \mathbb{N}}\vol_g(A_n)=&\mathbb{M}(\gamma_\#\jump{\{z\in [t_\varepsilon,1]\times\Sigma\vert d\vol_g(\nabla_{e_1}\gamma,\nabla_{e_2}\gamma,\de_t\gamma)(z)>0\}})\\=&\int_{\{z\in [t_\varepsilon,1]\times\Sigma\vert d\vol_g(\nabla_{e_1}\gamma,\nabla_{e_2}\gamma,\de_t\gamma)(z)>0\}}\gamma^\ast\,d\vol_g\geq \varepsilon,
    \end{align*}
    where the last step holds by choice of $t_\varepsilon$.

    In the second case we have $\vol_g(A)>\frac{1}{2}\vol_g(\mathcal{M})$. We consider the set
    \begin{align*}
    B=\left\{p\in \operatorname{Im}(\gamma\vert_{[t_\varepsilon,1]})\big\vert\theta(p)<0\right\},
    \end{align*}
    which in this case satisfies $\vol_g (B)\leq \frac{1}{2}\vol_g(\mathcal{M})$. Since
    \begin{align*}
        \int_{ [t_\varepsilon,1]\times \Sigma}\gamma^\ast\,d\vol_g=\varepsilon,
    \end{align*}
    $\vol_g(A)>\frac{1}{2}\vol_g(\mathcal{M})$ implies that
    \begin{align*}
        -\int_{\{z\in [t_\varepsilon,1]\times \Sigma\cap \gamma^{-1}(B)\}}\gamma^\ast\,d\vol_g\geq \frac{1}{2}\vol_g(\mathcal{M})-\varepsilon.
    \end{align*}
    Note that now $\vol_g(A_{-n})\leq \frac{1}{2}\vol_g(\mathcal{M})$ for any $n\in \mathbb{N}$. As above, by the Levy-Gromov isoperimetric inequality,
    \begin{align*}
        \operatorname{Area}(\partial A_{-n})\geq C\vol_g^\frac{2}{3}(A_{-n})
    \end{align*}
    for any $n\in \mathbb{N}$,
    for some constant $C$ depending on $\mathcal{M}$.
    Thus, arguing as above (but taking the sum over the $A_n$ with negative indices in the first inequality in \ref{eq: chain-inequalities-area-gamma}), we obtain
    \begin{align*}
        \operatorname{Area}(\gamma(t_\varepsilon))\geq C\left(\sum_{n\in \mathbb{N}}\vol_g(A_{-n})\right)^\frac{2}{3}.
    \end{align*}
    Now
    \begin{align*}
        \sum_{n\in \mathbb{N}}\vol_g(A_{-n})=  -    \int_{\{z\in [t_\varepsilon,1]\times \Sigma\cap \gamma^{-1}(B)\}}\gamma^\ast\,d\vol_g\geq \frac{1}{2}\vol_g(\mathcal{M})-\varepsilon,
    \end{align*}
    therefore
    \begin{align*}
        \operatorname{Area}(\gamma(t_\varepsilon))\geq C\left(\frac{1}{2}\vol_g(\mathcal{M})-\varepsilon\right)^\frac{2}{3}.
    \end{align*}
    Thus, if $\varepsilon$ is small enough, in both cases we have
    \begin{align*}
        \operatorname{Area}(\gamma(t_\varepsilon))\geq C\varepsilon^\frac{2}{3},
    \end{align*}
    for some positive constant $C$ depending only on $g$.
    Then in both cases 
    \begin{align*}
        \operatorname{Area}(\gamma(t_\varepsilon))+H\vol_g(f_{\gamma,t_\varepsilon})\geq C\varepsilon^\frac{2}{3}+H(\vol_g(\mathcal{M})-\varepsilon).
    \end{align*}
    If $H=0$, choose $\varepsilon=\vol_g(\mathcal{M})/4$ to obtain
    \begin{align*}
        \operatorname{Area}(\gamma(t_\varepsilon))\geq C(\vol_g(\mathcal{M})/4)^\frac{2}{3}.
    \end{align*}
    If $H>0$, choosing $\varepsilon\leq\varepsilon_0:=\left(\frac{C}{2H}\right)^3$  we have
    \begin{align*}
         C\varepsilon^\frac{2}{3}+H(\vol_g(\mathcal{M})-\varepsilon)\geq H\vol_g(\mathcal{M})+\frac{C}{2}\varepsilon^\frac{2}{3}.
    \end{align*}
    Thus if $\varepsilon\leq\varepsilon_0$,
    \begin{align*}
        E_H(\gamma(t_\varepsilon), f_{\gamma, t_\varepsilon})\geq H\vol_g(\mathcal{M})+\frac{C}{2}\varepsilon^\frac{2}{3}.
    \end{align*}
    Now for a general $\gamma\in \mathcal{P}$ and $H\geq 0$, let $\varepsilon=\frac{\varepsilon_0}{2}$ (depending on $H$ and $C$), and let $t_\varepsilon$ be such that $\vol(f_{\gamma,t_\varepsilon})=\vol_g(\mathcal{M})-\varepsilon$.
    For any $\eta>0$, let $\gamma_\eta\in C^1([0,1]\times\Sigma)$ be a $C^1$-approximation of $\gamma$ as in the first part of the proof (i.e. a map descending to $\mathcal M$ and inducing a degree one map from $\mathcal{M}$ to itself), satisfying
    \begin{align*}
        \lvert E_H(\gamma(t_\varepsilon), f_{\gamma, t_\varepsilon})-E_H(\gamma_\eta(t_{\varepsilon}), f_{\gamma_\eta, t_{\varepsilon}})\rvert\leq& \eta,\\
        \lvert\vol(f_{\gamma_\eta,t_{\varepsilon}})-(\vol_g(\mathcal{M})-\varepsilon)\rvert\leq \eta.
    \end{align*}
    Such $\gamma_\eta$ can be constructed by mollifying $\gamma$ in $([-\eta,1+\eta]\times\Sigma, dt\times g_{\gamma(t_\varepsilon)})$ (where $\gamma$ is continuously extended on $[-\eta,0]\times \Sigma$, $[1, 1+\eta]\times \Sigma$, independently of $t$) and projecting on $\mathcal{M}$ via nearest point projection. 
    The fact that the enclosed volume can be approximated this way follows from Lemma \ref{lem: define of volume}.
    Then for $H>0$
    \begin{align*}
        E_H(\gamma(t_\varepsilon), f_{\gamma, t_\varepsilon})\geq&E_H(\gamma_\eta(t_{\varepsilon}), f_{\gamma_\eta, t_{\varepsilon}})-\eta\geq H\vol_g(\mathcal{M})+\frac{C}{2}(\varepsilon-\eta)^\frac{2}{3}-\eta,
    \end{align*}
    while for $H=0$
    \begin{align*}
        \operatorname{Area}(\gamma(t_\varepsilon))\geq \operatorname{Area}(\gamma_\eta(t_{\varepsilon}))-\eta\geq C(\vol_g(\mathcal{M})/4)^\frac{2}{3}-\eta.
    \end{align*}
    In both cases, by choosing $\eta$ sufficiently small, we obtain the desired result.
\end{proof}

In view of Proposition \ref{prop: non-trivial of min-max}, for a sweep-out $\gamma\in\mathcal{P}$, $E_{H,0}(\gamma(t), f_{\gamma,t})$ can be close to the width $\omega_H$ only for $t$ with $E_{H,\sigma}(\gamma(t),f_{\gamma,t})>H\vol_g(\mathcal{M})$, if $H>0$, or for $t$ with $E_{0,\sigma}(\gamma(t), f_{\gamma,t})$, if $H=0$.
Let $H_0>0$, we will study the existence of $H$-CMC surfaces for $H<H_0$. As $H_0$ is arbitrary, this impose no limitations.
From now on, assume that $H<H_0$.
Set
\begin{align*}
    \zeta:=\frac{1}{2}\left(\frac{\omega_{H_0}}{H_0}-\vol_g(\mathcal{M})\right)>0.
\end{align*}
For the case $H=0$, set $\zeta_0:=\frac{1}{2}C_\mathcal{M}$, where $C_\mathcal{M}$ is the constant from Proposition \ref{prop: non-trivial of min-max}.
\begin{definition}\label{def: good-intervals}
    Let $\gamma\in \mathcal{P}$. If $H>0$, set
\begin{align*}
    I_\gamma^H=\left\{t\in [0,1]\big\vert\, H\left(\vol_g(\mathcal{M})+\zeta\right)\leq \operatorname{Area}(\gamma(t))+H\vol(f_{\gamma,t})\right\}.
\end{align*}
If $H=0$, set
\begin{align*}
    I_\gamma^0=\left\{t\in [0,1]\big\vert \operatorname{Area}(\gamma(t))>\zeta_0\right\}.
\end{align*}
Note that--- just as in the proof of Lemma \ref{lem: properties-omega} below--- $H\mapsto \omega_H/H$ is non-increasing, thus by Proposition \ref{prop: non-trivial of min-max}, $I_\gamma^H$ is non-empty for any $H\geq 0$. We then define, for any $H\geq 0$,
\begin{align*}
    \omega_{H, \sigma}=\inf_{\gamma\in {\mathcal{P}}}\sup_{t\in I_\gamma^H} E_{H,\sigma}({\gamma}(t), f_{\gamma,t}).
\end{align*}
Note that
\begin{align*}
    H\vol_g(\mathcal{M)}\leq \omega_{H,\sigma}<\infty.
\end{align*}
The lower bound follows from the definition of $I_\gamma^H$, and finiteness can be verified by considering a smooth sweep-out.\\
The modified width still satisfies monotonicity properties  analogous to those in \cite[Proposition 3.2]{Cheng-Zhou}.
\begin{lemma}\label{lem: properties-omega}
    The following properties hold for $\omega_{H,\sigma}$.
    \begin{enumerate}
        \item For each $\sigma\geq0$, the function $H\mapsto\omega_{H,\sigma}/H$ is non-increasing (for $H>0$);
        \item For each $H\geq0$, the function $\sigma\mapsto\omega_{H,\sigma}$ is non-decreasing;
        \item For a.e. $H>0$, we can find a sequence $\sigma_n\rightarrow 0$ and $c>0$ such that 
        \begin{align}
            \label{eq: conclusion-Lem-Struwe-1} \sigma_n\log\sigma_n^{-1}\frac{\de\omega_{H,\sigma}}{\de \sigma}\bigg|_{\sigma=\sigma_n}\rightarrow0\text{ and }\frac{\de}{\de H}\left(-\frac{\omega_{H,\sigma_n}}{H}\right)\leq c.
        \end{align}
        \end{enumerate}
\end{lemma}
For convenience, let us set
\begin{align*}
    A_\sigma(u)=\operatorname{Area}(u)+\sigma^4\int_{\Sigma} \lvert \sff^u\rvert^4  d\vol_{g_u}\text{ for }u\in\mathfrak{M}.
\end{align*}
\begin{proof}
    \begin{enumerate}
        \item Let $H>H'>0$, then for any $\gamma\in \mathcal{P}$, for $t\in [0,1]$ we have
    \begin{align*}
        \frac{E_{H,\sigma}(\gamma(t), f_{\gamma,t})}{H}-\frac{E_{H',\sigma}(\gamma(t), f_{\gamma,t})}{H'}=\frac{H'-H}{HH'}A_\sigma(\gamma(t))\leq 0.
    \end{align*}
    Next for any $\delta>0$, there exist $\gamma\in \mathcal{P}$ such that
    \begin{align*}
        \max_{t\in I_{\gamma}^{H'}}E_{H',\sigma}(\gamma(t), f_{\gamma, t})\leq\omega_{H',\sigma}+\delta.
    \end{align*}
Note that if $t\in I_\gamma^H$, then
\begin{align*}
    \frac{1}{H'}(\operatorname{Area}(\gamma(t))+H'\vol(f_{\gamma,t}))\geq\frac{1}{H}(\operatorname{Area}(\gamma(t))+H\vol(f_{\gamma,t}))\geq\vol_g(\mathcal{M})+\zeta,
\end{align*}
therefore $I_\gamma^{H}\subset I_\gamma^{H'}$. Hence,
    \begin{align*}
        \frac{{\omega}_{H,\sigma}}{H}\leq \max_{t\in I_\gamma^H}\frac{E_{H,\sigma}(\gamma(t),f_{\gamma,t})}{H}\leq \max_{t\in I_\gamma^{H'}}\frac{E_{H',\sigma}(\gamma(t),f_{\gamma,t})}{H'}\leq \frac{{\omega}_{H',\sigma}}{H'}+\frac{\delta}{H'}.
    \end{align*}
    As $\delta$ is arbitrary, the statement follows.
    \item The second statement follows from a similar proof as (1) and the fact that for any $u\in \mathfrak{M}$, the function $\sigma\mapsto A_{\sigma}(u)$ is non-decreasing.
    \item By the previous point, for any $H\geq 0$, for a.e. $\sigma\geq 0$ the derivative $\frac{\de{\omega}_{H,\sigma}}{\partial \sigma}$ exists and is non-negative, and there holds
    \begin{align*}
        \omega_{H,1}-\omega_{H,0}\geq \int_0^1 \frac{\partial \omega_{H,\sigma}}{\partial \sigma} d\sigma.
    \end{align*}
    Integrate the above expression over $H\in[a,b]$ (with $b<H_0)$. Swapping the order of integration by Fubini's Theorem, we obtain 
    \begin{align}\label{eq: estimate-derivative-Struwe-1}
        \int_0^1\int_a^b\frac{\de\omega_{H,\sigma}}{\de \sigma}\,d\sigma dH\leq \int_a^b(\frac{\omega_{H,1}}{H}-\frac{\omega_{H,0}}{H})H dH\leq \frac{\omega_{a,1}}{a}\int_a^bH dH=\frac{\omega_{a,1}}{2a}(b^2-a^2)<\infty.
    \end{align}

    Hence, we can find a sequence $\sigma_n\rightarrow0$ such that
    \begin{align*}
        \sigma_n\log\sigma_n^{-1}\int_a^b\frac{\de\omega_{H,\sigma}}{\de\sigma}\bigg|_{\sigma=\sigma_n}dH\rightarrow0,
    \end{align*}
    otherwise the integral on the left side of \eqref{eq: estimate-derivative-Struwe-1} would be infinite.
    Thus the sequence of functions $f_n:=\sigma_n\log\sigma_n^{-1}\frac{\de\omega_{H,\sigma}}{\de\sigma}\big|_{\sigma=\sigma_n}$ converges to $0$ in $L^1([a,b])$. Then there exists a subsequence (not relabeled) such that $f_n\to 0$ almost everywhere, meaning that for a.e. $H\in [a,b]$
    \begin{align*}
        \lim_{n\to \infty}\sigma_n\log{\sigma_n}^{-1}\frac{\de\omega_{H,\sigma}}{\de\sigma}\bigg|_{\sigma=\sigma_n}=0.
    \end{align*}
    Fix such a sequence $\{\sigma_n\}_{n\in \mathbb{N}}$. Then for any $n\in \mathbb{N}$, by part $(1)$ the derivative $\frac{\partial}{\partial H}\left(-\frac{\omega_{H,\sigma_n}}{H}\right)$ exists for almost any $H>0$ and is non-negative. Hence by Fatou's Lemma
    \begin{align*}
        \int_a^b\liminf_{n\rightarrow\infty}\frac{\de}{\de H}\left(-\frac{\omega_{H,\sigma_n}}{H}\right)\,dH\leq\liminf_{n\rightarrow\infty}\int_a^b\frac{\de}{\de H}\left(-\frac{\omega_{H,\sigma_n}}{H}\right)\,dH\leq\frac{\omega_{a,1}}{a}-\frac{\omega_{b,0}}{b}<\infty.
    \end{align*}
    Thus for a.e. $H\in[a,b]$ we have
    \begin{align*}
        \liminf_{n\rightarrow\infty}\frac{\de}{\de H}\left(-\frac{\omega_{H,\sigma_n}}{H}\right)<\infty.
    \end{align*}
    By passing to a subsequence of $\{\sigma_n\}_{n\in \mathbb{N}}$ (possibly depending on $H$), we obtain the desired estimates.
    \end{enumerate}
\end{proof}
\begin{rmk}\label{rem: Struwe-1-H-0}
    If $H=0$, we observe that
    \begin{align*}
        \int_0^1\frac{\partial \omega_{0,\sigma}}{\partial \sigma}ds\leq \omega_{0,1}<\infty.
    \end{align*}
    Therefore arguing as in the proof of Lemma~\ref{lem: properties-omega} we obtain that there is a sequence $\sigma_n\to 0^+$ with
    \begin{align*}
        \lim_{n\to\infty}\sigma_n \log\sigma_n^{-1}\frac{\partial\omega_{0,\sigma}}{\partial\sigma}\bigg\vert_{\sigma=\sigma_n}=0.
    \end{align*}
\end{rmk}
In the next Lemma, we generalize Struwe's monotonicity trick to our two parameter case, to derive both the control on the area and the second fundamental form. This result corresponds to Lemma 3.3 in \cite{Cheng-Zhou}.
\begin{lemma}\label{lem: monotonicity-trick}
    For any $H,\sigma>0$ such that
    \begin{align}\label{eq: assumptions-good-sweep-outs}
        \frac{\de}{\de H}\left(-\frac{\omega_{H,\sigma}}{H}\right)\leq c\text{ and }\frac{\de\omega_{H,\sigma}}{\de\sigma}\leq\frac{\varepsilon}{\sigma\log\sigma^{-1}}
    \end{align}
    (for constants $c,\varepsilon>0$),
    there exists a sequence $\{\sigma_j\}_{j\in \mathbb{N}}$ of positive real numbers such that $\sigma_j\to \sigma^-$ and such that
    for any $j\in \mathbb{N}$ there exist $\gamma_j\in \mathcal{P}$ with
    \begin{enumerate}
        \item $\max_{t\in I_{\gamma_j}^{H}}E_{H,\sigma_j}(\gamma_j(t),f_{\gamma_j,t})\leq\omega_{H,\sigma_j}+\frac{1}{j}$;
        \item For all $t\in I_{\gamma_j}^H$ with $E_{H,\sigma_j}(\gamma_j(t),f_{\gamma_j,t})\geq\omega_{H,\sigma_j}-\frac{1}{j}$, there holds
        \begin{align*}
            \text{Area}(\gamma_j(t))\leq7H^2c\text{ and }\sigma^4\int_{\Sigma}\vert\sff^{\gamma_j(t)}\vert^4\,d\vol_{g_{\gamma_j(t)}}\leq\frac{3\varepsilon}{\log\sigma^{-1}},
        \end{align*}
    \end{enumerate}
    whenever $j$ is sufficiently large.
\end{lemma}
\begin{proof}
    For any $j\in \mathbb{N}$, set
    \begin{align}\label{eq: def-sigman-Hn}
        \sigma_j=\sigma-\frac{\sigma\log\sigma^{-1}}{8\varepsilon j},H_j=H-\frac{1}{8cHj}.
    \end{align}
    By definition of derivative, for $j$ sufficiently large we have
    \begin{align}\label{eq: derivative-estimate-omega-sigma}
        \omega_{H,\sigma}\geq\omega_{H,\sigma_j}\geq\omega_{H,\sigma}-\frac{2\varepsilon(\sigma-\sigma_j)}{\sigma\log\sigma^{-1}}=\omega_{H,\sigma}-\frac{1}{4j},
    \end{align}
    \begin{align}\label{eq: derivative-estimate-omega-H}
        \frac{\omega_{H_j,\sigma}}{H_j}\geq\frac{\omega_{H,\sigma}}{H}\geq\frac{\omega_{H_j,\sigma}}{H_j}-2c(H-H_j)=\frac{\omega_{H_j,\sigma}}{H_j}-\frac{1}{4Hj}.
    \end{align}
    Let $\gamma_j$ be a sweep-out in $\mathcal{P}$ such that
    \begin{align}\label{eq: choice-gamma-n}
        \max_{t\in I_{\gamma_j}^{H_j}}E_{H_j,\sigma}(\gamma_j(t),f_{\gamma_j,t})\leq\omega_{H_j,\sigma}+\frac{1}{4j}.
    \end{align}
    Then using the fact that $I_{\gamma_j}^H\subset I_{\gamma_j}^{H_j}$ (see the proof of Lemma \ref{lem: properties-omega}, $(1)$) and \eqref{eq: derivative-estimate-omega-H}, we obtain, for $j$ sufficiently large,
    \begin{align}\label{eq: estimate-H-sigma}
    \nonumber
        \max_{t\in I_{\gamma_j}^{H}}E_{H,\sigma}(\gamma_j(t),f_{\gamma_j,t})\leq&\frac{H}{H_j}\max_{t\in I^{H_j}_{\gamma_j}}E_{H_j,\sigma}(\gamma_j(t),f_{\gamma_j,t})\leq\frac{H}{H_j}\left(\omega_{H_j,\sigma}+\frac{1}{4j}\right)\\
        \leq&\omega_{H,\sigma}+\frac{1}{4j}+\frac{H}{4H_j j}\leq\omega_{H,\sigma}+\frac{1}{2j}.
    \end{align}
    Using \eqref{eq: derivative-estimate-omega-sigma}
    \begin{align*}
        \max_{t\in I_H^{\gamma_j}}E_{H,\sigma_j}(\gamma_j(t),f_{\gamma_j,t})\leq \omega_{H,\sigma_j}+\frac{1}{4j}+\frac{1}{2j}<\omega_{H,\sigma_j}+\frac{1}{j}.
    \end{align*}
    This yields the desired upper bound for $\max_{t\in I_{\gamma_j}^H}E_{H,\sigma_j}(\gamma_j(t),f_{\gamma_j,t})$.\\
    Next, let $t\in I_{\gamma_j}^{H}\subset I_{\gamma_j}^{H_j}$ (depending on $j$) be such that
    \begin{align}\label{eq: close-to-Hn-sigma}
        E_{H,\sigma_j}(\gamma_j(t),f_{\gamma_j,t})\geq\omega_{H,\sigma_j}-\frac{1}{j}\geq\omega_{H,\sigma}-\frac{2}{j}
    \end{align}
    Note that
    \begin{align*}
        E_{H_j,\sigma_j}(\gamma_j(t), f_{\gamma_j,t})\leq E_{H_j,\sigma}(\gamma_j(t), f_{\gamma_j, t})\leq\omega_{H_j,\sigma}+\frac{1}{4j},
    \end{align*}
    where the last inequality follows from our choice \eqref{eq: choice-gamma-n}.\\    
    Then we have
    \begin{align*}
        \frac{A_{\sigma_j}(\gamma_j(t))}{HH_j}&=\frac{1}{H-H_j}\left(\frac{E_{H_j,\sigma_j}(\gamma_j(t),f_{\gamma_j,t})}{H_j}-\frac{E_{H,\sigma_j}(\gamma_j(t),f_{\gamma_j,t})}{H}\right)\\&\leq\frac{1}{H-H_j}\left(\frac{\omega_{H_j,\sigma}}{H_j}+\frac{1}{H_j4j}-\frac{\omega_{H,\sigma}}{H}-\frac{2}{Hj}\right)\\
        &\leq 8cHj\left(\frac{\omega_{H_j,\sigma}}{H_j}-\frac{\omega_{H,\sigma}}{H}\right)+2c\frac{H}{H_j}\leq2c\left(1+\frac{H}{H_j}\right),
    \end{align*}
    where the last inequality follows from \eqref{eq: derivative-estimate-omega-H}.\\
    In particular, for $j$ large enough,
    \begin{align*}
        \text{Area}(\gamma_{j}(t))\leq 7H^2c.
    \end{align*}
    On the other hand, by \eqref{eq: estimate-H-sigma} and \eqref{eq: close-to-Hn-sigma},
    \begin{align}\label{eq: computation-estimate-entropy}
        (\sigma^4-\sigma_j^4)\int_{\Sigma}\vert\sff^{\gamma_j(t)}\vert^4\,d\vol_{g_{\gamma(t)}}&=E_{H,\sigma}(\gamma_j(t),f_{\gamma_j,t})-E_{H,\sigma_j}(\gamma_j(t),f_{\gamma_j,t})\\
        \nonumber
        &\leq\omega_{H,\sigma}-\omega_{H,\sigma_j}+\frac{1}{j}.
    \end{align}
    Dividing by $(\sigma-\sigma_j)$, we then have
    \begin{align*}
        4\sigma_j^3\int_{\Sigma}\vert\sff^{\gamma_j(t)}\vert^4\leq\frac{\omega_{H,\sigma}-\omega_{H,\sigma_j}}{\sigma-\sigma_j}+\frac{8\varepsilon}{\sigma \log \sigma^{-1}}\leq\frac{10\varepsilon}{\sigma\log\sigma^{-1}}
    \end{align*}
    for $j$ sufficiently large, where the last inequality follows from the second assumption in \eqref{eq: assumptions-good-sweep-outs}. Thus for $j$ sufficiently large we obtain the desired properties.
\end{proof}

\begin{rmk}\label{rem: Struwe-2-H-0}
     If $H=0$, all terms in the energy $E_{0,\sigma}$ are non-negative. Therefore, for a sweep-out $\gamma_n$ as in \eqref{eq: choice-gamma-n} (with $H_n=0$), we have
    \begin{align*}
        \operatorname{Area}(\gamma_n(t))\leq \omega_{0,\sigma}+1\quad \forall t\in I_{\gamma_n}^0.
    \end{align*}
    Moreover, choosing $\sigma_n$ as in \eqref{eq: def-sigman-Hn}, computation \eqref{eq: computation-estimate-entropy} shows that if $\sigma$ satisfies the second assumption in \eqref{eq: assumptions-good-sweep-outs} (for $H=0$), then for any $t\in I_{\gamma_n}^0$ with
    \begin{align*}
        A_{\sigma_n}(\gamma_n(t))\geq\omega_{0,\sigma_n}-\frac{1}{n},
    \end{align*}
    there holds
    \begin{align*}
        \sigma^4 \int_\Sigma\lvert \sff^{\gamma_n(t)}\rvert^4\leq \frac{3\varepsilon}{\log\sigma^{-1}}
    \end{align*}
    if $n$ is large enough.
\end{rmk}

\subsection{The deformation Lemma}
In this Subsection we show that for fixed $H$ and $\sigma$, for $n$ sufficiently large, for the "good" sweep-outs $\gamma_n$ constructed in Lemma \ref{lem: monotonicity-trick}, there exist a time $t\in I_{\gamma_n}^H$ such that $\gamma_n(t)$ is an almost-critical point of $E_{H, \sigma_n}$. These kinds of properties are typically proved by means of a pseudo-gradient flow, see \cite{Struwe-VM} for related discussions.\par
In our specific case, $E_{H,\sigma}$ is not directly defined on $\mathfrak{M}$. Nevertheless, we can still get a good notion of pseudo gradient flow as follows. In Definition \ref{def: reduction of functional}, we introduced the local reduction of the functional $E_{H,\sigma}$ to simply connected open sets $\mathfrak{A}$ in $\mathfrak{M}$ and showed that $\delta E_{H,\sigma}$ is well defined. We recall that the construction of the pseudo-gradient vector field (see for example \cite[Chapter II.3]{Struwe-VM}) makes only use of the continuity of $\delta E_{H,\sigma}$ (Proposition \ref{prop: continuity of first variation}), and works even if the original functional $E_{H,\sigma}$ is not globally well defined. Therefore, we have the following result.

\begin{lemma}\label{lem: existence-pseudo-gradient-flow}
    Set
    \begin{align*}
        \tilde{V}=\{u\in \mathfrak{M}|\delta E_{H,\sigma}(u)\ne0\}.
    \end{align*}
    There exists a locally Lipschitz map $X:\tilde{V}\rightarrow T\mathfrak{M}$ such that
    \begin{enumerate}
        \item $X(u)\in T_u\mathfrak{M}$ for all $u\in \tilde{V}$;
        \item $\vert X(u)\vert_u<2\min\{\Vert\delta E_{H,\sigma}(u)\Vert_{(T_u\mathfrak{M})^\ast},1\}$;
        \item $\delta E_{H,\sigma}(u)(X(u))<-\min\{\Vert\delta E_{H,\sigma}(u)\Vert_{(T_u\mathfrak{M})^\ast},1\}\Vert\delta E_{H,\sigma}(u)\Vert_{(T_u\mathfrak{M})^\ast}.$
    \end{enumerate}
\end{lemma}
We refer readers to \cite[Chapter II.3]{Struwe-VM} for the construction of the pseudo-gradient vector field.

Now we are ready to derive the existence of almost critical points. In the proof of the following result we follow some ideas from \cite{Michelat-Riviere}.

\begin{lemma}\label{lem: new-existence-almost-critical-points}

Let $\sigma_0,A>0$, $\varepsilon\in (0,1)$, $H\geq 0$. 
Let $\gamma\in\mathcal P$ be a sweep-out and suppose that for some $\alpha>0$, for $\sigma\leq \sigma_0$:
    \begin{enumerate}
        \item $\max_{t\in I_{\gamma}^{H}}E_{H,\sigma}(\gamma(t),f_{\gamma,t})\leq\omega_{H,\sigma}+\alpha$;
        \item For all $t\in I_\gamma^H$ with $E_{H,\sigma}(\gamma(t),f_{\gamma,t})\geq\omega_{H,\sigma}-\alpha$, we have
        \begin{align}\label{eq: assumtions-good-sweep-outs}
        \operatorname{Area}(\gamma(t))\leq A\text{ and }\sigma^4\int_{\Sigma}\vert\sff^{\gamma(t)}\vert^4d\vol_{\gamma(t)}\leq\frac{\varepsilon}{\log\sigma^{-1}}.
        \end{align}
    \end{enumerate}
    Then for any $\theta>0$, if $\alpha$ and $\sigma_0$ are sufficiently small, there exist $t_0\in I_\gamma^H$ with
    \begin{enumerate}
        \item $\vert E_{H,\sigma}(\gamma(t_0),f_{\gamma,t_0})-\omega_{H,\sigma}\vert\leq\alpha;$
        \item $\gamma(t_0)$ is $\theta$-almost critical for $E_{H,\sigma}$.
    \end{enumerate}
\end{lemma}

\begin{proof}
    Let us focus on the sub-intervals
    \begin{align*}
        I=\left\{t\in {I}^H_\gamma|E_{{H,\sigma}}(\gamma(t))\geq\omega_{H,\sigma}-\alpha/2\right\},
    \end{align*}
    \begin{align*}
        J=\left\{t\in {I}^H_\gamma|E_{H,\sigma}(\gamma(t))>\omega_{H,\sigma}-\alpha\right\}.
    \end{align*}
    If $H>0$, assume that $\alpha<H\zeta$; if $H=0$ assume that $\alpha<\zeta_0$.
    We wish to prove $\inf_J\Vert\delta E_{H,\sigma}(\gamma(t))\Vert_{(T_{\gamma(t))}\mathfrak{M})^\ast}<\theta$ provided $\alpha$ is small. We prove this by contradiction. Here we consider the case $H>0$. The proof for $H=0$ is similar, see Remark \ref{rmk: deformation-H-0}.\par
    Assume that $\lVert\delta E_{H,\sigma}(\gamma(t))\rVert_{(T_{\gamma(t))}\mathfrak{M})^\ast}\geq\theta$ for all $t\in J$.
    Let $\phi\in C_c^{\infty}((0,1),[0,1])$ be a cut-off function with $\phi=1$ on $I$ and $\phi=0$ on $[0,1]\smallsetminus J$.\\
    Let $\eta$ be a smooth, non-decreasing function supported in $\mathbb{R}_+$ such that $\eta\equiv 1$ on $[1,\infty)$. For each sweep-out $\gamma\in \mathcal{P}$, for $t\in[0,1]$ set
    \begin{align*}
        \psi^t(\gamma(t)):=\eta\left(\frac{4}{\zeta}\left(H^{-1}\operatorname{Area}(\gamma(t))+\vol(f_{\gamma,t})-\left(\vol_g(\mathcal{M})+\zeta\right)\right)\right).
    \end{align*}
Let $\delta_\mathcal{M}$ be the injectivity radius of $\mathcal{M}$, and note that $\psi^t$ can be continuously extended in a natural way to a $\delta_\mathcal{M}$-neighborhood of $\gamma(t)$ in $\mathfrak{M}$. We continue to denote this extension by $\psi^{t}$; it is still given by the same formula as above. Then, if $\gamma'\in \mathcal{P}$ satisfies $\gamma'(0)=\gamma(0)$ and $\operatorname{dist}_\mathfrak{M}(\gamma(t),\gamma'(t))<\delta_\mathcal{M}$ for any $t\in (0,1)$, its corresponding function $\psi^t_{\gamma'}$ satisfies $\psi_{\gamma'}^t(\gamma'(t))=\psi^t(\gamma'(t))$.

    Consider the flow $\Phi^s(\gamma(t))$ associated to $\psi X$, where $X$ is as in Lemma \ref{lem: existence-pseudo-gradient-flow}, starting from $\gamma(t)$:
    \begin{align*}
        \begin{cases}
            \frac{d}{ds}\Phi^s(\gamma(t))=(\psi^t X)(\Phi^s(\gamma(t)))\\
            \Phi^0(\gamma(t))=\gamma(t).
        \end{cases}
    \end{align*}
   
   \textbf{Claim 1}: The flow exists for a time $T_0\leq\frac{\delta_\mathcal{M}}{2}$ which depends only on $\mathcal{M},\theta,H,\sigma_0$ and $A$.
   \begin{proof}
    To verify this, it is enough to show that $\lVert \delta E_{H,\sigma}(\Phi^s(\gamma(t)))\rVert_{(T_{\Phi^s(\gamma(t))}\mathfrak{M})^\ast}$ stays bounded from below away from $0$ for $s\in [0,T_0]$ and $T_0$ as above. In fact, note that, as long as the flow exists,
   \begin{align}\label{eq: control-distance-on-M}
       \operatorname{dist}_\mathfrak{M}( \Phi^s(\gamma(t)),\gamma(t))\leq \int_0^s\lVert (\psi^t X)(\Phi^\tau(\gamma(t)))\rVert_{T_{\Phi^\tau(\gamma(t))}\mathfrak{M}}\, d\tau\leq 2s.
   \end{align}
   Thus by Proposition \ref{prop: continuity of first variation} there exist $s_0\in (0,\frac{\delta_\mathcal{M}}{2})$ such that
   \begin{align*}
       \lVert \delta E_{H,\sigma}(\Phi^s(\gamma(t)))\rVert_{(T_{\Phi^s(\gamma(t))}\mathfrak{M})^\ast}\geq  \lVert \delta E_{H,\sigma}(\gamma(t))\rVert_{(T_{\gamma(t)}\mathfrak{M})^\ast}-Cs\geq \theta-Cs
   \end{align*}
   for any $s\in (0,s_0)$, where $C$ depends on $A,\sigma_0, H$ and $\mathcal{M}$. Therefore, if we choose $T_0< \min\{\frac{s_0}{2},\frac{1}{2C}\theta\}$, we obtain the desired statement.
   \end{proof}
   Let $s\in (0,T_0)$, and assume that $\theta\leq 1$. By Lemma \ref{lem: existence-pseudo-gradient-flow},
   \begin{align}\label{eq: negative-derivative}
       \frac{d}{ds}E_{H,\sigma}(\Phi^{s\phi(t)}(\gamma(t)), f_{\Phi^{s\phi}\circ\gamma,t})\leq -\psi^t(\Phi^{s\phi(t)}(\gamma(t)))\phi(t)\theta^2.
   \end{align}
   \textbf{Claim 2}: for $s\in [0,T_0]$, $ I_\gamma^{H}= I^{H}_{\Phi^{s\phi}\circ\gamma}$. 
   \begin{proof}
    If $t\notin I_\gamma^{H}$, then $\psi^t((\gamma(t))=0$, thus $\Phi^s(\gamma(t))$ is constant in $s$ for any such $t$ and $t\notin I^H_{\Phi^{s\phi}\circ\gamma}$. On the other hand, if $t\in I_\gamma^{H}$, either $\psi^t(\gamma(t))=0$  and the flow is constant in $s$ (so that $t\in I^H_{\Phi^{s\phi}\circ \gamma}$, or $\psi^t(\gamma(t))>0$ and the flow is non-constant. In this case
\begin{align}\label{eq: staying-above-threshold}
    H^{-1}\operatorname{Area}((\gamma(t)))+\vol(f_{\gamma, t})-\left(\vol_g(\mathcal{M})+\zeta\right)>\xi\frac{\zeta}{4}>0,
\end{align}
where $\xi=\sup\{t\in \mathbb{R}_+\vert \eta(t)=0\}$. Then either \eqref{eq: staying-above-threshold} continues to hold for all $s\in [0,T_0]$,
or there exist $s_0\in (0,T_0]$ with
\begin{align*}
    H^{-1}\operatorname{Area}(\Phi^{s_0\phi(t)}(\gamma(t)))+\vol(f_{\Phi^{s_0\phi}\circ\gamma, t})-\left(\vol_g(\mathcal{M})+\zeta\right)=\xi\frac{\zeta}{4}.
\end{align*}
In this case $\psi^t(\Phi^{s_0\phi(t)}(\gamma(t)))=0$ and the flow is constant in $s$ for $s\geq s_0$. Hence in both cases $t\in I_{\Phi^s\circ\gamma}^H$ for all $0\leq s\leq T_0$.
\end{proof}
    Thus for $s\in [0, T_0]$ there holds
    \begin{align*}
       {\omega}_{H,\sigma}\leq\max_{t\in I_{\Phi^{s\phi}\circ\gamma}^{H}}E_{H, \sigma}(\Phi^{s\phi(t)}(\gamma(t)), f_{\Phi^{s\phi}\circ\gamma,t})
       \leq\max_{t\in I_{\gamma}^{H}}E_{H, \sigma}(\gamma(t), f_{\gamma,t}) \leq \omega_{H,\sigma}+\alpha.
    \end{align*}
    Here, the second inequality follows from \eqref{eq: negative-derivative} and the third one from assumption $(1)$.\par
    Recall that for all $t$ with $E_{H,\sigma}(\gamma(t),f_{\gamma,t})\geq\omega_{H,\sigma}-\alpha$, by assumption (2) we have
    \begin{align*}
        \sigma^4\int_{\Sigma}\vert\sff^{\gamma(t)}\vert^4dvol_{\gamma(t)}\leq\frac{\varepsilon}{\log\sigma^{-1}}.
    \end{align*}

    Since $\operatorname{dist}_\mathfrak{M}(\Phi^{s\phi(t)}(\gamma(t)), \gamma(t))\leq 2s$ (by \eqref{eq: control-distance-on-M}), by Lemma \ref{lem: closeness-A-E} (with $E=\frac{\varepsilon}{\sigma^4\log\sigma^{-1}}$) there holds
    \begin{align*}
        \left\lvert\int_{\Sigma}\lvert \sff^{\gamma(t)}\rvert^4d\vol_{\gamma(t)}-\int_{\Sigma}\lvert \sff^{\Phi^{s\phi(t)}(\gamma(t))}\rvert^4d\vol_{\Phi^{s\phi(t)}(\gamma(t))}\right\rvert\leq C_A\left(\frac{\varepsilon}{\sigma^4\log\sigma^{-1}}+1\right)2s,
    \end{align*}
    where $C_A$ depends only on $A$.
    In particular,
    \begin{align*}
        \sigma^4\int_{\Sigma}\lvert\sff^{\Phi^{s\phi(t)}(\gamma(t))}\rvert^4d\vol_{\Phi^{s\phi(t)}(\gamma(t))}\leq& \sigma^4\int_{\Sigma}\lvert \sff^{\gamma(t)}\rvert^4d\vol_{\gamma(t)}+C_A\left(\frac{\varepsilon}{\log\sigma^{-1}}+\sigma^4\right)2s\\
        \leq &\frac{\varepsilon}{\log\sigma^{-1}}+2C_A\left(\frac{\varepsilon}{\log{\sigma^{-1}}}+\sigma^4\right)s
    \end{align*}
    Thus, for any $s\in [0,T_0]$ and $t\in I_{\gamma}^H=I_{\Phi^{s\phi(t)}(\gamma(t))}^H$, either
    \begin{equation}\label{eq: flow-below-width}
        E_{H,\sigma}(\Phi^{s\phi(t)}(\gamma(t)), f_{\Phi^{s\phi}\circ\gamma, t})<\omega_{H,\sigma}-\frac{H\zeta}{10},
    \end{equation}
    or
    \begin{align}\label{eq: flow-below-width-second-case}
        \operatorname{Area}(\Phi^{s\phi(t)}(\gamma(t)))+H\vol(f_{\Phi^{s\phi}\circ\gamma,t})
        \geq&\omega_{H,\sigma}-\frac{H\zeta}{10}-\sigma^4\int_{\Sigma}\lvert\sff^{\Phi^{s\phi}(\gamma(t))}\rvert^4d\vol_{\Phi^{s\phi}(\gamma(t))}
        \\
        \nonumber
        \geq&
        \omega_{H,\sigma}-\frac{H\zeta}{10}-2C_A\left(\frac{\varepsilon}{\log{\sigma^{-1}}}+\sigma^4\right)s.
    \end{align}
In the second case, if $\sigma_0$ is small enough (depending on $A$, $\mathcal{M}$ and $H$), we see that $\psi^t(\Phi^{s\phi(t)}((\gamma(t)))\geq 1$.
Note that if \eqref{eq: flow-below-width} holds for some $t\in I_\gamma^H$, $s\in [0,T_0]$, then \eqref{eq: flow-below-width} also holds for $t, s'$, where $s'\in [s,T_0]$, by \eqref{eq: negative-derivative}. Thus for any $t\in I_\gamma^H$, either \eqref{eq: flow-below-width} holds for $s=T_0$, or \eqref{eq: flow-below-width-second-case} holds for all $s\in [0,T_0]$. In this case, \eqref{eq: negative-derivative} implies that
    \begin{align*}
        E_{H,\sigma}(\Phi^{T_0\phi(t)}(\gamma(t)))\leq \max\{\omega_{H,\sigma}+\alpha-T_0\theta^2, \omega_{H,\sigma}-\alpha\}.
    \end{align*}
    If $\alpha$ is sufficiently small, in both cases we obtain
    \begin{align*}
        \max_{t\in\Phi^{T_0\phi(t)}(\gamma(t))}E_{H,\sigma}(\Phi^{T_0\phi(t)}(\gamma(t)))<\omega_{H,\sigma},
    \end{align*}
    which contradicts the definition of $\omega_{H,\sigma}$. This proves the Lemma.

\end{proof}
\begin{rmk}\label{rmk: deformation-H-0}
   If $H=0$, the proof is similar. Assume again that $\lVert\delta E_{0,\sigma}(\gamma(t))\rVert_{(T_{\gamma(t)}\mathfrak{M})^\ast}\geq\theta$ for all $t\in J$. Let $\psi:\mathfrak{M}\to\mathbb{R}$ be defined as
   \begin{align*}
       \psi(u):=\eta\left(\frac{4}{\zeta_0}(\operatorname{Area}(u)-\zeta_0)\right),
   \end{align*}
   and define $\Phi^s$ as above. Claims 1 and 2 can be proved as before; instead of \eqref{eq: staying-above-threshold}, one uses the fact that if $\psi(\gamma(t))>0$, then
   \begin{align*}
       \operatorname{Area}(\gamma(t))-\zeta_0> \xi\frac{\zeta_0}{4}>0.
   \end{align*}
   Therefore we find again that for any $s\in [0,T_0]$ and $t\in I^0_\gamma$, either
   \begin{equation}
        E_{0,\sigma}(\Phi^{s\phi(t)}(\gamma(t)))<\omega_{0,\sigma}-\frac{\zeta_0}{10},
    \end{equation}
    or
    \begin{align*}
        \operatorname {Area}(\Phi^{s\phi(t)}(\gamma(t)))
        \geq
        \omega_{0,\sigma}-\frac{\zeta_0}{10}-2C_A\left(\frac{\varepsilon}{\log{\sigma^{-1}}}+\sigma^4\right)s.
    \end{align*}
   Thus, arguing as above, we see that both cases lead to a contradiction for $\sigma_0$ and $\alpha$ sufficiently small.
\end{rmk}
Combining the results of this Section we obtain the following theorem, which guarantees the existence of almost critical points for $E_{H,\sigma}$ with control of the area and the second fundamental form.
\begin{thm}\label{thm: existence-critical-points}
    For $H=0$ and for a. e. $H>0$, and for any sequence $\theta_k\rightarrow0$, there exist  constants $C,a_0,A_0,\varepsilon_0>0$ and a sequence of real positive numbers $\{\sigma_k\}_{k\in \mathbb{N}}$ such that $\sigma_k\to 0$ and for any $k\in \mathbb{N}$ sufficiently large there exist $\gamma_k\in \mathcal{P}$, $t_k\in [0,1]$, with
    \begin{enumerate}
        \item $\lvert E_{H, \sigma_k}(\gamma_{k}(t_k), f_{\gamma_k, t_k})-\omega_{H,\sigma_k}\rvert\leq \frac{1}{k}$;
        \item $\lVert \delta E_{H, \sigma_k}(\gamma_{k}(t_k))\rVert\leq \theta_k$;
        \item $\varepsilon_0\leq\operatorname{Area}(\gamma_k(t_k))\leq A_0$;
        \item $\sigma_k^4{\log\sigma_k^{-1}}\int_{\Sigma}\lvert \sff^{\gamma_k(t_k)}\rvert^4 d\operatorname{vol}_{g_{\gamma_k(t_k)}}\rightarrow0$.
    \end{enumerate}
\end{thm}
\begin{proof}
    Let $H>0$ be as in Lemma \ref{lem: properties-omega}. Then there exists a sequence $\sigma_n\to 0$ and $c>0$ such that \eqref{eq: conclusion-Lem-Struwe-1} holds. For any $n\in \mathbb{N}$, set $\varepsilon_n:=\sigma_n\log\sigma_n^{-1}\frac{\de\omega_{H,\sigma}}{\de \sigma}\big|_{\sigma=\sigma_n}$, so that $\lim_{n\to\infty}\varepsilon_n=0$.
    Note that $H, \sigma_n$ satisfy the assumptions of Lemma \ref{lem: monotonicity-trick} with $\varepsilon=\varepsilon_n$. Therefore for any $j$ there exist $\sigma_{n,j}$ and $\gamma_j\in \mathcal{P}$
    satisfying the conclusions of the Lemma. For any $n$, we choose $j_n= n$. We can then apply Lemma \ref{lem: new-existence-almost-critical-points} as follows: for any $\theta_k$, let $n_k\geq k$ such that $\sigma_k:=\sigma_{n_k, n_k}$, $\alpha_{k}:=\frac{1}{n_k}$, $\varepsilon_{k}:=\varepsilon_{n_k}$ are sufficiently small so that the Lemma applies for $\theta_k$, $\sigma_{k}, \alpha_k,\varepsilon_{k}$ and $A=7H^2c$ (where $c$ is the constant from \eqref{eq: conclusion-Lem-Struwe-1}). Then by Lemma \ref{lem: new-existence-almost-critical-points} there exists $t_k\in (0,1)$ such that $\gamma_{n_k}$, $t_k$ satisfy all the conclusions of the Theorem, except for the lower area bound. To prove the lower bound we argue as follows. Assume by contradiction that the bound doesn't hold. Then there exists a subsequence (not relabeled) such that
    $\operatorname{Area}(\gamma_k(t_k))\to0$ as $k\to\infty$. Note that by $(1)$, $(4)$ and Proposition \ref{prop: non-trivial of min-max}, we must have $\vol(f_{\gamma_k,t_k})>\frac{1}{2}\vol(\mathcal{M})$ for $k$ large enough. Let $\varepsilon\in (0,\frac{1}{2}\vol(\mathcal{M}))$ be a constant to be determined later, depending only on $\mathcal{M}$ and $H$, and for any $k\in \mathbb{N}$ let $t_k^\varepsilon\in (0, t_k)$ such that
    \begin{align*}
        \vol(f_{\gamma_k, t_k^\varepsilon})=\vol(f_{\gamma_k, t_k})-\varepsilon.
    \end{align*}
    Arguing just as in the proof of Proposition \ref{prop: non-trivial of min-max} (now for the current ${\gamma_k}_\#\jump{[t_k^\varepsilon,t_k]\times \Sigma}$) we obtain
    \begin{align*}
        \operatorname{Area}(\gamma_k(t_k^\varepsilon))+\operatorname{Area}(\gamma_k(t_k))\geq C_{\mathcal{M}}\varepsilon^\frac{2}{3},
    \end{align*}
where $C_\mathcal{M}$ is a constant depending only on $\mathcal{M}$.
Then
\begin{align}
\label{eq: area-positive}
    E_{H,\sigma_k}(\gamma_k(t_k^\varepsilon),f_{\gamma_k,t_k^\varepsilon})\geq &C_\mathcal{M}\varepsilon^\frac{2}{3}-\operatorname{Area}(\gamma_k(t_k))+H(\vol(f_{\gamma_k, t_k})-\varepsilon)\\
    \nonumber
    \geq& C_\mathcal{M}\varepsilon^\frac{2}{3}-H\varepsilon+E_{H,\sigma_k}(\gamma_k(t_k),f_{\gamma_k,t_k})-2\operatorname{Area}(\gamma_k(t_k))\\
    \nonumber
    \geq& C_\mathcal{M}\varepsilon^\frac{2}{3}-H\varepsilon+\omega_{H,\sigma_k}-\frac{1}{k}-2\operatorname{Area}(\gamma_k(t_k)),
\end{align}
where in the last inequality we used $(1)$.
Choose $\varepsilon$ such that
\begin{align*}
    C_\mathcal{M}\varepsilon^\frac{2}{3}-H\varepsilon>0.
\end{align*}
    Then for $k$ sufficiently large, \eqref{eq: area-positive} contradicts the fact that $\gamma_k$ satisfies property $(1)$ in Lemma \ref{lem: monotonicity-trick}.
    
    If $H=0$, the argument is the same: Lemma \ref{lem: properties-omega} and Lemma \ref{lem: monotonicity-trick} are replaced by Remark \ref{rem: Struwe-1-H-0} and Remark \ref{rem: Struwe-2-H-0}; and Lemma \ref{lem: new-existence-almost-critical-points} provides again the existence of almost-critical points satisfying the desired properties.
    In this case, $\varepsilon_0, A_0$ can be chosen to be any constants satisfying $\varepsilon_0<\omega_{0}<A_0$.
\end{proof}
Finally, we combine the existence and regularity theory for parametrized CMC varifolds developed in the previous Sections with Theorem \ref{thm: existence-critical-points} to prove Theorem \ref{Prop: existence of qualifying sequence}.
\begin{proof}[Proof of Theorem \ref{Prop: existence of qualifying sequence}]
    For $H$ as in Theorem \ref{thm: existence-critical-points} and $k\in \mathbb{N}$, let $u_k:\Sigma\rightarrow\mathcal{M}$ be the immersion $\gamma_k(t_k)$ given by Theorem \ref{thm: existence-critical-points}, where the numbers $\theta_k$ are as in Step 2 of the proof of Theorem \ref{thm: multiplicity-one}. Then $u_k$ satisfies the assumption of Theorem \ref{thm: main-theorem}. In particular, we can pass to a subsequence, to get the convergence of the varifolds $\textbf{v}_{u_k}$ (i.e. the varifolds induced by the maps $u_k$) to $\sum_{i=1}^N\mathbf{v}_{u_{\infty}^i}$, where $u^i_{\infty}:\Sigma_\infty^i\rightarrow\mathcal{M}$ is a $H$-CMC branched immersion from a Riemann surface $\Sigma_\infty^i$,
    and $\sum_{i=1}^Ng(\Sigma_{\infty}^i)\leq h$. In particular, if $\Sigma_\infty$ denotes the union of the surfaces $\Sigma_\infty^i$, we have a branched $H$-CMC immersion $u:{\Sigma}_\infty\rightarrow\mathcal{M}$, and ${\Sigma}_\infty$ satisfies $g({\Sigma}_\infty)\leq h$.
\end{proof}

\end{definition}
\appendix
\section{Technical results}
\begin{lemma}[Cf. {\cite[Lemma A.4]{Pigati-FB}}]\label{lem: alnerative for density lower bound}
    There exists a constant $c_v(H,\mathcal{M})>0$, depending only on $\mathcal{M}$ and $H$ (and denoted $c_v$ in the following), with the following property. Given any $p\in\mathcal{M}$ and $0<s<c_v$, for any varifold $V$ in $\mathcal{M}$ which has first variation bounded by $H$ outside $B_s(p)$ and density $\Theta(V,\cdot)\geq\bar{\theta}$ on $\spt V\setminus B_s(p)$, we have either
    \begin{enumerate}
        \item $\spt V\subset B_{2s}(p)$; or
        \item $\vert V\vert(\mathcal{M}\setminus \bar{B}_s(p))\geq c_v\bar{\theta}$.
    \end{enumerate}
\end{lemma}
\begin{proof}
    Multiplying $V$ by $\bar{\theta}^{-1}$, we may assume $\bar{\theta}=1$. Take a small $\gamma\geq c_v$ to be determined later. There are two possibilities.\par
    The first possibility is $\spt V\subset B_{2\gamma}(p)$. In this case let us prove $\spt V\subset B_{2s}(p)$ provided $\gamma$ and $c_v$ are small. Take a coordinate system $\xi:B_{5\gamma}(p)\rightarrow\R^3$. For $\gamma$ small, we can assume $\Vert g_{ij}-\delta_{ij}\Vert_{C^1}\leq C\gamma$ under this coordinate, where the upper bound for $\gamma$ here is independent of $p\in\mathcal{M}$. Let us consider the test vector field $X=\chi(\vert\xi\vert)\sum_i\xi^i\frac{\de}{\de\xi^i}$, where $\chi$ is a smooth function satisfying
    \begin{align*}
        \chi(t)=\begin{cases}
            0, &0\leq t\leq \frac{4}{3}s\text{ or }t\geq 4\gamma;\\
            1, &\frac{5}{3}s\leq t\leq3\gamma;\\
            \text{monotone}, & \text{ for other }t.
        \end{cases}
    \end{align*}
    Since $g_{ij}$ is $C^1$ close to $\delta_{ij}$, we deduce that $\div (\sum_i\xi^i\frac{\de}{\de\xi^i})\geq 2-C\gamma$ on $B_{3\gamma}(p)$. Since $\spt V\subset B_{2\gamma}(p)$,
    \begin{align*}
        \begin{cases}
            \operatorname{div}(X)=0&\text{in }\overline{B_{\frac{4s}{3}}(p)}\\
            \operatorname{div}(X)>0&\text{in }B_{2\gamma}(p)\smallsetminus\overline{B_{\frac{4s}{3}}(p)}.
        \end{cases}
    \end{align*}
Take $X$ in the first variation formula, we deduce that
    \begin{align*}
        \int_{Gr_2(\mathcal{M)}}\div X\,dV(p,L)&\geq(2-C\gamma)\int_{Gr_2(B_{2\gamma}(p))}\chi(\vert\xi\vert)\,d V\\&=(2-C\gamma)\int_{B_{2\gamma}(p)}\chi(\vert\xi\vert)\,d\vert V\vert.
    \end{align*}

    On the other hand 
    \begin{align*}
        H\int_{\mathcal{M}}\vert X\vert d\vert V\vert\leq2\gamma H\int_{B_{2\gamma}(p)}\chi(\vert\xi\vert)\,d\vert V\vert.
    \end{align*}
    If $(2-C\gamma)\geq 2\gamma H$, we have a contradiction here, unless $\spt V\subset B_{4s/3}(p)$. Hence $\spt V\subset B_{4s/3}(p)\subset B_{2s}(p)$ provided $\gamma$ is small enough.\par
    If the first possibility is not true, then $\spt V\setminus B_{2\gamma}(p)$ is not empty. For each $q\in \spt V\setminus B_{2\gamma}(p)$, note that by assumption $V$ has bounded first variation inside $B_{\gamma}(q)$. Hence, the monotononicity formula gives
    \begin{align*}
        \vert V\vert (B_\gamma(q))\geq c(H,\mathcal{M})\gamma^2\Theta(V,q)\geq c(H,\mathcal{M})\gamma^2.
    \end{align*}
    Since $\gamma\geq c_v$, this already provide the desired lower bound if we set a new constant $\bar{c}_v=\min\{c_v,c(H,\mathcal{M})c_v^2\}$.
\end{proof}
\begin{lemma}[{Cf. \cite[Lemma A.6]{Pigati-FB}}]\label{lem: degeneration-conf}
    There exist $c_v>0$ and a function $\delta:(0,\infty)^2\to(0,\infty)$ with $\lim_{s\rightarrow0}\delta(s,t)=0$ for each $t>0$, depending only on $H$ and $\mathcal{M}$, with the following property:
    given $p_1,p_2\in\mathcal{M},s>0$, if a 2-varifold $V$ has generalized mean curvature bounded by $K$ outside $B=\bar{B}_s(p_1)\cup\bar{B}_s(p_2)$, has density $\Theta(V,\cdot)\geq\bar{\theta}$ on $\spt V\setminus B$ and $\vert V\vert(B_r(q))\leq \alpha r^2$ for all $q\in\mathcal{M},r>0$; then either
    \begin{enumerate}
        \item $\vert V\vert(\mathcal{M})\leq\bar{\theta}\delta(s,\alpha/\bar{\theta})$, or 
        \item $\vert V\vert(\mathcal{M})\geq c_v\bar{\theta}$.
    \end{enumerate}
\end{lemma}
\begin{proof}
    By considering $\bar{\theta}^{-1}V$, we may assume $\bar{\theta}=1$. Let
    \begin{align*}
        \delta(s,\alpha)=\sup\{\vert V\vert(\mathcal{M}):&V\text{ integral varifold with first variation}\\ &\text{ bounded by } K\text{ in }B \text{ and }\vert V\vert(\mathcal{M})<c_v\},
    \end{align*}
    where $c_v$ is a constant to be determined later.
    Take a sequence $s_n\rightarrow0$ of positive numbers and a sequence $\{V_n\}_{n\in \mathbb{N}}$ satisfying the assumptions of the Lemma with $s=s_n$ and
    \begin{align*}
        (1-2^{-n})\delta(s_n,\alpha)\leq\vert V_n\vert(\mathcal{M})<c_v.
    \end{align*}
    After passing to a subsequence, we may assume $V_n$ converges to $V$ in the sense of varifolds. Then $\vert V\vert(B_r(q))\leq\alpha r^2$ for all $q\in\mathcal{M},r>0$, and $V$ has generalized mean curvature bounded by $K$ away from two points $p_1,p_2$. Combining these two facts, one shows that $V$ has generalized mean curvature bounded by $K$ on the whole $\mathcal{M}$ (by writing any variation $X$ as $\eta X+(1-\eta)X$, with $\eta$ supported in a $\frac{1}{k}$-neighborhood of $p_1,p_2$ and equal to $1$ in a $\frac{1}{2k}$-neighborhood of the points).
    Moreover, as $V$ is integral, its density is at least one a.e. on its support, and thus the monotonicity formula implies that either $V=0$, or there holds $\lvert V\rvert(\mathcal{M})\geq c(\mathcal{M)}$ for some constant $c(\mathcal{M})$.
    Choose $c_v$ less than $c(\mathcal{M},H)$, then we must have $V=0$, and since $\lvert V_n\rvert(\mathcal{M})\to\lvert V\rvert(\mathcal{M})$, we must have
    $\lim_n\delta(s_n,\alpha)=0$. Since $s_n$ was arbitrary, the proof is complete.
\end{proof}

\begin{lemma}\label{lem: Levy-Gromov}[Levy-Gromov isoperimetric inequality]
Let $A\subset\mathcal{M}$ be a Caccioppoli set with $\lvert A\rvert\leq \frac{1}{2}\lvert \mathcal{M}\rvert$. Then
\begin{align*}
    \lvert A\rvert^\frac{2}{3}\leq C\lvert \partial A\rvert,
\end{align*}
where $C$ is a constant depending on the minimum of the Ricci curvature of $\mathcal{M}$ and on the diameter and volume of $\mathcal{M}$.
\end{lemma}
\begin{proof}
    First notice that it is enough to prove the statement for smooth subsets, the general statement follows by approximation (the argument in the proof of Theorem 3.42 in \cite{AFP} can be adapted to this setting).\\
    By the main Theorem in \cite{BBG-LG}, there holds
    \begin{align}\label{eq: estimate BBG}
        \vol(\partial A)\geq C_1\,\vol(\mathcal{M})\,\vol(\partial B),
    \end{align}
    where $C_1$ is a positive constant depending on the minimum of the Ricci curvature of $\mathcal{M}$ and on the diameter of $\mathcal{M}$, and $B$ is any geodesic ball in $\mathbb{S}^3$ satisfying
    \begin{align*}
        \frac{\vol(B)}{\vol(\mathbb{S}^3)}=\frac{\vol(A)}{\vol(\mathcal{M})}.
    \end{align*}
    Since $\lvert A\rvert\leq \frac{1}{2}\lvert \mathcal{M}\rvert$, we have $\lvert B\rvert\leq \frac{1}{2}\lvert {\mathbb{S}^3}\rvert$, therefore the isoperimetric inequality in the sphere (or an explicit computation) yields
    \begin{align}\label{eq: isoperimetric-S3}
        \frac{1}{C_2}\vol(\partial B)\geq\vol(B)^\frac{2}{3}=\left(\frac{\vol(\mathbb{S}^3)}{\vol(\mathcal{M})}\right)^\frac{2}{3}\vol(A)^\frac{2}{3}.
    \end{align}
    Combining \eqref{eq: estimate BBG} and \eqref{eq: isoperimetric-S3}, we obtain the desired estimate.
\end{proof}

\section{Results about quasiconformal homeomorphisms}
In this appendix we collect some of the results of the works \cite{PR-Regularity}, \cite{PR-Multiplicity-One} and \cite{Pigati-FB} by Pigati and Rivière.
\begin{lemma}\label{lem: appendixB-Lem4.3}[Lemma 4.3 in \cite{PR-Multiplicity-One}]
    Given $K\geq 1$, and $s>0$, there exists a constant $\delta_0>0$, depending only on $Q, K, s$, with the following properties: whenever
    \begin{enumerate}
        \item $u\in W^{1,2}\cap C^0(\overline{B_1(0)}, \mathbb{R}^Q)$ has $\lVert u\vert_{\partial B_1(0)}-f(s\cdot)\rVert_{C^0(\partial B_1(0))}\leq \delta_0$ for some $f\in \mathscr{D}_K^\Pi$,
        \item $u\circ\varphi^{-1}$ is harmonic and weakly conformal on $\varphi(B_1(0))$, where $\varphi:\mathbb{R}^2\to\mathbb{R}^2$ is a $K$-quasiconformal homeomorphism,
    \end{enumerate}
    then $\pi_\Pi\circ u\circ\varphi^{-1}$ is a diffeomorphism from $\varphi(\overline{B_\frac{1}{2}(0)})$ onto its image.
    In particular $\pi_\Pi\circ u$ is injective on $\overline{B_\frac{1}{2}(0)}$.
\end{lemma}
\begin{lemma}\label{lem: appendixB-LemmaA.1}[Lemma A.1 in \cite{PR-Multiplicity-One}]
Assume that $F\in C^0(\overline{B_1(0)}, \mathbb{R}^2)$ satisfies
\begin{align*}
    \lvert F(x)-\varphi(x)\rvert<\delta\text{ for all }x\in \partial B_1(0),
\end{align*}
for some $\delta\in (0,1)$ and some homeomorphism $\varphi: \mathbb{R}^2\to\mathbb{R}^2$, with $\varphi(0)=0$ and $\min_{\lvert x\rvert=1}\lvert\varphi(x)\rvert\geq 1$. Then $F(B_1(0))\supset B_{1-\delta}(0)$.
\end{lemma}
\begin{lemma}\label{lem: appendixB-LemmaA.1'}[Lemma 5.6 in \cite{PR-Multiplicity-One}]
    Assume that $v\in C^\infty(\overline{B_1(0)}, \mathcal{M}_{p,\ell})$
is a conformal immersion and that $\Pi$ is a 2-
plane such that there exists a $K$-quasiconformal homeomorphism $f:B_1(x)\rightarrow \Pi$
        in $\mathscr{D}_K^\Pi$ with
        \begin{align*}
            \vert v-f\vert<\delta_0\text{ on }\de B_1(0)\cup \de B_{s(K)}(0)\cup\de B_{s(K)^2}(0)
        \end{align*}
and $\frac{1}{2}\int_{B_1(0)}\lvert\nabla v\rvert^2\leq E$. Then there exist $\varepsilon_{E, K,\delta_0}$ such that if
\begin{align*}
    \int_{B_1(0)}\lvert \sff^v\rvert^4d\vol_{v}\leq \varepsilon_{E, K,\delta_0}\text{ and }\ell\leq \varepsilon_{E, K,\delta_0},
\end{align*}
then $\pi_\Pi\circ v$ is a diffeomorphism from $B_{2
s(K)^2}$
onto its image.
    
\end{lemma}

\printbibliography
\end{document}